\documentclass[reqno,11pt]{amsart}
\usepackage{amssymb}
\usepackage{amsmath}
\usepackage{amsthm}
\usepackage{pgf}
\usepackage{color}
\usepackage{mathrsfs}
\usepackage{graphicx}   
\usepackage{hyperref}
\usepackage{graphicx}
\usepackage{tikz}
\usepackage{pgf}
\usepackage{enumerate}
\usepackage{todonotes}
\usepackage{bbm}
\usepackage{mathtools}
\usepackage{ulem}
\usepackage{esint}
\usepackage{caption}

\usetikzlibrary{positioning}
\usetikzlibrary{decorations.pathreplacing,decorations.markings}
\tikzset{->-/.style={decoration={
			markings,
			mark=at position #1 with {\arrow{>}}},postaction={decorate}}}
\usepackage{enumitem} 
\usepackage{comment}

\usepackage{tikz}
\usetikzlibrary{decorations.pathmorphing,arrows.meta}

\usetikzlibrary{positioning}
\usetikzlibrary{decorations.pathreplacing,decorations.markings}
\usetikzlibrary{hobby}
\tikzset{->-/.style={decoration={
			markings,
			mark=at position #1 with {\arrow{>}}},postaction={decorate}}}

\usepackage{pgfplots}
\usetikzlibrary{intersections, pgfplots.fillbetween}
\pgfdeclarelayer{bg}
\pgfsetlayers{bg,main}
\pgfplotsset{compat=newest}

\usetikzlibrary{decorations.pathreplacing}
\tikzset{%
	show curve controls/.style={
		postaction={
			decoration={
				show path construction,
				curveto code={
					\draw [blue] 
					(\tikzinputsegmentfirst) -- (\tikzinputsegmentsupporta)
					(\tikzinputsegmentlast) -- (\tikzinputsegmentsupportb);
					\fill [red, opacity=0.5] 
					(\tikzinputsegmentsupporta) circle [radius=.5ex]
					(\tikzinputsegmentsupportb) circle [radius=.5ex];
				}
			},
			decorate
}}}

\tikzset{use path/.code=\tikz@addmode{\pgfsyssoftpath@setcurrentpath#1}}

\newtheorem{theorem}{Theorem}[section]
\newtheorem{lemma}[theorem]{Lemma}
\newtheorem{proposition}[theorem]{Proposition}
\newtheorem{corollary}[theorem]{Corollary}
\newtheorem{remark}[theorem]{Remark}

\newtheorem{definition}[theorem]{Definition}

\newcommand{\R}{\mathbb{R}}

\newcommand{\eps}{\varepsilon}

\def\eps{\varepsilon}

\renewcommand{\d}{{\mathrm d}}

\newcommand{\argmin}{\mathop{\mathrm{Arg\,min}}\limits}   
\DeclareMathOperator*{\esssup}{ess\,sup}

\def\to{\rightarrow}

\numberwithin{equation}{section}

\makeatletter
\newcommand*\bigcdot{\mathpalette\bigcdot@{.5}}
\newcommand*\bigcdot@[2]{\mathbin{\vcenter{\hbox{\scalebox{#2}{$\m@th#1\bullet$}}}}}
\makeatother

\begin{document}
	\title[Dynamic viscoelastic membranes]
	{Dynamic viscoelastic membranes}
	
\author{Martin Kru\v{z}\'{i}k}
\address[Martin Kru\v{z}\'{i}k]{Czech Academy of Sciences, Institute of Information Theory and Automation, Pod vod\'{a}renskou v\v{e}\v{z}\'i 4,
	CZ-182 00 Praha 8, Czech Republic,\\
   Faculty of Civil Engineering, Czech Technical University, Th\'{a}kurova 7, 166 29 Prague 6, Czech Republic }
\email{kruzik@utia.cas.cz}
\urladdr{http://staff.utia.cas.cz/kruzik/}

\author{Riccardo Voso}
\address[Riccardo Voso]{Czech Academy of Sciences, Institute of Information Theory and Automation, Pod vod\'{a}renskou v\v{e}\v{z}\'i 4,
	CZ-182 00 Praha 8, Czech Republic,\\
    Faculty of Electrical Engineering, Czech Technical University, Technick\'a 2, 160 00 Prague 6, Czech Republic}
\email{voso@utia.cas.cz}
	
	\subjclass[2020]{74K15, 74D10, 74A30, 74H20, 49J45, 47J30, 28A33, 35Q74}
	
	\keywords{Dimension reduction, membranes, nonlinear viscoelasticity, Kelvin-Voigt rheology, nonsimple materials, hyperbolic minimizing movements, $\Gamma$-convergence, Young measures, measure-valued solutions}

\begin{abstract}
	We rigorously derive a dynamic membrane theory for nonlinear viscoelastic solids with Kelvin-Voigt rheology from the corresponding three-dimensional bulk theory, in the limit of vanishing thickness. The bulk model accounts for inertia, a nonsimple elastic energy depending on the second deformation gradient, and a frame-indifferent dissipation. Orientation preservation is ensured by the energy, while self-interpenetration of the limiting membrane is prevented by a uniform averaged invertibility condition. Our analysis relies on a hyperbolic minimizing movements scheme with two time-scales, which recasts the evolution into a hierarchy of incremental, sequential, and time-delayed problems. For fixed thickness, we prove the existence of solutions to the scheme and their convergence to a weak solution of the continuous-time problem as the time-scales vanish. We then identify the membrane limit at each level of the hierarchy, proving the $\Gamma$-convergence of the incremental functionals and the convergence of solutions to the sequential and time-delayed problems. Since the constant in the Korn inequality is not uniform in the thickness, these limits rely on estimates for the time derivative of the Cauchy-Green tensor rather than for the velocity gradient. Finally, as the outer time-scale vanishes, the membrane time-delayed solutions converge to a measure-valued solution of the membrane continuous-time problem.
\end{abstract}

	\maketitle

	\section{Introduction}
	\label{introduction}	
	We are interested in the rigorous derivation, by dimension reduction, of a dynamic membrane theory for viscoelastic materials with Kelvin-Voigt rheology from the corresponding bulk theory. The bulk dynamics is given by
	\begin{align}
		&\rho \partial_{tt}u + D \mathcal{E}_{\eps} (u) + D_2 \mathcal{R}_{\eps} (u,\partial_t u) = f \quad \text{in}\; (0,T)\times\Omega_{\eps},\label{eq1}\\
		&u(0) = u_{0} \quad \text{and} \quad \partial_t u(0)=w_{0} \quad \text{in} \; \Omega_{\eps}, \label{eq11}
	\end{align}
	complemented by clamped boundary conditions on the lateral boundary of $\Omega_\eps$, which are included in the set of admissible deformations (see Section~\ref{Model}). Here, for every $\eps >0$, $\Omega_{\eps} = \omega \times (-\eps/2,\eps/2)$ is the reference configuration of the material, with $\omega \subset \mathbb{R}^2$ a bounded Lipschitz domain, $u: (0,T) \times \Omega_{\eps} \rightarrow \mathbb{R}^3$ is the deformation of the body, $\rho>0$ is the mass density of the material, $\mathcal{E}_{\eps}$ is the stored elastic energy, $\mathcal{R}_{\eps}$ is the pseudopotential of dissipation, and $f$ is the density of the external load. As customary, $\partial_t$ and $\partial_{tt}$ denote the first and the second time derivatives, $D$ denotes the G\^{a}teaux derivative, and $D_2$ the G\^{a}teaux derivative with respect to the second argument.

	The elastic energy $\mathcal{E}_{\eps}$ depends on the first and the second deformation gradients. The second-gradient term models the so-called nonsimple material behavior, while an additional term, which blows up as the determinant of the deformation gradient approaches zero, ensures the orientation preservation of admissible deformations and an infinite resistance of the material to infinite compression. The dissipative mechanism is captured by the pseudopotential $\mathcal{R}_{\eps}$, which accounts for the internal friction of the material by coupling the deformation gradient with the strain rate, as expected in the large-strain regime. Combined with $\mathcal{E}_\eps$, it results in a Kelvin-Voigt rheological model which allows for both the static frame-indifference of the energy and the dynamic frame-indifference of the dissipation \cite{Antman}.
	
	The global injectivity of the deformations is not enforced by imposing the so-called Ciarlet-Ne\v{c}as condition \cite{CiarletNecas}, as its membrane counterpart would fail to prevent self-interpenetration (for the inertial evolution with (self-)contact under this condition, see \cite{CGK,CGK2}). Instead, we require the deformations to satisfy a {\it uniform averaged invertibility condition} (see \cite{Davoli}). This condition allows the vertical fibers of the material to intersect, but requires their averages to remain at a distance bounded from below by a multiple of the distance between the corresponding points of the cross-section. In the membrane limit, it delivers a lower bound of the same type for the limit deformation, which prevents self-interpenetration of the membrane.
	
	The dynamics described in \eqref{eq1}--\eqref{eq11} arises from the balance of all forces acting on the body, accounting for the external load, the inertial reaction, and the constitutive elastic and viscous stresses.
	When the applied loads vary slowly over time, the inertial term becomes negligible and the equation of the system may be studied via a (quasi)static evolution approximation. In this case, a number of results have been established by means of variational techniques (see \cite{Ambrosio,DeGiorgi,martin,MielkeRoub} for an overview).
	The dimension-reduction approach, initially introduced for plate models \cite{CiarletDesuynder,CiarletKesavan}, is widely applied in modern variational mechanics. Following the foundational works by \textsc{Le Dret \& Raoult} \cite{LeDretRaoult} and \textsc{Bhattacharya \& James} \cite{Bhattacharya}, this methodology has been employed to derive lower-dimensional theories from full-dimensional elasticity (see also \cite{Acerbi,CasarinoPercivale,FrieseckeJamesMuller} among others), where characterizing the relation between the bulk theory and the corresponding membrane theory remains central \cite{FrieseckeJamesMuller2,LecumberryMuller,Pantz}. In the viscoelastic setting, plate and beam theories were derived in the quasistatic regime in \cite{FriedrichKruzik,FriedrichMachill}. By means of $\Gamma$-convergence arguments \cite{Braides,dalMaso}, the minimizers of the bulk energy are typically proved to converge to the minimizers of the reduced-dimensional model in the limit of vanishing thickness.		
		
	Dynamical problems, where inertial effects are not neglected, are far less understood. The presence of the inertial term introduces significant challenges, as variational techniques established for (quasi)static elasticity fail to apply directly to the dynamical framework, requiring additional arguments tailored for the time-dependent setting. For instance, the time-dependent von K\'arm\'an plate equation was derived from three-dimensional nonlinear elastodynamics in \cite{AbelsMoraMuller1}, and large time existence for thin vibrating plates was proved in \cite{AbelsMoraMuller2}.
	Evolutionary $\Gamma$-convergence (see \cite{MielkeRoub} for an overview) and $\Gamma$-convergence of gradient flows (see \cite{Ortner,Serfaty,Stefanelli} among others) are two examples of dynamical variational approaches. The former was introduced for rate-independent systems by {\sc Mielke, \sc Roub\'i\v{c}ek, \& \sc Stefanelli} \cite{MielkeRoubStefanelli}, the latter by {\sc Sandier \& Serfaty} \cite{Serfaty}. For gradient flows, a natural time discretization is given by the parabolic minimizing movements of {\sc De Giorgi} \cite{DeGiorgi}, in which each time step is a minimization problem.
	
	In this paper, we take advantage of the hyperbolic minimizing movements approach introduced by {\sc Bene\v{s}ov\'a, \sc Kampschulte, \& Schwarzacher} \cite{BenesovaJEMS,BenesovaNARWA} to characterize a notion of variational convergence for hyperbolic equations, allowing for the rigorous derivation of the corresponding membrane model. Two time-scales $h$ and $\tau$ are introduced, and the problem of finding solutions to \eqref{eq1}--\eqref{eq11} is recast into a hierarchy of auxiliary problems. For every $\eps>0$ and on each time interval of length $h$, one first considers a family of minimization problems on an inner time grid of step $\tau$ (the {\it incremental problem}). The natural candidate limit as $\tau\to0$ is a weak solution on that interval of an equation in which the inertial term is replaced by a difference quotient, over the outer time step $h$, involving the velocity on the previous interval (the {\it sequential problem}). Concatenating such solutions over the whole time interval leads to a single equation on $[0,T]$ with a time delay $h$ in the inertial term (the {\it time-delayed problem}), whose formal limit as $h\to0$ is \eqref{eq1}--\eqref{eq11}.
	
	Our main results can be summarized as follows (see also Figure~\ref{fig:limits}).

\medskip
\noindent{\bf Main Theorem} (informal). {\it Assume {\rm (A0)--(A3)} of Subsection~\ref{Assumptions}, and let $T \leq T_0$, where $T_0>0$ is independent of $\eps$, $h$, and $\tau$. Let $y_\eps$ denote the deformation rescaled to the domain $\Omega = \omega \times (-1/2,1/2)$ of unit thickness. Then, up to subsequences, the following statements hold.
	\begin{enumerate}
		\item[{\rm i)}] For fixed $\eps$, the two-time-scale scheme admits solutions, which converge, as $\tau\to0$ and then $h\to0$, to a weak solution of \eqref{eq1}--\eqref{eq11} {\rm (Theorem \ref{Theorem1})}.
		\item[{\rm ii)}] As $\eps\to0$, the solutions of the bulk incremental, sequential, and time-delayed problems constructed in {\rm i)} converge to solutions of the corresponding membrane problems. These are formulated in terms of a pair $(y,b)$, where $y:\omega\to\R^3$ is the membrane deformation and $b$ is the limit of $\eps^{-1}\partial_3 y_\eps$, and $\nabla_\eps y_\eps \to (\nabla' y\vert b)$ uniformly in space and time for the sequential and time-delayed problems {\rm (Theorems \ref{Theorem2}, \ref{Theorem3}, and \ref{Theorem5}, and Corollary \ref{CorConcat})}.
		\item[{\rm iii)}] As $h\to0$, the membrane time-delayed solutions converge to a measure-valued solution $(y,b,\nu)$ of the membrane dynamics
		\begin{align*}
			&\rho\, \partial_{tt}y + D_1 E(y,b) + D_2 R(y,\partial_t y,b,\partial_t b) = f_0 \quad \text{in}\; (0,T)\times\omega,\\
			&D_2 E(y,b) + D_4 R(y,\partial_t y,b,\partial_t b) = 0 \quad \text{in}\; (0,T)\times\omega,
		\end{align*}
		where $E$ and $R$ are the membrane energy and dissipation of Subsection~\ref{Membrane}, $f_0$ is the limit load of Assumption {\rm (A3)}, and the second-order stress is described by the Young measure $\nu$ generated by the second gradients {\rm (Theorem \ref{Theorem6})}. Moreover, $\lvert y(t,x') - y(t,z') \rvert \geq c_0 \lvert x'-z' \rvert$ for every $t \in [0,T]$ and all $x',z' \in \omega$, so that the limit membrane does not interpenetrate.
\end{enumerate}}
\medskip
	
	For fixed $\eps$, we prove that the incremental problem admits minimizers, that their time interpolants converge as $\tau\to0$ to a solution of the sequential problem, that the concatenation of such solutions solves the time-delayed problem, and that, up to subsequences, the solutions of the latter converge as $h\to0$ to a solution of \eqref{eq1}--\eqref{eq11} (Theorem \ref{Theorem1}). Existence of solutions to \eqref{eq1}--\eqref{eq11} by means of the same scheme was established in \cite[Theorem 3.2, p.~4653]{BenesovaJEMS} for a regularized version of the model, in which the energy and the dissipation are augmented by higher-order gradient terms of the deformation and of its velocity, respectively. Although these terms vanish in the continuous-time limit, they are not compatible with our dimension-reduction analysis. Indeed, constructing recovery sequences for the incremental functionals step by step in time would require the previous incremental solution to converge strongly in the higher-order norm introduced by the regularization, and such a convergence does not seem to be available. For this reason, Theorem \ref{Theorem1} is proved here directly for the unregularized scheme, on a time interval $[0,T_0]$ on which the averaged invertibility constraint remains inactive, with $T_0$ independent of $\eps$, $h$, and $\tau$ (see Subsection \ref{MainRes}). We then perform the limit $\eps\to0$ in the incremental problem, by $\Gamma$-convergence of the corresponding functionals (Theorem \ref{Theorem2}), and in the sequential and time-delayed problems (Theorems \ref{Theorem3} and \ref{Theorem5}). At each incremental step, $\tau$ is fixed, and the discrete velocity is a difference quotient of two deformations whose energy is bounded independently of $\eps$. Hence, provided that the data converge, the $\Gamma$-convergence relies only on these energy bounds. In the sequential and time-delayed problems, instead, compactness in time is needed, while the dissipation provides estimates independent of $\eps$ only on the time derivative of the Cauchy-Green tensor, and not on the velocity gradient (see Remark~\ref{RemKorn}). From the second outer time step on, the data of the incremental problem are given by the sequential solution on the previous interval, and their convergence relies on the same compactness. At the membrane level, where the same hierarchy of problems is defined, the solution of the membrane time-delayed problem obtained in this limit coincides with the concatenation of the solutions of the membrane sequential problem (Corollary~\ref{CorConcat}). The two-time-scale scheme could in principle be analyzed directly at the membrane level. We do not follow this route, since, from the second outer time step on, the membrane incremental problem takes its initial data from the limit of the bulk solutions on the previous interval, and a membrane scheme analyzed on its own would not be related to the bulk solutions (see Remark~\ref{RemMain}.$(i)$). Performing the limit $\eps\to0$ directly on \eqref{eq1}--\eqref{eq11}, as well as beyond the time $T_0$, remains open. Figure~\ref{fig:limits} summarizes the limits performed in the paper.
	
	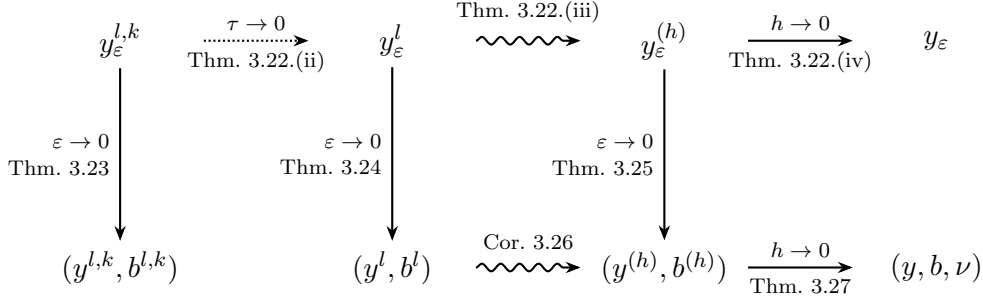
\begin{figure}
		\centering
		\begin{tikzpicture}[
			obj/.style={minimum width=2.2cm},
			lbl/.style={font=\scriptsize},
			conv/.style={thick,-{Stealth[length=2mm]}},
			interp/.style={thick,densely dotted,-{Stealth[length=2mm]}},
			concat/.style={thick,-{Stealth[length=2mm]},decorate,decoration={snake,amplitude=0.5mm,segment length=2.5mm,post length=2mm}}
			]
			\node[obj] (B1) at (0,0)     {$y_\eps^{l,k}$};
			\node[obj] (B2) at (3.6,0)   {$y_\eps^{l}$};
			\node[obj] (B3) at (7.2,0)   {$y_\eps^{(h)}$};
			\node[obj] (B4) at (10.8,0)  {$y_\eps$};
			\node[obj] (M1) at (0,-3)    {$(y^{l,k},b^{l,k})$};
			\node[obj] (M2) at (3.6,-3)  {$(y^{l},b^{l})$};
			\node[obj] (M3) at (7.2,-3)  {$(y^{(h)},b^{(h)})$};
			\node[obj] (M4) at (10.8,-3) {$(y,b,\nu)$};
			
			\draw[interp] (B1) -- (B2) node[midway,above,lbl]{$\tau\to0$} node[midway,below,lbl]{Thm.~\ref{Theorem1}.(ii)};
			\draw[concat] (B2) -- (B3) node[midway,above=1mm,lbl]{Thm.~\ref{Theorem1}.(iii)};
			\draw[conv]   (B3) -- (B4) node[midway,above,lbl]{$h\to0$} node[midway,below,lbl]{Thm.~\ref{Theorem1}.(iv)};
			
			\draw[concat] (M2) -- (M3) node[midway,above=1mm,lbl]{Cor.~\ref{CorConcat}};
			\draw[conv]   (M3) -- (M4) node[midway,above,lbl]{$h\to0$} node[midway,below,lbl]{Thm.~\ref{Theorem6}};
			
			\draw[conv] (B1) -- (M1) node[midway,left,lbl,align=right]{$\eps\to0$\\Thm.~\ref{Theorem2}};
			\draw[conv] (B2) -- (M2) node[midway,left,lbl,align=right]{$\eps\to0$\\Thm.~\ref{Theorem3}};
			\draw[conv] (B3) -- (M3) node[midway,left,lbl,align=right]{$\eps\to0$\\Thm.~\ref{Theorem5}};
		\end{tikzpicture}
		\caption{Overview of the limits performed in the paper. The top row refers to the bulk problem and the bottom row to the membrane problem. Here $y_\eps$ denotes the deformation $u$ rescaled to the domain $\Omega=\omega\times(-1/2,1/2)$ of unit thickness (see Subsection~\ref{Configuration}), while the membrane is described by the pair $(y,b)$, where $y$ is the membrane deformation and $b$ is the limit of the rescaled transverse derivative $\eps^{-1}\partial_3 y_\eps$. The superscripts $(l,k)$, $l$, and $(h)$ refer to the incremental, sequential, and time-delayed problems, respectively, and the last column to the continuous-time problem, which at the membrane level is solved in the measure-valued sense, $\nu$ being the Young measure generated by the second gradients. Solid arrows denote the convergence of solutions, whereas the dotted arrow denotes the convergence of the time interpolants of the incremental minimizers. Wavy arrows indicate the concatenation of the sequential solutions into the global interpolant, which is not a limit. At the membrane level, no limit as $\tau\to0$ is performed, since the solutions of the sequential problem are obtained as limits of the bulk ones (see Remark~\ref{RemMain}.$(i)$).}
		\label{fig:limits}
	\end{figure}

	The paper is organized as follows. Section \ref{Model} introduces the notation, the thin, scaled, and membrane configurations of the model together with the underlying assumptions, and the two-time-scale approximation scheme. Section \ref{sec:pre} introduces the variational framework, collects some auxiliary results, defines the notions of solutions, and states our main results (Theorems \ref{Theorem1}, \ref{Theorem2}, \ref{Theorem3}, \ref{Theorem5}, and \ref{Theorem6}, together with Corollary \ref{CorConcat}), which are proved in the remaining sections. Section \ref{sec:4} establishes existence of solutions to the bulk two-time-scale scheme (Theorem \ref{Theorem1}), and Section \ref{sec:5} proves its $\Gamma$-convergence to the membrane scheme (Theorem \ref{Theorem2}). Section \ref{S6} identifies the membrane limit of the sequential problem (Theorem \ref{Theorem3}). Section \ref{ProofT5} identifies the membrane limit of the time-delayed problem (Theorem \ref{Theorem5}) and its relation with the membrane sequential problem (Corollary \ref{CorConcat}). Finally, Section \ref{ProofT6} proves the convergence, as $h\to0$, of the solutions of the membrane time-delayed problem to a measure-valued solution of the membrane continuous-time problem (Theorem \ref{Theorem6}).
    
	\section{Model}
	\label{Model}
	In this section, we fix the notation (Subsection \ref{Notation}), introduce the thin, scaled, and membrane configurations of the model (Subsections \ref{setup}, \ref{Configuration}, and \ref{Membrane}), and present the two-time-scale approximation scheme (Subsection \ref{twoTSA}).
	
	\subsection{Notation}
		\label{Notation}
		Points of $\mathbb{R}^3$ are written as $x=(x',x_3)$, with $x'=(x_1,x_2)$. Throughout the paper, $\omega\subset\mathbb{R}^2$ is a bounded Lipschitz domain and, for every $\eps>0$, the reference configuration of the material is $\Omega_\eps = \omega\times(-\eps/2,\eps/2)$, whose lateral boundary is $P_\eps = \partial\omega\times(-\eps/2,\eps/2)$. We set $\Omega := \Omega_1$ and $P := P_1$ (see Figure~\ref{figure}), and we denote by ${\rm id}_\eps:\bar{\Omega}\to\mathbb{R}^3$ and ${\rm id}_0:\bar{\omega}\to\mathbb{R}^3$ the maps ${\rm id}_\eps(x) := (x',\eps x_3)$ and ${\rm id}_0(x') := (x',0)$.
		
		We denote by $\mathbb{M}^{m\times n}$ the space of real $m\times n$ matrices and by $\mathbb{M}^{m\times n\times k}$ the space of real third-order tensors, endowed with the Frobenius norm $\lvert\cdot\rvert$, namely, the square root of the sum of the squares of all entries. The scalar product is denoted by $\cdot$ for vectors and by $:$ for matrices and third-order tensors. We adopt the column notation, namely, $(v_1\vert v_2\vert v_3)$ is the matrix whose $i$-th column is $v_i\in\mathbb{R}^3$. For $F\in\mathbb{M}^{3\times3}$, $F^T$ is the transpose and ${\rm cof}\,F$ the cofactor matrix of $F$. Finally, ${\rm Id}\in\mathbb{M}^{3\times3}$ denotes the identity matrix, and we note that $\nabla'{\rm id}_0 = (e_1\vert e_2)\in\mathbb{M}^{3\times2}$, with $e_1,e_2$ the first two vectors of the canonical basis of $\mathbb{R}^3$.
		
		For a map $v:\Omega\to\mathbb{R}^3$, we write $\nabla v = (\nabla' v\vert\partial_3 v)$, where $\nabla' v = (\partial_1 v\vert\partial_2 v)\in\mathbb{M}^{3\times2}$ is the planar gradient. The second gradient $\nabla^2 v$ is the third-order tensor whose $(j,k)$-th entry is the vector $\partial_{jk} v\in\mathbb{R}^3$, for $j,k=1,2,3$, and we write it in block form as
		\begin{equation*}
			\nabla^2 v=
			\begin{pmatrix}
				(\nabla')^2 v & \nabla' \partial_3 v\\
				(\nabla' \partial_3 v)^T & \partial_{33} v
			\end{pmatrix},
		\end{equation*}
		where $(\nabla')^2 v\in\mathbb{M}^{3\times2\times2}$ collects the planar second derivatives and $\nabla'\partial_3 v = (\partial_{13}v\vert\partial_{23}v)$. For every $\eps>0$, the scaled gradient and the scaled second gradient are defined as
		\begin{equation*}
			\nabla_{\eps} v = (\nabla' v \vert \eps^{-1} \partial_3 v), \qquad 
			\nabla_{\eps}^2 v=
			\begin{pmatrix}
				(\nabla')^2 v & \eps^{-1}\nabla' \partial_3 v\\
				\eps^{-1}(\nabla' \partial_3 v)^T & \eps^{-2}\partial_{33} v
			\end{pmatrix},
		\end{equation*}
		so that $\lvert \nabla_\eps^2 v \rvert^2 = \lvert (\nabla')^2 v \rvert^2 + 2 \lvert \eps^{-1}\nabla'\partial_3 v \rvert^2 + \lvert \eps^{-2}\partial_{33} v \rvert^2$.
		
		We denote by $L^p$ and $W^{k,p}$ the standard Lebesgue and Sobolev spaces on any of the domains $\Omega_\eps$, $\Omega$, and $\omega$, we set $H^k := W^{k,2}$, and we omit the target space $\mathbb{R}^3$ when no confusion arises. The subscript $0$, as in $W_0^{k,p}$, indicates vanishing trace on the lateral boundary, namely, on $P_\eps$, $P$, or $\partial\omega$, and $H^{-1}$ is the dual of $H_0^1$. We write $\lVert\cdot\rVert$ and $(\cdot,\cdot)$ for the norm and the scalar product in $L^2$, the domain being clear from the context. For a Banach space $X$, we denote by $X'$ its dual and by $\langle\cdot,\cdot\rangle_X$ the duality pairing between $X'$ and $X$. Given a further Banach space $Y$, we write $X\hookrightarrow Y$ and $X\overset{c}{\hookrightarrow}Y$ for the continuous and the compact embedding of $X$ into $Y$, respectively. For $T>0$, the spaces $L^p(0,T;X)$, $H^1(0,T;X)$, and $C^0([0,T];X)$ are the usual Bochner spaces, and $C_w([0,T];X)$ is the space of weakly continuous maps from $[0,T]$ to $X$. The subscript $c$ indicates compact support, as in $C_c^\infty(\omega)$, $C_c^0(\mathbb{R}^m)$, or $C_c^1([0,T))$, the latter consisting of $C^1$ functions on $[0,T)$ which vanish in a neighborhood of $T$.
		
		Functions independent of $x_3$ are identified with functions defined on $\omega$. Accordingly, when a sequence of functions defined on $\Omega$ converges to a function independent of $x_3$, we state the convergence in the corresponding space over $\omega$ or $\bar{\omega}$, meaning convergence in the space over $\Omega$ or $\bar{\Omega}$ with the limit extended constantly in $x_3$. For such functions, the restriction to $P$ stands for the trace on $\partial\omega$. Finally, $c$ denotes a generic positive constant, which may change from line to line.

	\subsection{Thin configuration}
	\label{setup}
	The elastic energy $\mathcal{E}_\eps$ and the dissipation $\mathcal{R}_\eps$ are given by
	\begin{align*}
		&\mathcal{E}_\eps: W^{2,q}(\Omega_\eps;\R^3) \to \mathbb{R} \cup \{\infty\},\\[3mm]
		&\mathcal{E}_\eps(u) = \mathcal{K}_\eps(u) + \mathcal{J}_\eps(u) + \mathcal{T}_\eps(u),\\[3mm]
		&\qquad\qquad\qquad \mathcal{K}_\eps(u) = \displaystyle\int_{\Omega_\eps} \frac{1}{8} \lvert \nabla u^T \nabla u  - {\rm Id} \rvert_{\mathcal{C}}^2 \; \d x,\\[2mm]
		&\qquad\qquad\qquad \mathcal{J}_\eps(u) = \begin{cases}
			\displaystyle\int_{\Omega_\eps} \frac{1}{(\det \nabla u )^a}  \; \d x &\text{if} \; u \in \mathcal{A}_\eps,\\[5mm]
			\infty &\text{else},
		\end{cases}\\[2mm]
		&\qquad\qquad\qquad \mathcal{T}_\eps(u) = \displaystyle\int_{\Omega_\eps} \frac{1}{q} \lvert \nabla^2 u \rvert^q \; \d x,	\\[3mm]
		&\mathcal{R}_\eps: W^{2,q}(\Omega_\eps;\R^3) \times H^1(\Omega_\eps;\R^3) \to \mathbb{R},\\[3mm]
		&\mathcal{R}_\eps(u,v) = \int_{\Omega_\eps} \lvert \nabla v^T \nabla u  + \nabla u ^T \nabla v \rvert^2 \; \d x,
	\end{align*}
	where the set of admissible deformations $\mathcal{A}_\eps$ is defined, for all $\eps >0$ and a given constant $c_0>0$, as
	\begin{align*}
		\mathcal{A}_{\eps} = \biggl\{ &u \in W^{2,q}(\Omega_\eps;\mathbb{R}^3) \; : \; \det \nabla u (x) > 0 \text{ for a.e. }x \in \Omega_{\eps}, \\ &\; u(x) = x \; \text{ and } \; \nabla u(x) = {\rm Id} \; \text{ for } x \in P_\eps, \nonumber\\
		& \text{ and for all } x',z'\in\omega: \,  \biggl\lvert	\fint_{-\eps/2}^{\eps/2} u(x',x_3) \; \d x_3 - \fint_{-\eps/2}^{\eps/2} u(z',x_3) \; \d x_3	 \biggr\rvert \geq c_0 \lvert x' - z' \rvert \biggr\}.
	\end{align*}
	The first term $\mathcal{K}_\eps$ in the definition of $\mathcal{E}_{\eps}$ is the St.~Venant-Kirchhoff energy \cite{martin}, where $\lvert \nabla u^T \nabla u - {\rm Id} \rvert_{\mathcal{C}}^2 = \mathcal{C} (\nabla u^T \nabla u - {\rm Id}) : (\nabla u^T \nabla u - {\rm Id})$ for a positive definite tensor $\mathcal{C}$ of elastic constants. The exponents $a$ and $q$ in $\mathcal{J}_\eps$ and $\mathcal{T}_\eps$ ensure the strict positivity of the determinant of deformations with finite energy \cite{healeyKromer} (see Assumption (A0)). In the definition of $\mathcal{A}_\eps$, the Ciarlet-Ne\v{c}as condition is replaced by the uniform averaged invertibility condition \cite{Davoli}, and the constant $c_0$ is a datum of the problem (see Assumption (A0)).
	
	\subsubsection{Assumptions}
		\label{Assumptions}
		We assume that the data of the problem satisfy
		$$
		\textbf{(A0)} \; T,q,a,c_0 \in \mathbb{R}, \; \text{s.t. } T>0, \; q>3, \; a > (3q)/(q-3), \; c_0 > 0,
		$$
		and, throughout the paper, we fix $\alpha \in (0,1-3/q)$, so that $W^{2,q}$ and $W^{1,q}$ embed compactly into $C^{1,\alpha}$ and $C^{0,\alpha}$, respectively. Moreover, we assume that the initial deformation is admissible and that its average elastic energy is finite independently of $\eps$, and that, as $\eps\to0$, the initial velocity and the external load become independent of the thickness variable, in average and uniformly, respectively. More precisely, there exist constants $c_1>c_0$ and $c_2>0$, $v_0 \in L^2(\omega;\R^3)$, and $f_0 \in C^0([0,T]\times\bar{\omega};\R^3)$ such that
		\begin{align*}
			&\textbf{(A1)} \; u_0 \in \mathcal{A}_\eps \; \text{s.t. } \mathcal{E}_\eps(u_0) \leq c_2 \text{ meas}(\Omega_{\eps}) \; \text{and} \\
			&\qquad \biggl\lvert \fint_{-\eps/2}^{\eps/2} u_0(x',x_3) \; \d x_3 - \fint_{-\eps/2}^{\eps/2} u_0(z',x_3) \; \d x_3 \biggr\rvert \geq c_1 \lvert x' - z' \rvert \; \text{for all } x',z' \in \omega,
		\end{align*}
		$$
		\textbf{(A2)} \; w_0 \in L^2(\Omega_{\eps};\R^3) \; \text{s.t. } \frac{1}{\text{meas}(\Omega_\eps)} \int_{\Omega_\eps} \lvert w_0(x) - v_0(x') \rvert^2 \; \d x \to 0 \; \text{as } \eps \to 0,
		$$
		$$
		\textbf{(A3)} \; f \in C^0([0,T] \times \bar{\Omega}_{\eps};\R^3) \; \text{s.t. } \sup_{(t,x) \in [0,T]\times\bar{\Omega}_\eps} \lvert f(t,x) - f_0(t,x') \rvert \to 0 \; \text{as } \eps \to 0.
		$$
		Observe that the bound in Assumption (A1) reflects the extensive property of the stored elastic energy and that the lack of such a condition would allow for a blowup of the energy in the membrane limit. The second condition in (A1) requires the initial deformation to satisfy the averaged invertibility condition with a constant strictly larger than the one in $\mathcal{A}_\eps$. Assumption (A2) implies, in particular, that the average kinetic energy of the initial configuration is finite independently of $\eps$ (see Subsection \ref{SubIn}).

	\subsection{Scaled configuration}
	\label{Configuration}
	
	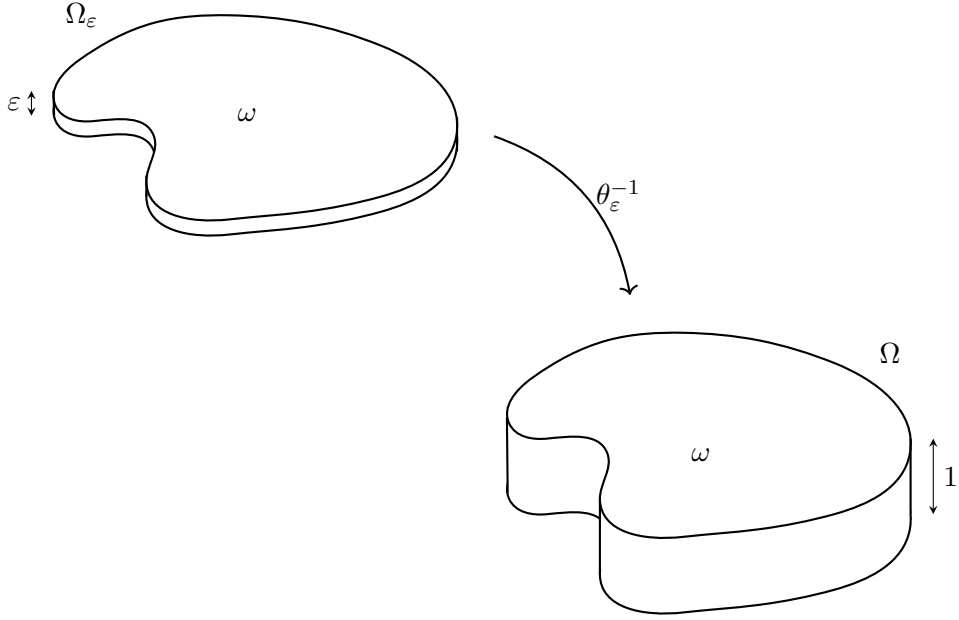
\begin{figure}
		\centering
		\begin{tikzpicture}
			\draw[thick] (1,-0.2) to[closed, curve through =
			{ (0.9,-0.8) (2,-1.3) (3,-1.2) (4,-1) (5,-0.2) (4,1) (3,1.3) (2,1.4) (1,1.3) (0,0.8) (-0.3,0.2) }] (0.2,0);
			
			\fill[white] (-0.32,0.27) -- (-0.32,3) -- (5.01,3) -- (5.01,0);
			
			\draw[thick, fill=white] (1,0) to[closed, curve through =
			{ (0.9,-0.6) (2,-1.1) (3,-1) (4,-0.8) (5,0) (4,1.2) (3,1.5) (2,1.6) (1,1.5) (0,1) (-0.3,0.4) }] (0.2,0.2);

			\draw[thick] (5.01,-0.2) -- (5.01,0.25);
			\draw[thick] (-0.32,0.27) -- (-0.33,0.6);
			\draw[thick] (0.9,-0.8) -- (0.9,-0.52);
			
			
			\draw[name path=A, thick] (7,-5.2) to[closed, curve through =
			{ (6.9,-5.8) (8,-6.3) (9,-6.2) (10,-6) (11,-5.2) (10,-4) (9,-3.7) (8,-3.6) (7,-3.7) (6,-4.2) (5.7,-4.8) }] (6.2,-5);
			
			\fill[white] (5.68,-4.73) -- (5.68,-2) -- (11.01,-2) -- (11.01,-5);
			
			\draw[name path=B, thick, fill=white] (7,-4.2) to[closed, curve through =
			{ (6.9,-4.8) (8,-5.3) (9,-5.2) (10,-5) (11,-4.2) (10,-3) (9,-2.7) (8,-2.6) (7,-2.7) (6,-3.2) (5.7,-3.8) }] (6.2,-4);
			
			\fill[white] (6.9,-5.8) -- (6.9,-5) -- (7.1,-5.2);
			
			\draw[thick] (11.01,-5) -- (11.01,-4.00);
			\draw[thick] (5.68,-4.73) -- (5.67,-3.72);
			\draw[thick] (6.9,-5.8) -- (6.9,-4.8);
			
			
			\draw [stealth-stealth](-0.62,0.27) -- (-0.62,0.6) node[midway,left] () {$\eps$};
			\draw [stealth-stealth](11.31,-5) -- (11.31,-4.00) node[midway,right] () {$1$};
			\draw [->,thick] (5.5,0) to [bend left=30]  node [midway,right]  () {$\theta_{\eps}^{-1}$} (7.3,-2.1);
			
			
			\draw (2.5,0.5) node[anchor=north east] {$\omega$};
			\draw (8.5,-4) node[anchor=north east] {$\omega$};
			
			\draw (0.4,1.9)  node[anchor=north east] {$\Omega_{\eps}$};
			\draw (11.01,-2.6)  node[anchor=north east] {$\Omega$};
			
		\end{tikzpicture}
		\caption{Change of variables: thin and scaled domains}
		\label{figure}
	\end{figure}
	
	To perform the dimension reduction analysis, we map the thin domain $\Omega_{\eps}$ onto the unit-thickness cylinder $\Omega$ (see Figure~\ref{figure}). Let $\theta_{\eps}:\Omega \rightarrow \Omega_{\eps}$ be given by $\theta_{\eps}(x) = (x',\eps x_3)$, and introduce the rescaled deformation $y_{\eps} = u \circ \theta_{\eps}$ as well as the rescaled data $f_{\eps} = f \circ \theta_{\eps}$, $y_{0\eps}=u_0 \circ \theta_\eps$, and $v_{0\eps} = w_0 \circ \theta_{\eps}$. We also define the scaled elastic energy and the scaled dissipation as the energy and the dissipation per unit thickness, namely, $E_\eps(y) := \eps^{-1}\mathcal{E}_\eps(y\circ\theta_\eps^{-1})$ and $R_\eps(y,v) := \eps^{-1}\mathcal{R}_\eps(y\circ\theta_\eps^{-1},v\circ\theta_\eps^{-1})$. Since $(\nabla u)\circ\theta_\eps = \nabla_\eps y_\eps$ and $(\nabla^2 u)\circ\theta_\eps = \nabla_\eps^2 y_\eps$, a change of variables gives
		\begin{align}
			& E_\eps: W^{2,q}(\Omega;\R^3) \to \mathbb{R} \cup \{\infty\},\label{defEA}\\[3mm]
			&E_\eps(u) = K_\eps(u) + J_\eps(u) + T_\eps(u),\label{defE}\\[3mm]
			&\qquad\qquad\qquad K_\eps(u) = \displaystyle\int_{\Omega} \frac{1}{8} \lvert \nabla_\eps u^T \nabla_\eps u  - {\rm Id} \rvert_{\mathcal{C}}^2 \; \d x,	\label{defEComp1}\\[2mm]
			&\qquad\qquad\qquad J_\eps(u) = \begin{cases}
				\displaystyle\int_{\Omega} \frac{1}{(\det \nabla_\eps u )^a}  \; \d x &\text{if} \; u \in A_\eps,\\[5mm]
				\infty &\text{else},
			\end{cases}\label{defEComp}\\[2mm]
			&\qquad\qquad\qquad T_\eps(u) = \displaystyle\int_{\Omega} \frac{1}{q} \lvert \nabla_\eps^2 u \rvert^q \; \d x,	\label{defEComp2}\\[3mm]
			& R_\eps: W^{2,q}(\Omega;\R^3) \times H^1(\Omega;\R^3)  \to \mathbb{R},\\[3mm]
			&R_\eps(u,v) = \int_{\Omega} \lvert \nabla_\eps v^T \nabla_\eps u  + \nabla_\eps u ^T \nabla_\eps v \rvert^2 \; \d x,\label{defR}
		\end{align}
		where the set $A_\eps = \{ u \circ \theta_\eps \, : \, u \in \mathcal{A}_\eps \}$ of scaled admissible deformations reads as
		\begin{align*}
			A_{\eps} = \biggl\{ &u \in W^{2,q}(\Omega;\mathbb{R}^3) \; : \; \det \nabla_\eps u (x) > 0 \text{ for  a.e. } x \in \Omega, \\ &\; u = {\rm id}_\eps \; \text{ and } \; \nabla_\eps u = {\rm Id} \; \text{ on } P, \nonumber\\
			& \text{ and for all } x', z'\in\omega: \biggl\lvert	\int_{-1/2}^{1/2} u(x', x_3) \; \d x_3 - \int_{-1/2}^{1/2} u(z',x_3) \; \d x_3	 \biggr\rvert \geq c_0 \lvert x' - z' \rvert \biggr\}.
		\end{align*}
		Note that ${\rm id}_\eps = \theta_\eps$, so that the boundary conditions in $A_\eps$ are those of $\mathcal{A}_\eps$ in the scaled variables. For $u \in C^0(\bar{\Omega};\R^3)$, we set $\bar{u}(x') := \int_{-1/2}^{1/2} u(x',x_3) \, \d x_3$ and
		$$
		\Lambda(u) := \inf_{x',z' \in \omega, \; x' \neq z'} \frac{\lvert \bar{u}(x') - \bar{u}(z') \rvert}{\lvert x' - z' \rvert},
		$$
		so that the averaged invertibility condition in $A_\eps$ reads $\Lambda(u) \geq c_0$. Testing \eqref{eq1} with $\phi\circ\theta_\eps^{-1}$, where $\phi$ is a test function on $\Omega$, changing variables, and dividing by $\eps$, we obtain the scaled version of the problem \eqref{eq1}--\eqref{eq11}, namely,
		\begin{align}
			&\rho \partial_{tt}y_{\eps} + D E_{\eps} (y_{\eps}) + D_2 R_{\eps} (y_{\eps},\partial_t y_{\eps}) = f_{\eps} \quad \text{in}\; (0,T)\times\Omega,\label{eq1S}\\
			&y_{\eps}(0) = y_{0\eps} \quad \text{and} \quad \partial_t y_{\eps}(0)=v_{0\eps} \quad \text{in} \; \Omega. \label{eq11S}
	\end{align}

	\subsubsection{Scaled data}
	\label{SubIn}
	For later use in the proofs, we reformulate Assumptions (A1)--(A3) on the $\eps$-independent domain $\Omega$.
	Note that we will still refer to them simply as Assumptions (A1)--(A3), as no additional requirements are imposed. The derivation below relies on Lemma \ref{ClosAB}, which is stated and proved in Subsection \ref{Aux}, independently of the present subsection.
		
		Assumption (A1) delivers $y_{0\eps} \in A_\eps$, $\Lambda(y_{0\eps}) \geq c_1$, and $E_\eps(y_{0\eps}) = \eps^{-1}\mathcal{E}_\eps(u_0) \leq c_2 \lvert \omega \rvert$. Hence, by Lemma \ref{ClosAB} with $c = c_1$, there exist $y_0 \in W^{2,q}(\omega;\R^3)$, $b_0 \in W^{1,q}(\omega;\R^3)$, and $d_0 \in L^q(\Omega;\R^3)$ such that, up to a subsequence, as $\eps \to 0$,
		\begin{equation}
			\label{y0conv}
			y_{0\eps} \rightharpoonup y_0 \quad \text{in} \quad W^{2,q}(\omega) \qquad \text{and} \qquad y_{0\eps} \to y_0 \quad \text{in} \quad C^{1,\alpha}(\bar{\omega}),
		\end{equation}
		\begin{equation}
			\label{b0conv}
			\nabla_\eps y_{0\eps} \to (\nabla' y_0 \vert b_0) \qquad \text{in} \quad C^{0,\alpha}(\bar{\omega}),
		\end{equation}
		and
		\begin{equation}
			\label{d0conv}
			\nabla_\eps^2 y_{0\eps} \rightharpoonup \begin{pmatrix}
				(\nabla')^2 y_0 & \nabla' b_0 \\
				(\nabla' b_0)^T & d_0
			\end{pmatrix} \quad \text{in}\quad L^q(\Omega).
		\end{equation}
		In what follows, every limit as $\eps\to0$ is taken along this subsequence, or along a further subsequence of it, which we do not relabel. By the same lemma,
		\begin{align*}
			&y_0\rvert_P = {\rm id}_0\rvert_P, \quad \nabla' y_0\rvert_P= \nabla'{\rm id}_0, \quad b_0\rvert_P=(0,0,1)^T,\\
			&\lvert y_0(x') - y_0(z') \rvert \geq c_1 \lvert x' - z' \rvert \quad \text{for all } x',z' \in \omega,
		\end{align*}
		and $\det(\nabla' y_0\vert b_0) \geq \mu > 0$ in $\bar{\omega}$.
	
	By the change of variables $\theta_\eps$, Assumption (A2) results in
	$$
	\lVert v_{0\eps} - v_0 \rVert^2_{L^2(\Omega)} = \frac{\lvert \omega \rvert}{\text{meas}(\Omega_\eps)} \int_{\Omega_\eps} \lvert w_0(x) - v_0(x') \rvert^2 \; \d x \to 0 \quad \text{as } \eps \to 0,
	$$
	namely,
	\begin{equation}
		\label{v0conv}
		v_{0\eps} \to v_0 \quad \text{in} \quad L^2(\omega).
	\end{equation}
	Similarly, Assumption (A3) reads
	\begin{equation}
		\label{f0conv}
		f_{\eps} \to f_0 \quad \text{in} \quad  C^0([0,T] \times \bar{\omega};\R^3).
	\end{equation}
	In particular, for every $\eps$ sufficiently small, one has $\lVert v_{0\eps} \rVert_{L^2(\Omega)} \leq \lVert v_0 \rVert_{L^2(\omega)} + 1$ and $\lVert f_\eps \rVert_{C^0([0,T]\times\bar{\Omega})} \leq \lVert f_0 \rVert_{C^0([0,T]\times\bar{\omega})} + 1$, and we always consider such values of $\eps$.

	\subsection{Membrane configuration}
	\label{Membrane}
	The membrane is described by a pair $(y,b)$, where $y:\omega\to\mathbb{R}^3$ is the deformation of the mid-surface and $b:\omega\to\mathbb{R}^3$ plays the role of the limit of the scaled transverse derivative $\eps^{-1}\partial_3 y_\eps$ (see Lemma \ref{ScaledComp}). The membrane elastic energy $E$ and the membrane dissipation $R$ are obtained from $E_\eps$ and $R_\eps$ by replacing $\nabla_\eps u$ and $\nabla_\eps v$ with $(\nabla' u\vert\beta)$ and $(\nabla' v\vert w)$, respectively, and $\nabla_\eps^2 u$ with
		$$
		\begin{pmatrix}
			(\nabla')^2 u & \nabla' \beta\\
			(\nabla' \beta)^T & 0
		\end{pmatrix},
		$$
		namely,
		\begin{align}
			&E: W^{2,q}(\omega)\times W^{1,q}(\omega) \to \mathbb{R} \cup \{\infty\},\label{MdefEA}\\[3mm]
			&E(u,\beta) = K(u,\beta) + J(u,\beta) + T(u,\beta),\label{MdefE}\\[3mm]
			&\qquad\qquad\qquad K(u,\beta) = \displaystyle\int_{\omega} \frac{1}{8} \lvert (\nabla' u\vert \beta)^T (\nabla' u \vert \beta) - {\rm Id} \rvert_{\mathcal{C}}^2 \; \d x,	\label{MdefEComp1}\\[2mm]
			&\qquad\qquad\qquad J(u,\beta) = \begin{cases}
				\displaystyle\int_{\omega} \frac{1}{(\det (\nabla' u \vert \beta))^a}  \; \d x &\text{if} \; (u,\beta) \in B,\\[5mm]
				\infty &\text{else},
			\end{cases}\label{MdefEComp}\\[3mm]
			&\qquad\qquad\qquad T(u,\beta) = \displaystyle\int_{\omega} \frac{1}{q} \biggl( \lvert (\nabla')^2 u \rvert^2 + 2 \lvert \nabla' \beta \rvert^2 \biggr)^{q/2}\; \d x,	\label{MdefEComp2}\\[2mm]
			&R: W^{2,q}(\omega) \times H^1(\omega) \times W^{1,q}(\omega) \times L^2(\omega) \to \mathbb{R},\\[3mm]
			&R(u,v,\beta,w) = \int_{\omega} \lvert (\nabla'v \vert w)^T (\nabla' u \vert \beta) + (\nabla' u \vert \beta)^T (\nabla' v \vert w) \rvert^2 \; \d x,\label{MdefR}
		\end{align}
		where the set of admissible membrane deformations $B$ reads as
		\begin{align*}
			B = \biggl\{ &(y,b) \in W^{2,q}(\omega) \times W^{1,q}(\omega) \; : \; \det (\nabla' y(x') \vert b(x')) > 0 \; \text{for a.e. } x' \in \omega,  \nonumber\\[3mm]
			&\; y = {\rm id}_0, \; \nabla' y =  \nabla'{\rm id}_0, \; \text{ and } \; b = (0,0,1)^T \; \text{ on } P, \\ & \text{ and for all } x',z' \in \omega: \lvert	 y(x') - y(z')	 \rvert \geq c_0 \lvert x' - z' \rvert \biggr\}.
		\end{align*}
		The dynamics of the membrane, corresponding to the bulk problem \eqref{eq1S}--\eqref{eq11S}, is then given by
		\begin{align}
			&\rho \partial_{tt}y + D_1 E (y,b) + D_2 R (y,\partial_t y,b,\partial_t b) = f_0 \quad \text{in}\; (0,T)\times\omega,\label{Meq1S}\\[3mm]
			&D_2 E (y,b) + D_4 R (y,\partial_t y,b,\partial_t b) = 0 \quad \text{in}\; (0,T)\times\omega,\label{Meq2S}\\[3mm]
			&y(0) = y_0, \quad b(0) = b_0, \quad \text{and} \quad \partial_t y(0)=v_0 \quad \text{in} \; \omega. \label{Meq11S}
	\end{align}

	The initial data $y_0$, $b_0$, $v_0$ and the load $f_0$ are those obtained in Subsection \ref{SubIn}. In particular, the properties of $(y_0,b_0)$ established there guarantee $(y_0,b_0) \in B$. Moreover, by Lemma \ref{ClosAB}, $K_\eps(y_{0\eps}) \to K(y_0,b_0)$, $J_\eps(y_{0\eps}) \to J(y_0,b_0)$, and $T(y_0,b_0) \leq \liminf_{\eps\to0} T_\eps(y_{0\eps})$. Hence, $E(y_0,b_0) \leq \liminf_{\eps\to0} E_\eps(y_{0\eps}) \leq c$.

	\subsection{Two-time-scale approximation scheme}
	\label{twoTSA}
	We introduce the hyperbolic minimizing movements approach \cite{BenesovaJEMS,BenesovaNARWA}, which is a two-time-scale approximation scheme. Let $\Pi_h$ be a uniform partition of the time interval $[0,T]$ with time step $h = T/N$ ($N \in \mathbb{N}$), namely, $\Pi_h := \{	0= t_0 < t_1 < \dots < t_N = T \}$ with $t_{l} - t_{l-1} = h$ for every $l = 1, \dots, N$. For every $l=1,\dots,N$, let $\Pi_h^\tau := \{ (l-1)h = t_{l_0} < t_{l_1} < \dots < t_{l_M} = lh \}$ be the uniform partition of the time interval $[(l-1)h,lh]$ with time step $t_{l_k} -t_{l_{k-1}} = \tau >0$ for every $k=1,\dots,M=h/{\tau}$ ($M \in \mathbb{N}$). The same partitions are used for both the bulk and the membrane problems.
	
	\subsubsection{Bulk scheme}
	At first, we approximate the problem \eqref{eq1S}--\eqref{eq11S} with a time-delayed problem by replacing the inertial term in \eqref{eq1S} with a difference quotient, namely
\begin{equation}
	\label{tDelay}
	\rho \frac{\partial_t y_\eps- \partial_t y_\eps(\cdot-h)}{h} + D E_{\eps} (y_{\eps}) + D_2 R_{\eps} (y_{\eps},\partial_t y_{\eps}) = f_{\eps} \quad \text{in}\; (0,T)\times\Omega,
\end{equation}
with initial conditions given by
\begin{equation}
	\label{initialVelocity}
	y_{\eps}(0) = y_{0\eps} \quad \text{and} \quad \partial_t y_{\eps}(t-h)=v_{0\eps} \quad \text{for } t \in (0,h).
\end{equation}
	
	Then, we decompose the time-delayed problem \eqref{tDelay}--\eqref{initialVelocity} into $N$ sequential problems. For each $l = 1,\dots, N$, we consider
\begin{equation}
	\label{sequential}
	\rho \frac{\partial_t y_{\eps}^l- v_\eps^{l-1}}{h} + D E_{\eps} (y_\eps^l) + D_2 R_{\eps} (y_\eps^l,\partial_t y_\eps^l) = f_\eps \quad \text{in}\; ((l-1)h,lh)\times\Omega,
\end{equation}
satisfying
\begin{equation}
	\label{initialSequential}
	y_\eps^l((l-1)h) = y_\eps^{l-1}((l-1)h), \quad v_{\eps}^{l-1} = \partial_t y_{\eps}^{l-1}(\cdot-h),
\end{equation}
for $l=2,\dots,N$, while
\begin{equation}
	\label{initialSequentialI}
	y_\eps^1(0) = y_{0\eps}, \quad v_\eps^0 = v_{0\eps},
\end{equation}
	for the first iteration, i.e. $l=1$.
	As equation \eqref{sequential} is meant to be solved sequentially, the initial conditions play a crucial role. For each $l$, the term $\partial_t y_\eps^{l-1} (t-h)$ is the initial condition on the velocity obtained as a solution of the previous time step. Note, however, that for $l=1$ the initial conditions are those of the time-delayed problem \eqref{initialVelocity}.
	
	On each time interval $[(l-1)h,lh]$, the problem is then parabolic, and it can be studied by means of De Giorgi's parabolic minimizing movements scheme on the partition $\Pi_h^\tau$. For every $l = 1, \dots, N$, we define the following incremental minimization problem
	\begin{align}
		y_\eps^{l,k} \in  \argmin_{y \in A_{\eps}} \biggl\{ &E_{\eps}(y) + \tau R_{\eps} \biggl( y_\eps^{l,k-1} , \frac{y - y_\eps^{l,k-1}}{\tau}  \biggr) - \tau \biggl( f_\eps^{l,k-1} , \frac{y-y_\eps^{l,k-1}}{\tau} \biggr) \nonumber \\&+ \frac{\tau \rho}{2h} \biggl\lVert \frac{y - y_\eps^{l,k-1}}{\tau} - v_\eps^{l-1, k-1} 	\biggr\rVert^2 	 \biggr\}, \quad \text{for} \; k = 1,\dots,M, \label{incrP}\\[5mm]
		=: \argmin_{y \in A_{\eps}} \; &W_{\eps}^{l, k-1} (y), \quad \text{for} \; k = 1,\dots,M,\nonumber
	\end{align}
	with initial conditions
	\begin{align}
		&y_\eps^{1,0}=y_{0\eps}, \quad y_\eps^{l,0} = y_\eps^{l-1}((l-1)h) \quad \text{for}\; l = 2,\dots,N,\label{incrIC1}\\
		&v_\eps^{l-1,k-1} = \frac{1}{\tau} \int_{(k-1)\tau + (l-1)h}^{k\tau + (l-1)h} v_\eps^{l-1}(t) \, \d t \quad \text{for} \; k = 1,\dots,M\label{incrIC2},\\
		&\qquad\;\;\;\;\; = \frac{1}{\tau} \int_{(k-1)\tau + (l-2)h}^{k\tau + (l-2)h} \partial_t y_\eps^{l-1}(t) \, \d t \quad \text{for} \; k = 1,\dots,M .\nonumber
	\end{align}
	Note that, due to \eqref{initialSequentialI}, the initial velocity at the first time step equals $v_{0\eps}$ for every $k = 1,\dots,M$. The forcing $f_\eps^{l,k-1}$ in \eqref{incrP} is given as the time average of $f_\eps$ over the corresponding subinterval, namely,
	$$
	f_\eps^{l,k-1} = \frac{1}{\tau}\int_{(k-1)\tau+(l-1)h}^{k\tau+(l-1)h} f_\eps(t) \; \d t , \quad \text{for} \; k = 1,\dots,M,
	$$
	and, for $l=1$, we simply write $f_\eps^{k-1} := f_\eps^{1,k-1}$.
	
	\subsubsection{Membrane scheme}
The membrane time-delayed problem reads as
\begin{align}
	&\rho \frac{\partial_t y- \partial_t y(\cdot-h)}{h} + D_1 E (y,b) + D_2 R (y,\partial_t y,b,\partial_t b) = f_0 \quad \text{in}\; (0,T)\times\omega, \label{MtDelay}\\[3mm]
	&D_2 E (y,b) + D_4 R (y,\partial_t y,b,\partial_t b) = 0 \quad \text{in}\; (0,T)\times\omega,\label{MtDelay1}
\end{align}
with initial conditions given by
\begin{equation}
	\label{MinitialVelocity}
	y(0) = y_0, \quad b(0) = b_0, \quad \text{and} \quad \partial_t y(t-h)=v_0 \quad \text{for } t \in (0,h).
\end{equation}
	Then, the membrane sequential problem reads, for each $l = 1,\dots, N$, as
\begin{align}
	&\rho \frac{\partial_t y^l- v^{l-1}}{h} + D_1 E (y^l,b^l) + D_2 R (y^l,\partial_t y^l,b^l,\partial_t b^l) = f_0 \quad \text{in}\; ((l-1)h,lh)\times\omega,		\label{Msequential}\\[3mm]
	&D_2 E (y^l,b^l)+ D_4 R (y^l,\partial_t y^l,b^l,\partial_t b^l) = 0 \quad \text{in}\; ((l-1)h,lh)\times\omega,\label{Msequential1}
\end{align}
satisfying
\begin{align}
	&y^l((l-1)h) = y^{l-1}((l-1)h), \quad b^l((l-1)h) = b^{l-1}((l-1)h), \nonumber\\
	&v^{l-1} = \partial_t y^{l-1}(\cdot-h), \qquad \text{for } l=2,\dots,N, \label{MinitialSequential}
\end{align}
while
\begin{equation}
	\label{MinitialSequentialI}
	y^1(0) = y_0, \quad b^1(0) = b_0, \quad \text{and} \quad v^0 = v_0.
\end{equation}
	
	Finally, the incremental minimization problem for the membrane is given, for every $l = 1,\dots,N$, by
	\begin{equation}
		\label{MincrP}
		(y^{l,k},b^{l,k}) \in  \argmin_{(y,b) \in B} \; W^{l, k-1} (y,b), \quad \text{for} \; k = 1,\dots,M,
	\end{equation}
	where
	\begin{align*}
		W^{l, k-1} (y,b) =&E(y,b) + \tau R \biggl( y^{l,k-1} , \frac{y - y^{l,k-1}}{\tau} , b^{l,k-1} , \frac{b - b^{l,k-1}}{\tau} \biggr) \nonumber \\[3mm] &- \tau \biggl( f^{l,k-1} , \frac{y-y^{l,k-1}}{\tau} \biggr) + \frac{\tau \rho}{2h} \biggl\lVert \frac{y - y^{l,k-1}}{\tau} - v^{l-1, k-1} 	\biggr\rVert^2 	,
	\end{align*}
	with initial conditions
\begin{align}
	&y^{1,0}=y_0, \quad b^{1,0}=b_0, \label{MincrIC0}\\[3mm]
	&y^{l,0} = y^{l-1}((l-1)h), \quad b^{l,0} = b^{l-1}((l-1)h) \quad \text{for}\; l = 2,\dots,N,\label{MincrIC1}\\[3mm]
	&v^{l-1,k-1} = \frac{1}{\tau} \int_{(k-1)\tau + (l-1)h}^{k\tau + (l-1)h} v^{l-1}(t) \, \d t \quad \text{for} \; k = 1,\dots,M.\label{MincrIC2}
\end{align}
	Due to \eqref{MinitialSequentialI}, for $l=1$, the initial condition on the velocity is equal to $v_0$ for every $k = 1,\dots,M$. The forcing $f^{l,k-1}$ in the definition above is given as the time average of $f_0$ over the corresponding subinterval, namely,
	$$
	f^{l,k-1} = \frac{1}{\tau}\int_{(k-1)\tau+(l-1)h}^{k\tau+(l-1)h} f_0(t) \; \d t , \quad \text{for} \; k = 1,\dots,M,
	$$
	and, for $l=1$, we simply write $f^{k-1} := f^{1,k-1}$.
	
	\section{Preliminaries and main results}
\label{sec:pre}
In this section, we introduce the variational framework, collect some auxiliary results, define the notions of solutions to the problems of the two-time-scale approximation scheme, and state our main results.

	\subsection{Variational framework and duality pairings}
	\label{SVar}
  We characterize the operators appearing in equations \eqref{eq1S}, \eqref{Meq1S}, and \eqref{Meq2S}, namely, the derivatives of the functionals in \eqref{defEA}--\eqref{defR} and \eqref{MdefEA}--\eqref{MdefR}, by establishing their explicit structure and their mapping properties.
	
	Let $A_\eps \subset W^{2,q}(\Omega)$ be the set of admissible deformations. Its tangent space $V=\{ \phi \in W^{2,q}(\Omega) \cap H_0^1(\Omega) \; : \; \nabla_\eps \phi \rvert_P = 0 \}$ serves as the test space for the variational formulation. In the effective domain of $E_\eps$, namely, for $y \in A_\eps$ with $E_\eps(y) < \infty$, the expression \eqref{defDE} below defines $DE_\eps(y) \in (W^{2,q}(\Omega))'$, which is the G\^{a}teaux derivative of $E_\eps$ along $V$ if $\Lambda(y) > c_0$. Similarly, for any $y \in A_\eps$ and any $v \in H^1(\Omega)$, the derivative of $R_\eps$ with respect to its second argument $D_2 R_\eps (y,v)$ is a well-defined element in $(H^1(\Omega))'$. In view of the minimization problem, we characterize those derivatives through their action along the directions $\phi \in V$ and $\nu \in H^1(\Omega)$ as follows. For any $s \in \mathbb{R}$, we compute the derivative of the real differentiable functions $s \mapsto E_\eps (y + s \phi)$ and $s \mapsto R_\eps (y , v + s \nu)$ at $s=0$. One has
	\begin{align}
		\label{defDE}
		\langle DE_\eps(y) , \phi \rangle_{V} =& \langle D K_\eps (y) , \phi \rangle_{V} + \langle D J_\eps (y) , \phi \rangle_{V} + \langle D T_\eps (y) , \phi \rangle_{V} \nonumber \\
		=& \int_{\Omega} \biggl( \frac{1}{4} \mathcal{C} \bigl( \nabla_\eps y^T \nabla_\eps y - {\rm Id} \bigr) : \bigl( \nabla_\eps \phi^T \nabla_\eps y + \nabla_\eps y^T \nabla_\eps \phi  \bigr)  \nonumber\\& \qquad- a \frac{{\rm cof}\, \nabla_\eps y}{(\det \nabla_\eps y)^{a+1}} : \nabla_\eps \phi + \lvert \nabla_\eps^2 y \rvert^{q-2} \nabla_\eps^2 y : \nabla_\eps^2 \phi \biggr) \; \d x,
	\end{align}
	and
	\begin{align}
		\label{defDR}
		\langle D_2 R_\eps(y,v) , \nu \rangle_{H^1} =&  \int_{\Omega} 2 \bigl( \nabla_\eps v^T \nabla_\eps y + \nabla_\eps y^T \nabla_\eps v \bigr) : \bigl( \nabla_\eps \nu^T \nabla_\eps y + \nabla_\eps y^T \nabla_\eps \nu \bigr) \; \d x  .
	\end{align}
	
	Analogously, we consider the set of admissible membrane deformations $B \subset W^{2,q}(\omega) \times W^{1,q}(\omega)$, its tangent space $V_0 = V_0^1 \times W_0^{1,q}(\omega)$, where $V_0^1 := \{ \phi \in W^{2,q}(\omega)\cap H_0^1(\omega) \; : \; \nabla' \phi \rvert_P = 0 \}$, and the membrane energy and dissipation \eqref{MdefEA}--\eqref{MdefR}. For every $(y,b)$ in the effective domain of $E$, namely, $(y,b) \in B$ with $E(y,b) < \infty$, and every $(v,w) \in H^1(\omega) \times L^2(\omega)$, we compute
	\begin{align*}
		&\frac{\d}{\d s} E(y + s \phi,b + s \beta) \biggr\rvert_{s=0} \quad \text{for any } (\phi,\beta) \in V_0,\\
		&\frac{\d}{\d s} R(y,v + s \nu ,b,w + s \zeta) \biggr\rvert_{s= 0} \quad \text{for any } (\nu,\zeta) \in H^1(\omega)\times L^2(\omega),
	\end{align*}
	obtaining
	\begin{align}
		\label{MdefDE}
		&\langle DE (y,b) , (\phi,\beta) \rangle_{V_0} = \langle D_1E (y,b) , \phi \rangle_{W^{2,q}\cap H_0^1} + \langle D_2E (y,b) , \beta \rangle_{W_0^{1,q}} \nonumber\\
		&= \langle D_1 K (y,b) , \phi \rangle_{W^{2,q}\cap H_0^1} + \langle D_1 J (y,b) , \phi \rangle_{W^{2,q}\cap H_0^1} + \langle D_1 T (y,b) , \phi \rangle_{W^{2,q}\cap H_0^1} \nonumber \\
		&\quad +\langle D_2 K (y,b) , \beta \rangle_{W_0^{1,q}} + \langle D_2 J (y,b) , \beta \rangle_{W_0^{1,q}} + \langle D_2 T (y,b) , \beta \rangle_{W_0^{1,q}} \nonumber \\
		&= \int_{\omega} \biggl( \frac{1}{4} \mathcal{C} \bigl( (\nabla' y \vert b)^T (\nabla' y \vert b ) - {\rm Id} \bigr) : \bigl( (\nabla' \phi \vert \beta)^T (\nabla' y \vert b) + (\nabla' y \vert b)^T (\nabla' \phi \vert \beta)  \bigr)  \nonumber\\& \qquad- a \frac{{\rm cof}\, (\nabla' y \vert b)}{(\det (\nabla' y \vert b))^{a+1}} : (\nabla' \phi \vert \beta) \nonumber\\
		&\quad+ \biggl( \lvert (\nabla')^2 y \rvert^2 + 2 \lvert \nabla' b \rvert^2 \biggr)^{(q-2)/2} \begin{pmatrix}
			(\nabla')^2 y & \nabla' b \\
			(\nabla' b)^T & 0
		\end{pmatrix} : \begin{pmatrix}
			(\nabla')^2 \phi & \nabla' \beta \\
			(\nabla'\beta)^T & 0
		\end{pmatrix}  \biggr) \; \d x,
	\end{align}
	and
	\begin{align}
	&\langle D_2 R (y,v,b,w) , \nu \rangle_{H^1} + ( D_4 R (y,v,b,w) , \zeta ) \nonumber \\
	&\qquad\qquad=\int_{\omega} 2 \bigl( (\nabla' v \vert w)^T (\nabla' y \vert b) + (\nabla' y \vert b)^T (\nabla' v \vert w) \bigr) \nonumber \\ &\qquad\qquad\quad\qquad: \bigl( (\nabla' \nu \vert \zeta)^T (\nabla' y \vert b) + (\nabla' y \vert b)^T (\nabla' \nu \vert \zeta) \bigr) \; \d x  .\label{MdefDR}
	\end{align}
	Observe that $R$, $D_2 R$, and $D_4 R$ depend on the rates $(v,w)$ only through the symmetric matrix $(\nabla' v \vert w)^T(\nabla' y \vert b) + (\nabla' y \vert b)^T(\nabla' v \vert w)$. For a time-dependent pair $(y,b)$ and $(v,w) = (\partial_t y,\partial_t b)$, this matrix is the time derivative of the Cauchy-Green tensor $C = (\nabla' y \vert b)^T(\nabla' y \vert b)$. Accordingly, whenever $C \in H^1(0,T;L^2(\omega))$, we write $R(y,\partial_t y,b,\partial_t b)$, $D_2 R(y,\partial_t y,b,\partial_t b)$, and $D_4 R(y,\partial_t y,b,\partial_t b)$ for the expressions in \eqref{MdefR} and \eqref{MdefDR} in which that matrix is replaced by $\partial_t C$, also when $\partial_t \nabla' y$ and $\partial_t b$ are not defined separately as functions. If, in addition, $(\nabla' y \vert b) \in L^\infty((0,T)\times\omega)$, then $D_2 R(y,\partial_t y,b,\partial_t b) \in L^2(0,T;H^{-1}(\omega))$ and $D_4 R(y,\partial_t y,b,\partial_t b) \in L^2(0,T;L^2(\omega))$.
	
	For the sake of notational simplicity, we write $\langle \cdot , \cdot \rangle$, without subscript, for the action of each of the derivatives above on an admissible direction, as given by the explicit expressions \eqref{defDE}--\eqref{MdefDR}. When both arguments belong to $L^2$, the pairing reduces to the scalar product, and we write $(\cdot,\cdot)$. When a derivative is shown to belong to a larger dual space, as for $DT_\eps$ in Subsection \ref{S45}, we indicate the corresponding space as a subscript, as in $\langle \cdot , \cdot \rangle_{H_0^1}$.

\subsection{Auxiliary results}
\label{Aux}
In this subsection, we recall the definition of $\Gamma$-convergence and collect some auxiliary results which will be used throughout the paper.

\begin{definition}[$\Gamma$-convergence]
	\label{DefGamma}
	Let $X$ be a subset of a reflexive Banach space. Given functionals $F_{j},\, F : X \to \R\cup\{\infty\}$ for $j \in \mathbb{N}$, we say that the sequence $\{ F_j \}_j$ {\it $\Gamma$-converges} to $F$ with respect to the topology $\sigma$, and we write $F = \Gamma_\sigma\text{-}\lim_{j \to \infty} F_j$, if the following conditions hold.
	\begin{enumerate}
		\item[{\rm i)}] ($\Gamma$-$\liminf$ inequality) For every $y_{j} \to y$ in $\sigma$ we have
		\begin{equation*}
			F(y)\leq\liminf_{j\to \infty} F_{j}(y_{j}).
		\end{equation*}
		\item[{\rm ii)}] (Recovery sequence) For every $\hat{y} \in X$ there exists $\hat{y}_{j} \to \hat{y}$ in $\sigma$ such that
		\begin{equation*}
			\limsup_{ j\to \infty} F_{j}(\hat{y}_{j}) \leq F(\hat{y}).
		\end{equation*}
	\end{enumerate}
	See \cite{dalMaso} for an overview.
\end{definition}

\begin{lemma}[Compactness of scaled gradients]
	\label{ScaledComp}
	Let $p \in (1,\infty]$, and let $\{y_\eps\}_{\eps}$ be a family in $W^{1,p}(\Omega;\R^3)$ such that $\lVert y_\eps \rVert_{L^p(\Omega)} + \lVert \nabla_\eps y_\eps \rVert_{L^p(\Omega)} \leq c$ for every $\eps \in (0,1)$. Then, up to a subsequence (not relabeled), $y_\eps \rightharpoonup y$ in $W^{1,p}(\Omega)$ and $\nabla_\eps y_\eps \rightharpoonup (\nabla' y \vert b)$ in $L^p(\Omega)$, weakly-$*$ if $p=\infty$, where $y \in W^{1,p}(\omega;\R^3)$ is independent of $x_3$ and $b \in L^p(\Omega;\R^3)$. If, in addition, $y_\eps \in W^{2,p}(\Omega;\R^3)$ and $\lVert \nabla_\eps^2 y_\eps \rVert_{L^p(\Omega)} \leq c$ for every $\eps \in (0,1)$, then $y \in W^{2,p}(\omega;\R^3)$, $b \in W^{1,p}(\omega;\R^3)$ is independent of $x_3$, and
	$$
	\nabla_\eps^2 y_\eps \rightharpoonup \begin{pmatrix}
		(\nabla')^2 y & \nabla' b \\
		(\nabla' b)^T & d
	\end{pmatrix} \quad \text{in}\quad L^p(\Omega),
	$$
	weakly-$*$ if $p=\infty$, for some $d \in L^p(\Omega;\R^3)$.
\end{lemma}
\begin{proof}
	The statement is standard \cite{Bhattacharya,Davoli,LeDretRaoult}. The independence of $y$ from $x_3$ follows from $\lVert \partial_3 y_\eps \rVert_{L^p(\Omega)} = \eps \lVert \eps^{-1}\partial_3 y_\eps \rVert_{L^p(\Omega)} \leq c\,\eps$, and that of $b$ from $\lVert \partial_3 (\eps^{-1}\partial_3 y_\eps) \rVert_{L^p(\Omega)} = \eps \lVert \eps^{-2}\partial_{33} y_\eps \rVert_{L^p(\Omega)} \leq c\,\eps$.
\end{proof}

\begin{lemma}
		\label{TOT}
		Let $\mathcal{D}$ be a bounded, open, and Lipschitz domain in $\mathbb{R}^d$, where $d=2,3$, let $q>3$, and let $X$ be a closed subspace of $W^{2,q}(\mathcal{D}) \cap H_0^1(\mathcal{D})$.
		Let $\mathcal{H}_1$ be the set of functions $\Psi = \psi\,\phi$, where $\phi \in L^1(0,T;X)\cap H^1(0,T;H_0^1(\mathcal{D}))$ and $\psi \in C_c^1([0,T))$ is such that, for some $h_0 \in (0,T)$, $\psi > 0$ in $[0,T-h_0)$ and $\psi = 0$ in $[T-h_0,T)$. Then, $\mathcal{H}_1$ is dense in $\mathcal{H}_2:=\{ \Phi \in L^1(0,T;X)\cap H^1(0,T;H_0^1(\mathcal{D}))  \; : \; \Phi(T) = 0 \}$.
	\end{lemma}
	\begin{proof}
		Let $\theta \in C^1(\mathbb{R})$ be nondecreasing with $\theta = 0$ in $(-\infty,1]$, $\theta > 0$ in $(1,\infty)$, and $\theta = 1$ in $[2,\infty)$, and set $\theta_n(t) := \theta(n(T-t))$ for $n > 2/T$. Then, $\theta_n \in C_c^1([0,T))$, $\theta_n > 0$ in $[0,T-1/n)$, $\theta_n = 0$ in $[T-1/n,T)$, $\theta_n = 1$ in $[0,T-2/n]$, and $\lvert \theta_n' \rvert \leq c\,n$. Hence, for every $\Psi \in \mathcal{H}_2$, one has $\theta_n \Psi \in \mathcal{H}_1$, with $h_0 = 1/n$. By dominated convergence, $\theta_n \Psi \to \Psi$ in $L^1(0,T;X)$ and $\theta_n \partial_t \Psi \to \partial_t \Psi$ in $L^2(0,T;H_0^1(\mathcal{D}))$. Finally, as $\Psi(T) = 0$, one has $\lVert \Psi(t) \rVert_{H_0^1(\mathcal{D})} \leq (T-t)^{1/2} \lVert \partial_t \Psi \rVert_{L^2(t,T;H_0^1(\mathcal{D}))}$, so that
		\begin{align*}
			\int_0^T \lvert \theta_n' \rvert^2 \lVert \Psi \rVert_{H_0^1(\mathcal{D})}^2 \; \d t &\leq c\, n^2 \int_{T-2/n}^{T-1/n} (T-t) \; \d t \; \lVert \partial_t \Psi \rVert_{L^2(T-2/n,T;H_0^1(\mathcal{D}))}^2 \\
			&\leq c\, \lVert \partial_t \Psi \rVert_{L^2(T-2/n,T;H_0^1(\mathcal{D}))}^2 \to 0.
		\end{align*}
		Therefore, $\theta_n \Psi \to \Psi$ in $L^1(0,T;X)\cap H^1(0,T;H_0^1(\mathcal{D}))$.
\end{proof}
	
	\begin{lemma}[Coercivity of $E_\eps$]
		\label{W2qnormEps}
		Let $\eps \in (0,1)$, and let $\{u_n\}_n \subset A_\eps$ be a sequence of admissible deformations for which there exists a constant $\kappa>0$ independent of $n$ and $\eps$ such that $E_\eps (u_n) \leq \kappa$. Then, there exists a constant $M_\kappa>0$, depending only on $\kappa$, $\Omega$, $P$, $q$, and $\mathcal{C}$, and nondecreasing in $\kappa$, such that
		$$
		\lVert u_n \rVert_{W^{2,q}(\Omega)} + \lVert \nabla_\eps u_n \rVert_{W^{1,q}(\Omega)} + \lVert \nabla_\eps^2 u_n \rVert_{L^q(\Omega)} \leq M_\kappa.
		$$
	\end{lemma}
	\begin{proof}
		Throughout the proof, $c$ denotes a constant depending only on $\kappa$, $\Omega$, $P$, $q$, and $\mathcal{C}$. As $E_\eps(u_n) \leq \kappa$, the definitions of $T_\eps$ and $K_\eps$ and the positive definiteness of $\mathcal{C}$ give
		$$
		\lVert \nabla_\eps^2 u_n \rVert_{L^q(\Omega)} + \lVert \nabla_\eps u_n \rVert_{L^4(\Omega)} \leq c,
		$$
		which implies, as $\eps < 1$,
		$$
		\lVert \nabla^2 u_n \rVert_{L^q(\Omega)} + \lVert \nabla(\nabla_\eps u_n) \rVert_{L^q(\Omega)} + \lVert \nabla u_n \rVert_{L^4(\Omega)} \leq c,
		$$
		where $q >3$ by assumption. If $q > 4$, Theorem \ref{GN} gives $\lVert \nabla_\eps u_n \rVert_{L^q(\Omega)} + \lVert \nabla u_n \rVert_{L^q(\Omega)} \leq c$, whereas, if $3 < q \leq 4$, this bound follows from H\"older's inequality. Hence, $\lVert \nabla_\eps u_n \rVert_{W^{1,q}(\Omega)} + \lVert \nabla u_n \rVert_{W^{1,q}(\Omega)} \leq c$. Moreover, since $u_n\rvert_P = {\rm id}_\eps\rvert_P$ for every $n$, Theorem \ref{PoincareGen} gives $\lVert u_n \rVert_{L^q(\Omega)} \leq c$, and the claim follows with some constant $M_\kappa$. Finally, as the assumption $E_\eps(u_n) \leq \kappa$ also holds with any $\kappa' \geq \kappa$ in place of $\kappa$, the constant $M_\kappa$ can be replaced by $\inf_{\kappa' \geq \kappa} M_{\kappa'}$, which is nondecreasing in $\kappa$.
	\end{proof}

	\begin{lemma}[Positivity of the determinant everywhere]
		\label{PositiveDet}
		Let $\eps \in (0,1)$, and let $\{ u_n\}_n \subset A_\eps$ be a sequence of admissible deformations for which there exists a constant $\kappa >0$ independent of $\eps$ and $n$ such that $E_\eps(u_n) \leq \kappa$. Then, there exists a constant $\mu = \mu(\kappa, \Omega,q,a,\mathcal{C}) >0$ such that
		\begin{equation*}
			\det \nabla_\eps u_n (x) > \mu , \qquad \forall x \in \bar{\Omega},
		\end{equation*}
		where $\mu$ is independent of the deformation $u_n$ and $\eps$.
		
		In particular, it holds for the initial deformation $y_{0\eps}$, namely,
		$$
		\det \nabla_\eps y_{0\eps} (x) > \mu , \qquad \forall x \in \bar{\Omega}.
		$$
	\end{lemma}
	\begin{proof}
		The proof is a straightforward application of Theorem \ref{thHealeyKromer} to those sequences along which the energy $E_\eps$ is finite independently of $\eps$ and $n$. Observe that Lemma \ref{W2qnormEps} guarantees that $\det \nabla_\eps u_n $ is uniformly bounded in $C^{0,1-3/q}(\bar{\Omega})$, independently of $\eps$ and $n$. Also, the initial deformation $y_{0\eps}$ satisfies $E_\eps(y_{0\eps}) \leq c$ uniformly. 
	\end{proof}
	
	\begin{lemma}[Stability of the averaged invertibility condition]
		\label{InvStab}
		The following statements hold.
		\begin{enumerate}
			\item[{\rm i)}] There exists a constant $C_\omega>0$, depending only on $\omega$, such that $\lvert \Lambda(u) - \Lambda(w) \rvert \leq C_\omega \lVert u - w \rVert_{W^{1,\infty}(\Omega)}$ for every $u,w \in W^{1,\infty}(\Omega;\R^3)$.
			\item[{\rm ii)}] Let $\gamma_0 := (q-3)/(7q-6)$. For every $M>0$, there exists a constant $C_M>0$, depending only on $M$, $\Omega$, and $q$, and nondecreasing in $M$, such that
			$$
			\lvert \Lambda(u) - \Lambda(w) \rvert \leq C_M \lVert u - w \rVert^{2\gamma_0}
			$$
			for every $u,w \in W^{2,q}(\Omega;\R^3)$ with $\lVert u \rVert_{W^{2,q}(\Omega)} + \lVert w \rVert_{W^{2,q}(\Omega)} \leq M$. In particular, if $I \subset \R$ is an interval and $y \in H^1(I;L^2(\Omega))$ satisfies $\lVert y(t) \rVert_{W^{2,q}(\Omega)} \leq M/2$ for every $t \in I$ and $\int_I \lVert \partial_t y \rVert^2 \, \d t \leq M$, then $\lvert \Lambda(y(t)) - \Lambda(y(s)) \rvert \leq C_M M^{\gamma_0} \lvert t-s \rvert^{\gamma_0}$ for all $t,s \in I$.
		\end{enumerate}
	\end{lemma}
	\begin{proof}
		$(i)$ Let $g := \bar{u} - \bar{w}$, so that $\lVert \nabla' g \rVert_{L^\infty(\omega)} \leq \lVert u - w \rVert_{W^{1,\infty}(\Omega)}$. Since $\omega$ is a bounded Lipschitz domain, it is quasiconvex, namely, the inner distance in $\omega$ is bounded by $C_\omega$ times the Euclidean distance, where $C_\omega \geq 1$ depends only on $\omega$ \cite{ErnGuermond}. Hence, $\lvert g(x') - g(z') \rvert \leq C_\omega \lVert \nabla' g \rVert_{L^\infty(\omega)} \lvert x' - z' \rvert$ for all $x',z' \in \omega$. Therefore, for all $x' \neq z'$,
		$$
		\frac{\lvert \bar{u}(x') - \bar{u}(z') \rvert}{\lvert x' - z' \rvert} \geq \frac{\lvert \bar{w}(x') - \bar{w}(z') \rvert}{\lvert x' - z' \rvert} - C_\omega \lVert u - w \rVert_{W^{1,\infty}(\Omega)} \geq \Lambda(w) - C_\omega \lVert u - w \rVert_{W^{1,\infty}(\Omega)}.
		$$
		Taking the infimum over $x' \neq z'$, we obtain $\Lambda(u) \geq \Lambda(w) - C_\omega \lVert u - w \rVert_{W^{1,\infty}(\Omega)}$, and exchanging the roles of $u$ and $w$ delivers the claim.
		
		$(ii)$ Let $v := u - w$. We apply Theorem \ref{GN} with $d=3$, $m=2$, $r=q$, $p=\infty$, and $L^2$ in place of $L^q$, for $j=1$ and $j=0$, with exponents $\lambda_1 = 5q/(7q-6)$ and $\lambda_0 = 3q/(7q-6)$, respectively. Since $q>3$, one has $1/2 \leq \lambda_1 < 1$ and $\lambda_0 < \lambda_1$. As $\lVert v \rVert \leq C \lVert v \rVert_{W^{2,q}(\Omega)} \leq C M$, the lower-order terms can be absorbed, and we get $\lVert v \rVert_{W^{1,\infty}(\Omega)} \leq C_M \lVert v \rVert^{1-\lambda_1} = C_M \lVert v \rVert^{2\gamma_0}$. The first claim follows then from $(i)$. The second one follows from $\lVert y(t) - y(s) \rVert \leq \lvert t-s \rvert^{1/2} \bigl(\int_I \lVert \partial_t y \rVert^2 \, \d t\bigr)^{1/2} \leq M^{1/2} \lvert t-s \rvert^{1/2}$. Finally, as the assumption $\lVert u \rVert_{W^{2,q}(\Omega)} + \lVert w \rVert_{W^{2,q}(\Omega)} \leq M$ also holds with any $M' \geq M$ in place of $M$, the constant $C_M$ can be replaced by $\inf_{M' \geq M} C_{M'}$, which is nondecreasing in $M$.
	\end{proof}
	
	\begin{lemma}[Coercivity of $R_\eps$]
		\label{H1normEps}
		Let $\eps \in (0,1)$ be fixed and let $\{u_n\}_n \subset A_\eps$ be a sequence of admissible deformations for which there exists a constant $\kappa>0$ independent of $n$ such that $E_\eps (u_n) \leq \kappa$. Then, there exists a constant $C_K>0$, possibly depending on $\kappa$, $\eps$, $\Omega$, $P$, $q$, $a$, and $\mathcal{C}$, but independent of $n$, such that, for every $v \in H_0^1(\Omega)$,
		$$
		R_\eps(u_n,v) \geq C_K \int_\Omega \Bigl( \lvert v(x) \rvert^2 + \lvert \nabla_\eps v(x) \rvert^2 \Bigr) \; \d x.
		$$
	\end{lemma}
	\begin{proof}
		By Lemmas \ref{W2qnormEps} and \ref{PositiveDet}, the scaled gradients $\nabla_\eps u_n$ are bounded in $C^{0,\alpha}(\bar{\Omega})$ and satisfy $\det \nabla_\eps u_n \geq \mu$ in $\bar{\Omega}$, uniformly in $n$. Consider the maps $w_n = u_n \circ \theta_\eps^{-1}$ defined on the thin domain $\Omega_\eps$. Since $(\nabla w_n) \circ \theta_\eps = \nabla_\eps u_n$, the gradients $\nabla w_n$ belong to the set $\mathcal{F}_K$ in \eqref{FK}, with $D=\Omega_\eps$, $\lambda = \alpha$, and some $K$ depending on $\kappa$ and $\eps$. Given $v \in H_0^1(\Omega)$, we apply Theorem \ref{KornMR} on $\Omega_\eps$, with $\Gamma_D = P_\eps$, to $v \circ \theta_\eps^{-1}$. Changing variables and dividing by $\eps$ deliver the claim.
	\end{proof}
	
	\begin{remark}
			\label{RemKorn}
			The constant $C_K$ in Lemma \ref{H1normEps} degenerates as $\eps \to 0$. Indeed, let $u = {\rm id}_\eps$, so that $\nabla_\eps u = {\rm Id}$ and $E_\eps(u) = \lvert \Omega \rvert$, and, for a nonzero $\varphi \in C_c^\infty(\omega)$, consider
			$$
			v(x) = \bigl( -\eps x_3 \partial_1 \varphi(x'), -\eps x_3 \partial_2 \varphi(x'), \varphi(x') \bigr)^T.
			$$
			The off-diagonal blocks of $\nabla_\eps v$ are $-\nabla'\varphi$ and $(\nabla'\varphi)^T$, which cancel in $\nabla_\eps v + (\nabla_\eps v)^T$, so that the latter reduces to its upper-left block $-2\eps x_3 (\nabla')^2 \varphi$. Hence,
			$$
			R_\eps(u,v) = \frac{\eps^2}{3} \lVert (\nabla')^2 \varphi \rVert^2_{L^2(\omega)}, \qquad \text{whereas} \qquad \int_\Omega \lvert \nabla_\eps v \rvert^2 \; \d x \geq 2 \lVert \nabla' \varphi \rVert^2_{L^2(\omega)},
			$$
			and Lemma \ref{H1normEps} can hold only with $C_K \leq \eps^2 \lVert (\nabla')^2 \varphi \rVert^2_{L^2(\omega)} / \bigl( 6 \lVert \nabla' \varphi \rVert^2_{L^2(\omega)} \bigr)$. At the membrane level, the coercivity is lost. Indeed, as $\eps \to 0$, $(u,\eps^{-1}\partial_3 u)$ converges to $({\rm id}_0,(0,0,1)^T) \in B$ and $v$ converges to $(0,0,\varphi)^T$ in $H^1(\Omega)$, whereas $\eps^{-1}\partial_3 v = w := (-\partial_1\varphi, -\partial_2\varphi, 0)^T$ for every $\eps$. Thus, $\nabla_\eps v$ converges uniformly to the matrix
			$$
			\bigl( \nabla'(0,0,\varphi)^T \big\vert w \bigr) = \begin{pmatrix}
				0 & 0 & -\partial_1 \varphi \\
				0 & 0 & -\partial_2\varphi \\
				\partial_1\varphi & \partial_2\varphi & 0
			\end{pmatrix},
			$$
			and the membrane dissipation vanishes, although $w \neq 0$, namely,
			$$
			R\bigl({\rm id}_0,(0,0,\varphi)^T,(0,0,1)^T,w\bigr) = 0.
			$$
			For this reason, in the limits as $\eps\to0$, the bound on the dissipation is used only through the time derivative of the Cauchy-Green tensor $(\nabla_\eps y)^T \nabla_\eps y$ (see Lemma \ref{TimeComp} below).
	\end{remark}

	Next, we state the analogue of Lemmas \ref{W2qnormEps} and \ref{PositiveDet} for the membrane model. No analogue of Lemma \ref{H1normEps} holds for the membrane dissipation $R$ (see Remark \ref{RemKorn}).

\begin{lemma}[Coercivity of $E$ and positivity of the determinant]
	\label{MW2qnormEps}\label{MPositiveDet}
	Let $\{(u_n,\beta_n)\}_n \subset B$ be a sequence of membrane admissible deformations for which there exists a constant $\kappa>0$ independent of $n$ such that $E (u_n,\beta_n) \leq \kappa$. Then, there exist constants $M_\kappa>0$, depending only on $\kappa$, $\omega$, $P$, $q$, and $\mathcal{C}$, and $\mu = \mu(\kappa,\omega,q,a,\mathcal{C}) > 0$ such that
	$$
	\lVert u_n \rVert_{W^{2,q}(\omega)} + \lVert \beta_n \rVert_{W^{1,q}(\omega)} \leq M_\kappa \qquad \text{and} \qquad \det (\nabla' u_n\vert \beta_n) (x') > \mu \quad \forall x' \in \bar{\omega}.
	$$
	In particular, $\det (\nabla' y_0 \vert b_0) > \mu$ in $\bar{\omega}$.
\end{lemma}
\begin{proof}
	The proofs of Lemmas \ref{W2qnormEps} and \ref{PositiveDet} apply with $\omega$, $(\nabla' u_n \vert \beta_n)$, and $\bigl((\nabla')^2 u_n, \nabla' \beta_n\bigr)$ in place of $\Omega$, $\nabla_\eps u_n$, and $\nabla_\eps^2 u_n$, respectively, and with $E(y_0,b_0) \leq c$ by Subsection \ref{Membrane}.
\end{proof}

		\begin{lemma}[Limits of admissible deformations]
		\label{ClosA}
		The following statements hold.
		\begin{enumerate}
			\item[{\rm i)}] Let $\eps \in (0,1)$ be fixed, and let $\{u_n\}_n \subset A_\eps$ and $u \in C^0(\bar{\Omega};\R^3)$ be such that $u_n \to u$ in $C^0(\bar{\Omega})$ and $E_\eps(u_n) \leq \kappa$ for every $n$, with $\kappa > 0$ independent of $n$. Then, $u \in A_\eps$,
			$$
			u_n \rightharpoonup u \quad \text{in} \quad W^{2,q}(\Omega) \qquad \text{and} \qquad u_n \to u \quad \text{in} \quad C^{1,\alpha}(\bar{\Omega}),
			$$
			and $\det \nabla_\eps u \geq \mu$ in $\bar{\Omega}$, with $\mu$ as in Lemma \ref{PositiveDet}. Moreover, $\Lambda(u_n) \to \Lambda(u)$, $K_\eps(u_n) \to K_\eps(u)$, $J_\eps(u_n) \to J_\eps(u)$, and $T_\eps(u) \leq \liminf_{n\to\infty} T_\eps(u_n)$.
			\item[{\rm ii)}] Let $\{(u_n,\beta_n)\}_n \subset B$ and $u,\beta \in C^0(\bar{\omega};\R^3)$ be such that $u_n \to u$ and $\beta_n \to \beta$ in $C^0(\bar{\omega})$ and $E(u_n,\beta_n) \leq \kappa$ for every $n$, with $\kappa > 0$ independent of $n$. Then, $(u,\beta) \in B$,
			\begin{alignat*}{6}
				&u_n &&\rightharpoonup u \quad &&\text{in} \quad W^{2,q}(\omega), \qquad &&u_n &&\to u \quad &&\text{in} \quad C^{1,\alpha}(\bar{\omega}),\\
				&\beta_n &&\rightharpoonup \beta \quad &&\text{in} \quad W^{1,q}(\omega), \qquad &&\beta_n &&\to \beta \quad &&\text{in} \quad C^{0,\alpha}(\bar{\omega}),
			\end{alignat*}
			and $\det (\nabla' u \vert \beta) \geq \mu$ in $\bar{\omega}$, with $\mu$ as in Lemma \ref{MPositiveDet}. Moreover, $K(u_n,\beta_n) \to K(u,\beta)$, $J(u_n,\beta_n) \to J(u,\beta)$, and $T(u,\beta) \leq \liminf_{n\to\infty} T(u_n,\beta_n)$.
		\end{enumerate}
	\end{lemma}
	\begin{proof}
		$(i)$ By Lemma \ref{W2qnormEps}, $\{u_n\}_n$ is bounded in $W^{2,q}(\Omega)$. Hence, every subsequence has a further subsequence which converges weakly in $W^{2,q}(\Omega)$ and, by the compact embedding, strongly in $C^{1,\alpha}(\bar{\Omega})$. Its limit is $u$ by the uniform convergence, so that the whole sequence converges. Passing to the limit on $P$ gives $u = {\rm id}_\eps$ and $\nabla_\eps u = {\rm Id}$ on $P$. Lemma \ref{PositiveDet} gives $\det \nabla_\eps u_n > \mu$ in $\bar{\Omega}$, and consequently $\det \nabla_\eps u \geq \mu$ in $\bar{\Omega}$. By Lemma \ref{InvStab}.$(i)$, $\Lambda(u_n) \to \Lambda(u)$, so that $\Lambda(u) \geq c_0$. Thus, $u \in A_\eps$. Finally, the uniform convergence of $\nabla_\eps u_n$ and the lower bound on the determinants give $K_\eps(u_n) \to K_\eps(u)$ and $J_\eps(u_n) \to J_\eps(u)$, whereas $T_\eps(u) \leq \liminf_{n\to\infty} T_\eps(u_n)$ by the weak lower semicontinuity of the norm.
		
		$(ii)$ The proof is the same, with Lemmas \ref{MW2qnormEps} and \ref{MPositiveDet} in place of Lemmas \ref{W2qnormEps} and \ref{PositiveDet}. The condition $\lvert u(x') - u(z') \rvert \geq c_0 \lvert x' - z' \rvert$ passes to the limit by the uniform convergence, and $T$ is weakly lower semicontinuous on $W^{2,q}(\omega) \times W^{1,q}(\omega)$, being convex and continuous.
	\end{proof}
	
		\begin{lemma}[Limits of admissible deformations as $\eps \to 0$]
		\label{ClosAB}
		For every $\eps \in (0,1)$, let $u_\eps \in A_\eps$ be such that $E_\eps(u_\eps) \leq \kappa$, with $\kappa > 0$ independent of $\eps$. Then, there exist $(y,b) \in B$ and $d \in L^q(\Omega;\R^3)$ such that, up to a subsequence (not relabeled),
		\begin{alignat*}{6}
			&u_\eps &&\rightharpoonup y \quad &&\text{in} \quad W^{2,q}(\omega), \qquad &&u_\eps &&\to y \quad &&\text{in} \quad C^{1,\alpha}(\bar{\omega}),\\
			&\nabla_\eps u_\eps &&\rightharpoonup (\nabla' y \vert b) \quad &&\text{in} \quad W^{1,q}(\omega), \qquad &&\nabla_\eps u_\eps &&\to (\nabla' y \vert b) \quad &&\text{in} \quad C^{0,\alpha}(\bar{\omega}),
		\end{alignat*}
		and
		$$
		\nabla_\eps^2 u_\eps \rightharpoonup \begin{pmatrix}
			(\nabla')^2 y & \nabla' b\\
			(\nabla' b)^T & d
		\end{pmatrix} \quad \text{in} \quad L^q(\Omega).
		$$
		Moreover, $\det(\nabla' y \vert b) \geq \mu$ in $\bar{\omega}$, with $\mu$ as in Lemma \ref{PositiveDet}, $K_\eps(u_\eps) \to K(y,b)$, $J_\eps(u_\eps) \to J(y,b)$, and $T(y,b) \leq \liminf_{\eps\to0} T_\eps(u_\eps)$. Finally, if $\Lambda(u_\eps) \geq c$ for every $\eps$, for some $c \geq c_0$, then $\lvert y(x') - y(z') \rvert \geq c \lvert x' - z' \rvert$ for all $x',z' \in \omega$.
	\end{lemma}
	\begin{proof}
		By Lemma \ref{W2qnormEps}, the families $\{u_\eps\}_\eps$, $\{\nabla_\eps u_\eps\}_\eps$, and $\{\nabla_\eps^2 u_\eps\}_\eps$ are bounded in $W^{2,q}(\Omega)$, $W^{1,q}(\Omega)$, and $L^q(\Omega)$, respectively. Lemma \ref{ScaledComp} and the compact embeddings of $W^{2,q}(\Omega)$ in $C^{1,\alpha}(\bar{\Omega})$ and of $W^{1,q}(\Omega)$ in $C^{0,\alpha}(\bar{\Omega})$ give the convergences. Passing to the limit on $P$ in the boundary conditions in $A_\eps$ gives $y = {\rm id}_0$, $\nabla' y = \nabla' {\rm id}_0$, and $b = (0,0,1)^T$ on $P$. Lemma \ref{PositiveDet} gives $\det \nabla_\eps u_\eps > \mu$ in $\bar{\Omega}$, and consequently $\det(\nabla' y \vert b) \geq \mu$ in $\bar{\omega}$. Since $\bar{u}_\eps \to y$ uniformly in $\bar{\omega}$, the condition $\Lambda(u_\eps) \geq c$ gives $\lvert y(x') - y(z') \rvert \geq c \lvert x' - z' \rvert$, and, with $c = c_0$, we obtain $(y,b) \in B$. The uniform convergence of $\nabla_\eps u_\eps$ and the lower bound on the determinants give $K_\eps(u_\eps) \to K(y,b)$ and $J_\eps(u_\eps) \to J(y,b)$. Finally, as $\lvert \nabla_\eps^2 u_\eps \rvert^2 \geq \lvert (\nabla')^2 u_\eps \rvert^2 + 2 \lvert \eps^{-1} \nabla' \partial_3 u_\eps \rvert^2$, the convexity of $(G,F) \mapsto (\lvert G \rvert^2 + 2 \lvert F \rvert^2)^{q/2}$ gives $T(y,b) \leq \liminf_{\eps\to0} T_\eps(u_\eps)$.
	\end{proof}
	
	We close this subsection with a compactness result for time-dependent families of scaled deformations. It replaces, in the limits as $\eps\to0$, a Korn inequality independent of $\eps$, which is not available by Remark \ref{RemKorn}.
	
	\begin{lemma}[Compactness of time-dependent scaled deformations]
		\label{TimeComp}
		Let $I = [t_0,t_1]$ and, for every $\eps \in (0,1)$, let $y_\eps \in C_w(I;W^{2,q}(\Omega)) \cap H^1(I;H^1(\Omega))$ be such that $y_\eps(t) \in A_\eps$ for every $t \in I$ and
		\begin{equation}
			\label{TCbound}
			\sup_{t\in I} E_\eps(y_\eps(t)) + \int_I \Bigl( R_\eps(y_\eps,\partial_t y_\eps) + \lVert \partial_t y_\eps \rVert^2 \Bigr) \; \d t \leq c,
		\end{equation}
		with $c$ independent of $\eps$. Set $C_\eps := (\nabla_\eps y_\eps)^T \nabla_\eps y_\eps$. Then, there exist $y \in C_w(I;W^{2,q}(\omega)) \cap H^1(I;L^2(\omega))$, $b \in C_w(I;W^{1,q}(\omega))$, and $d \in L^\infty(I;L^q(\Omega))$ such that $(y(t),b(t)) \in B$ for every $t \in I$, $\det(\nabla' y \vert b) \geq \mu$ in $I \times \bar{\omega}$ with $\mu$ as in Lemma \ref{PositiveDet}, the Cauchy-Green tensor $C := (\nabla' y\vert b)^T(\nabla' y\vert b)$ belongs to $H^1(I;L^2(\omega))$, and, up to a subsequence (not relabeled), the following convergences hold
		\begin{alignat}{4}
			&y_\eps &&\overset{*}{\rightharpoonup} y \quad &&\text{in} \quad L^\infty(I;W^{2,q}(\omega)),\label{conv311a}\\
			&\nabla_\eps y_\eps &&\overset{*}{\rightharpoonup} (\nabla' y \vert b) \quad &&\text{in} \quad L^\infty(I;W^{1,q}(\omega)),\label{conv311b}\\
			&\nabla_\eps^2 y_\eps &&\overset{*}{\rightharpoonup} \begin{pmatrix}
				(\nabla')^2 y & \nabla' b\\
				(\nabla' b)^T & d
			\end{pmatrix} \quad &&\text{in} \quad L^\infty(I;L^q(\Omega)),\label{conv311c}\\
			&\partial_t y_\eps &&\rightharpoonup \partial_t y \quad &&\text{in} \quad L^2(I;L^2(\omega)),\label{conv311d}\\
			&\partial_t C_\eps &&\rightharpoonup \partial_t C \quad &&\text{in} \quad L^2(I;L^2(\omega)),\label{conv311e}\\
			&y_\eps &&\to y \quad &&\text{in} \quad C^0(I;C^{1,\alpha}(\bar{\omega})),\label{conv311f}\\
			&\nabla_\eps y_\eps &&\to (\nabla' y \vert b) \quad &&\text{in} \quad C^0(I;C^{0,\alpha}(\bar{\omega})).\label{conv311g}
		\end{alignat}
	\end{lemma}
	\begin{proof}
		We split the proof into five steps. After collecting the bounds independent of $\eps$ (Step 1) and identifying the corresponding weak limits (Step 2), we prove, up to a subsequence, the strong convergence of $y_\eps$ and of the Cauchy-Green tensors $C_\eps$ by means of Theorem \ref{cor4Simon} (Step 3). The same argument does not apply to the scaled gradients $\nabla_\eps y_\eps$, since no bound on the time derivative of $\eps^{-1}\partial_3 y_\eps$ independent of $\eps$ is available. The strong convergence of $\nabla_\eps y_\eps$ follows instead by expressing $\eps^{-1}\partial_3 y_\eps$ in terms of $\nabla' y_\eps$ and $C_\eps$ (Step 4). Finally, we establish the regularity and the admissibility of the limits (Step 5).
		
			$\bullet$ {\it Step 1: Bounds.}

		By \eqref{TCbound} and Lemmas \ref{W2qnormEps} and \ref{PositiveDet}, we get
		$$
			\lVert y_\eps \rVert_{L^{\infty}(I;W^{2,q}(\Omega))} + \lVert \nabla_\eps y_\eps \rVert_{L^{\infty}(I;W^{1,q}(\Omega))} + \lVert \nabla_\eps^2 y_\eps \rVert_{L^{\infty}(I;L^q(\Omega))} + \lVert \partial_t y_\eps \rVert_{L^2(I;L^2(\Omega))} \leq c,
		$$
		and
		$$
		\det \nabla_\eps y_\eps(t,x) \geq \mu , \qquad \forall  (t,x) \in I \times \bar{\Omega}.
		$$
		By using again \eqref{TCbound} and Lemma \ref{W2qnormEps}, in view of the continuous embedding $W^{1,q}(\Omega) \hookrightarrow C^{0,1-3/q}(\bar{\Omega})$, we obtain
		$$
		\lvert \nabla_\eps y_\eps (t,x) \rvert \leq M , \qquad \forall (t,x) \in I \times \bar{\Omega},
		$$
		for some $M>0$ independent of $\eps$, and, by H\"older's inequality,
		$$
		\lVert C_\eps \rVert_{L^\infty(I;W^{1,q}(\Omega))}\leq c.
		$$
	 	Moreover, as $y_\eps \in H^1(I;H^1(\Omega))$, one has
		$$
		\partial_t C_\eps = (\nabla_\eps \partial_t y_\eps)^T \nabla_\eps y_\eps + (\nabla_\eps y_\eps)^T \nabla_\eps \partial_t y_\eps,
		$$
		and thus $\lVert \partial_t C_\eps \rVert^2 = R_\eps(y_\eps,\partial_t y_\eps)$. Therefore, by \eqref{TCbound}, one has 
		$$
		\lVert \partial_t C_\eps \rVert_{L^2(I;L^2(\Omega))} \leq c.
		$$
		We stress that no bound on $\partial_t \nabla_\eps y_\eps$ independent of $\eps$ is available (see Remark \ref{RemKorn}).

		$\bullet$ {\it Step 2: Weak limits.}
		
		The bounds of Step 1 give the convergences \eqref{conv311a}--\eqref{conv311d}, where the structure of the limits is identified as in the proof of Lemma \ref{ScaledComp}, as well as $C_\eps \overset{*}{\rightharpoonup} \tilde{C}$ in $L^\infty(I;W^{1,q}(\Omega))$ and $\partial_t C_\eps \rightharpoonup \partial_t \tilde{C}$ in $L^2(I;L^2(\Omega))$, for some $\tilde{C}$.
		
		$\bullet$ {\it Step 3: Strong convergences of $y_\eps$ and $C_\eps$.}
		
		Since $W^{2,q}(\Omega) \overset{c}{\hookrightarrow} C^{1,\alpha}(\bar{\Omega}) \hookrightarrow L^2(\Omega)$ and $W^{1,q}(\Omega) \overset{c}{\hookrightarrow} C^{0,\alpha}(\bar{\Omega}) \hookrightarrow L^2(\Omega)$, the bounds of Step 1 and Theorem \ref{cor4Simon} guarantee that, up to a subsequence, $y_\eps \to y$ in $C^0(I;C^{1,\alpha}(\bar{\Omega}))$ and $C_\eps \to \tilde{C}$ in $C^0(I;C^{0,\alpha}(\bar{\Omega}))$.

		$\bullet$ {\it Step 4: Strong convergence of $\nabla_\eps y_\eps$.}
	
		We write $\nabla_\eps y_\eps = (a_\eps \vert b_\eps)$, with $a_\eps = \nabla' y_\eps = (\partial_1 y_\eps \vert \partial_2 y_\eps) \in \mathbb{M}^{3\times2}$ and $b_\eps = \eps^{-1}\partial_3 y_\eps$. By Step 3, $a_\eps \to \nabla' y$ in $C^0(I\times\bar{\Omega})$. To show the convergence of $b_\eps$, we express $b_\eps$ in terms of $a_\eps$ and $C_\eps$. In block form, the Cauchy-Green tensor reads
		$$
		C_\eps = \begin{pmatrix}
			a_\eps^T a_\eps & a_\eps^T b_\eps \\
			b_\eps^T a_\eps & \lvert b_\eps \rvert^2
		\end{pmatrix}
		= \begin{pmatrix}
			G_\eps & g_\eps \\
			g_\eps^T & \lvert b_\eps \rvert^2
		\end{pmatrix},
		$$
		where
		$$
		G_\eps = \begin{pmatrix}
			\lvert \partial_1 y_\eps \rvert^2 & \partial_1 y_\eps \cdot \partial_2 y_\eps \\
			\partial_2 y_\eps \cdot \partial_1 y_\eps & \lvert \partial_2 y_\eps \rvert^2
		\end{pmatrix}
		\qquad \text{and} \qquad
		g_\eps = \begin{pmatrix}
			\partial_1 y_\eps \cdot b_\eps \\
			\partial_2 y_\eps \cdot b_\eps
		\end{pmatrix}.
		$$
		Let $n_\eps := \partial_1 y_\eps \times \partial_2 y_\eps$. As $n_\eps \cdot b_\eps = \det \nabla_\eps y_\eps \geq \mu$ and $\lvert b_\eps \rvert \leq M$, one has $\lvert n_\eps \rvert \geq \mu/M$, so that $\{\partial_1 y_\eps, \partial_2 y_\eps, n_\eps\}$ is a basis of $\R^3$ and there exist $\lambda_\eps = (\lambda_{\eps,1},\lambda_{\eps,2})^T$ and $\sigma_\eps$ such that
		$$
		b_\eps = \lambda_{\eps,1} \, \partial_1 y_\eps + \lambda_{\eps,2} \, \partial_2 y_\eps + \sigma_\eps n_\eps = a_\eps \lambda_\eps + \sigma_\eps n_\eps .
		$$
		To compute $\lambda_\eps$, we multiply this identity by $a_\eps^T$. Since $a_\eps^T b_\eps = g_\eps$, $a_\eps^T a_\eps = G_\eps$, and $a_\eps^T n_\eps = 0$, we obtain $g_\eps = G_\eps \lambda_\eps$. As
		$$
		\det G_\eps = \lvert \partial_1 y_\eps \rvert^2 \lvert \partial_2 y_\eps \rvert^2 - (\partial_1 y_\eps \cdot \partial_2 y_\eps)^2 = \lvert n_\eps \rvert^2 > 0,
		$$
		the matrix $G_\eps$ is invertible and $\lambda_\eps = G_\eps^{-1} g_\eps$. To compute $\sigma_\eps$, we take the scalar product of the same identity with $n_\eps$. As $n_\eps \cdot \partial_1 y_\eps = n_\eps \cdot \partial_2 y_\eps = 0$, we obtain $\sigma_\eps \lvert n_\eps \rvert^2 = n_\eps \cdot b_\eps = \det \nabla_\eps y_\eps$. Moreover, $\det C_\eps = (\det \nabla_\eps y_\eps)^2$ and, as $\det \nabla_\eps y_\eps > 0$, one has $\det \nabla_\eps y_\eps = (\det C_\eps)^{1/2}$. Therefore,
		$$
		b_\eps = a_\eps G_\eps^{-1} g_\eps + \frac{(\det C_\eps)^{1/2}}{\lvert n_\eps \rvert^2} \, n_\eps .
		$$
		
		We now pass to the limit in this identity. As Step 3 implies the convergence of $a_\eps$ and $C_\eps$ in $C^0(I \times \bar{\Omega})$, we get that $G_\eps$, $g_\eps$, $(\det C_\eps)^{1/2}$, and $n_\eps$ converge in $C^0(I \times \bar{\Omega})$ as well. Moreover, as $\det G_\eps = \lvert n_\eps \rvert^2 \geq \mu^2/M^2$ uniformly in $\eps$, $G_\eps^{-1} = (\det G_\eps)^{-1} \, {\rm cof}\, G_\eps$ and $\lvert n_\eps \rvert^{-2}$ converge in $C^0(I \times \bar{\Omega})$ as well. Therefore, $b_\eps$ converges in $C^0(I\times\bar{\Omega})$ and its limit coincides with the limit $b$ of Step 2, so that $\nabla_\eps y_\eps \to (\nabla' y \vert b)$ in $C^0(I\times\bar{\omega})$. In particular, $\tilde{C} = C$, so that $C \in H^1(I;L^2(\omega))$ and $\partial_t C_\eps \rightharpoonup \partial_t C$, and $\det(\nabla' y \vert b) \geq \mu$ in $I \times \bar{\omega}$. Finally, since $\nabla_\eps y_\eps$ is bounded in $L^\infty(I;W^{1,q}(\Omega))$, and thus in $L^\infty(I;C^{0,1-3/q}(\bar{\Omega}))$, the interpolation inequality $[w]_{C^{0,\alpha}} \leq 2 \lVert w \rVert_{C^0}^{1-\alpha q/(q-3)} [w]_{C^{0,1-3/q}}^{\alpha q/(q-3)}$ results in $\nabla_\eps y_\eps \to (\nabla' y \vert b)$ in $C^0(I;C^{0,\alpha}(\bar{\omega}))$.
		
		$\bullet$ {\it Step 5: Regularity and admissibility.}
		
		By \cite[Lemma II.5.9, p.~95]{boyer}, the bounds of Step 1 and the continuity of $y$ and $b$ with values in $C^{1,\alpha}(\bar{\omega})$ and $C^{0,\alpha}(\bar{\omega})$, respectively, give $y \in C_w(I;W^{2,q}(\omega))$ and $b \in C_w(I;W^{1,q}(\omega))$. Finally, for every $t \in I$, Lemma \ref{ClosAB} applies to $\{y_\eps(t)\}_\eps$, whose limit is $(y(t),b(t))$ by Steps 3 and 4, and gives $(y(t),b(t)) \in B$.
	\end{proof}
	
	\begin{remark}
		\label{RemTimeComp}
		The proof of Lemma \ref{TimeComp} applies to families of membrane deformations. Let $I=[t_0,t_1]$ and let $\{(y_n,b_n)\}_n$ be such that $y_n \in C_w(I;W^{2,q}(\omega)) \cap H^1(I;L^2(\omega))$, $b_n \in C_w(I;W^{1,q}(\omega))$, $(y_n(t),b_n(t)) \in B$ for every $t \in I$, $C_n := (\nabla' y_n \vert b_n)^T (\nabla' y_n \vert b_n) \in H^1(I;L^2(\omega))$, and
		$$
		\sup_{t\in I} E(y_n(t),b_n(t)) + \int_I \Bigl( \lVert \partial_t C_n \rVert^2 + \lVert \partial_t y_n \rVert^2 \Bigr) \; \d t \leq c,
		$$
		with $c$ independent of $n$. Then, there exist $y \in C_w(I;W^{2,q}(\omega)) \cap H^1(I;L^2(\omega))$ and $b \in C_w(I;W^{1,q}(\omega))$ such that $(y(t),b(t)) \in B$ for every $t \in I$, $\det(\nabla' y \vert b) \geq \mu$ in $I \times \bar{\omega}$ with $\mu$ as in Lemma \ref{MPositiveDet}, $C := (\nabla' y \vert b)^T(\nabla' y \vert b) \in H^1(I;L^2(\omega))$, and, up to a subsequence (not relabeled),
		\begin{alignat*}{6}
			&y_n &&\overset{*}{\rightharpoonup} y \quad &&\text{in} \quad L^\infty(I;W^{2,q}(\omega)), \qquad &&b_n &&\overset{*}{\rightharpoonup} b \quad &&\text{in} \quad L^\infty(I;W^{1,q}(\omega)),\\
			&\partial_t y_n &&\rightharpoonup \partial_t y \quad &&\text{in} \quad L^2(I;L^2(\omega)), \qquad &&\partial_t C_n &&\rightharpoonup \partial_t C \quad &&\text{in} \quad L^2(I;L^2(\omega)),\\
			&y_n &&\to y \quad &&\text{in} \quad C^0(I;C^{1,\alpha}(\bar{\omega})), \qquad &&b_n &&\to b \quad &&\text{in} \quad C^0(I;C^{0,\alpha}(\bar{\omega})).
		\end{alignat*}
		Indeed, Lemma \ref{MW2qnormEps} replaces Lemmas \ref{W2qnormEps} and \ref{PositiveDet} in Step 1, the arguments of Steps 3 and 4 apply with $(a_\eps \vert b_\eps)$ replaced by $(\nabla' y_n \vert b_n)$, and Lemma \ref{ClosA}.$(ii)$ replaces Lemma \ref{ClosAB} in Step 5.
	\end{remark}

	\subsection{Notions of solutions}
\label{NotSol}
We now define the notions of solutions to the problems introduced in Subsections \ref{Configuration}, \ref{Membrane}, and \ref{twoTSA}, following the hierarchy of the two-time-scale approximation scheme, from the incremental problem to the continuous-time problem.

\begin{definition}[Solutions to the incremental problem]
	\label{defSolIncr}
	We say that the families $\{y^{l,k} \}_{l=1, \; k = 0 }^{N,M}$ and $\{y^l\}_{l=1}^{N-1}$ are a solution of the incremental problem \eqref{incrP}--\eqref{incrIC2} if, for all $l$ and $k$, $y^{l,k} \in A_\eps$, $y^l\in  C^0([(l-1)h,lh];W^{2,q}(\Omega)) \cap H^1((l-1)h,lh;H^1(\Omega))$, $y^l(t) \in A_\eps$ for every $t \in [(l-1)h,lh]$, $\partial_t y^l \in L^2((l-1)h,lh;H_0^1(\Omega))$, and \eqref{incrP}--\eqref{incrIC2} hold.
\end{definition}

\begin{definition}[Solutions to the sequential problem]
	\label{defSolSeq}
	For $l=1,\dots,N$, let $y^l \in C^0([(l-1)h,lh];W^{2,q}(\Omega)) \cap H^1((l-1)h,lh;H^1(\Omega))$ be such that $y^l(t) \in A_\eps$ for every $t \in [(l-1)h,lh]$ and $\partial_t y^l \in L^2((l-1)h,lh;H_0^1(\Omega))$.
	We say that $\{ y^l \}_{l=1}^N$ is a weak solution of the sequential problem \eqref{sequential}--\eqref{initialSequentialI} if
	\begin{align*}
		&y^1(0) = y_{0\eps}, \quad v^0 = v_{0\eps},\\
		&y^l((l-1)h) = y^{l-1}((l-1)h),\quad v^{l-1}(\cdot)=\partial_t y^{l-1}(\cdot -h), \quad  \text{for every}\quad l = 2,\dots, N,
	\end{align*}
	and
	   \begin{align}
		\label{seqPweakEq}
		\int_{(l-1)h}^{lh} \biggl( \langle D E_\eps (y^l) , \Psi \rangle_{H_0^1} &+ \langle D_2 R_\eps (y^l,\partial_t y^l) , \Psi \rangle \nonumber\\
		&- (f_\eps, \Psi) + \frac{\rho}{h} (\partial_t y^l - v^{l-1} , \Psi) \biggr) \; \d t = 0,
	\end{align}
	for every $\Psi \in L^2((l-1)h,lh;H_0^1(\Omega))$ and all $l=1,\dots,N$.
\end{definition}

\begin{definition}[Solutions to the time-delayed problem]
	\label{defSoltimedelayedT}
	Let $y \in C^0([0,T];W^{2,q}(\Omega))\cap H^1(0,T;H^1(\Omega))$ be such that $y(t) \in A_\eps$ for every $t \in [0,T]$ and $\partial_t y \in L^2(0,T;H_0^1(\Omega))$.
	We term $y$ a weak solution to the problem \eqref{tDelay}--\eqref{initialVelocity} if $y(0) = y_{0\eps}$, $\partial_t y (t-h) = v_{0\eps}$ for almost every $t \in (0,h)$, and
	\begin{equation}
		\label{timeDweakEw}
		\int_0^T \biggl( \langle D E_\eps (y) , \Psi \rangle_{H_0^1} + \langle D_2 R_\eps (y,\partial_t y) , \Psi \rangle - (f_\eps, \Psi) + \frac{\rho}{h} (\partial_t y - \partial_t y (\cdot - h) , \Psi) \biggr) \; \d t = 0,
	\end{equation}
	for every $\Psi \in L^2(0,T;H_0^{1}(\Omega))$.
\end{definition}

\begin{definition}[Solutions to the continuous-time problem]
	\label{defSolT}	
	Let $y$ be such that $y \in C_w([0,T];W^{2,q}(\Omega)) \cap H^1(0,T;H^1(\Omega))$, $y \in C^0([0,T];C^{1,\alpha}(\bar{\Omega}))$, $y(t) \in A_\eps$ for every $t \in [0,T]$, and $\partial_t y \in L^2(0,T;H_0^1(\Omega))$.
	We call $y$ a weak solution to the problem \eqref{eq1S}--\eqref{eq11S} if $y(0) = y_{0\eps}$ and
	\begin{equation*}
		\int_0^T \biggl( \langle D E_\eps (y) , \Psi \rangle + \langle D_2 R_\eps (y,\partial_t y) , \Psi \rangle - (f_\eps, \Psi) - \rho (\partial_t y , \partial_t \Psi) \biggr) \; \d t - \rho (v_{0\eps},\Psi(0)) = 0,
	\end{equation*}
		for every $\Psi \in L^1(0,T;V) \cap H^1(0,T;H_0^1(\Omega))$ such that $\Psi(T) = 0$, with $V$ as in Subsection \ref{SVar}.
\end{definition}

\begin{definition}[Solutions to the membrane incremental problem]
	\label{MdefSolIncr}
	We say that the families $\{(y^{l,k},b^{l,k}) \}_{l=1, \; k = 0 }^{N,M}$ and $\{(y^l,b^l)\}_{l=1}^{N-1}$ are a solution of the incremental problem \eqref{MincrP}--\eqref{MincrIC2} if, for all $l$ and $k$, $(y^{l,k},b^{l,k}) \in B$, $y^l\in C_w([(l-1)h,lh];W^{2,q}(\omega)) \cap H^1((l-1)h,lh;L^2(\omega))$, $b^l \in C_w([(l-1)h,lh];W^{1,q}(\omega))$, $(y^l(t),b^l(t)) \in B$ for every $t \in [(l-1)h,lh]$, $(\nabla' y^l \vert b^l)^T(\nabla' y^l \vert b^l) \in H^1((l-1)h,lh;L^2(\omega))$, and \eqref{MincrP}--\eqref{MincrIC2} hold.
\end{definition}

\begin{definition}[Solutions to the membrane sequential problem]
	\label{MdefSolSeq}
	For $l=1,\dots,N$, let $y^l \in C_w([(l-1)h,lh];W^{2,q}(\omega)) \cap H^1((l-1)h,lh;L^2(\omega))$ and $b^l \in C_w([(l-1)h,lh];W^{1,q}(\omega))$ be such that $(y^l(t),b^l(t)) \in B$ for every $t \in [(l-1)h,lh]$ and $(\nabla' y^l \vert b^l)^T(\nabla' y^l \vert b^l) \in H^1((l-1)h,lh;L^2(\omega))$.
	We say that $\{ (y^l,b^l) \}_{l=1}^N$ is a weak solution of the sequential problem \eqref{Msequential}--\eqref{MinitialSequentialI} if
	\begin{align*}
		&y^1(0)=y_0, \quad b^1(0) = b_0, \quad v^0 = v_0, \\
		&y^l((l-1)h) = y^{l-1}((l-1)h), \quad b^l((l-1)h) = b^{l-1}((l-1)h),\quad  \text{for every}\quad l = 2,\dots, N, \\
		&v^{l-1}(\cdot) = \partial_t y^{l-1}(\cdot -h)\quad  \text{for every}\quad l = 2,\dots, N,
	\end{align*}
	and
	\begin{align}
		\label{MseqPweakEq}
		\int_{(l-1)h}^{lh} \biggl( \langle D_1 E (y^l,b^l) , \Psi_1 \rangle_{H_0^1} &+ \langle D_2 R (y^l,\partial_t y^l,b^l,\partial_t b^l) , \Psi_1 \rangle \nonumber\\&- (f_0, \Psi_1) + \frac{\rho}{h} (\partial_t y^l - v^{l-1} , \Psi_1) \biggr) \; \d t = 0,
	\end{align}
	for every $\Psi_1 \in L^2((l-1)h,lh;H_0^1(\omega))$ and all $l=1,\dots,N$, and
	\begin{align}
		\label{MseqPweakEq2}
		\int_{(l-1)h}^{lh} \biggl( ( D_2 E (y^l,b^l) , \Psi_2 ) + ( D_4 R (y^l,\partial_t y^l,b^l,\partial_t b^l) , \Psi_2 ) \biggr) \; \d t = 0,
	\end{align}
	for every $\Psi_2 \in L^2((l-1)h,lh;L^2(\omega))$ and all $l=1,\dots,N$.
\end{definition}	

\begin{definition}[Solutions to the membrane time-delayed problem]
	\label{MdefSoltimedelayedT}
	Let $b \in C_w([0,T];W^{1,q}(\omega))$ and $y \in C_w([0,T];W^{2,q}(\omega))\cap H^1(0,T;L^2(\omega))$ be such that $(y(t),b(t)) \in B$ for every $t \in [0,T]$ and $(\nabla' y \vert b)^T(\nabla' y \vert b) \in H^1(0,T;L^2(\omega))$.
	We term $(y,b)$ a weak solution to the problem \eqref{MtDelay}--\eqref{MinitialVelocity} if $y(0) = y_0$, $b(0)= b_0$, $\partial_t y (t-h) = v_0$ for almost every $t \in (0,h)$, and
	\begin{align}
		\label{MtimeDweakEw}
		\int_0^T \biggl( \langle D_1 E (y,b) , \Psi \rangle_{H_0^1} +& \langle D_2 R (y,\partial_t y,b,\partial_t b) , \Psi \rangle \nonumber\\-& (f_0, \Psi) + \frac{\rho}{h} (\partial_t y - \partial_t y (\cdot - h) , \Psi) \biggr) \; \d t = 0,
	\end{align}
	for every $\Psi \in L^2(0,T;H_0^{1}(\omega))$, and
	\begin{equation}
		\label{MtimeDweakEw2}
		\int_0^T \biggl( ( D_2 E (y,b) , \Psi ) + ( D_4 R (y,\partial_t y,b,\partial_t b) , \Psi ) \biggr) \; \d t = 0,
	\end{equation}
	for every $\Psi \in L^2(0,T;L^2(\omega))$. 
\end{definition}

The last notion of solution is formulated in terms of Young measures, for which we follow the notation of \cite{GVK}. Let $\mathcal{G}_0 := \mathbb{M}^{3\times2\times2}\times\mathbb{M}^{3\times2}$ and $\omega_T := (0,T)\times\omega$. Identifying the space $\mathcal{M}(\mathcal{G}_0)$ of signed Radon measures on $\mathcal{G}_0$ with the dual of the space $C_0(\mathcal{G}_0)$ of continuous functions vanishing at infinity, we denote by $L^\infty_{w*}(\omega_T;\mathcal{M}(\mathcal{G}_0)) = L^1(\omega_T;C_0(\mathcal{G}_0))^*$ the space of weakly-$*$ measurable and essentially bounded maps $\omega_T \ni (t,x') \mapsto \nu_{t,x'} \in \mathcal{M}(\mathcal{G}_0)$, namely, such that
	$$
	(t,x') \mapsto \langle \nu_{t,x'} , g \rangle := \int_{\mathcal{G}_0} g(\lambda) \; \d\nu_{t,x'}(\lambda)
	$$
	is measurable for every $g \in C_0(\mathcal{G}_0)$ and $\lVert \nu_{t,x'} \rVert_{\mathcal{M}(\mathcal{G}_0)}$ is essentially bounded. A Young measure is an element $\nu$ of this space such that $\nu_{t,x'}$ is a probability measure for a.e. $(t,x') \in \omega_T$. By \cite[Theorem 4.1]{Rindler} (see also \cite{Ball}), every bounded sequence $\{\Theta_n\}_n$ in $L^q(\omega_T;\mathcal{G}_0)$ has a subsequence (not relabeled) which generates a Young measure $\nu$, namely,
	$$
	g(\Theta_n) \rightharpoonup \langle \nu_{t,x'} , g \rangle \quad \text{in} \quad L^1(\omega_T;\mathbb{R}^m)
	$$
	for every continuous $g : \mathcal{G}_0 \to \mathbb{R}^m$ such that $\{g(\Theta_n)\}_n$ is equiintegrable. In particular, $\Theta_n \rightharpoonup \langle \nu_{t,x'} , \lambda \rangle$ in $L^q(\omega_T;\mathcal{G}_0)$ \cite[Lemma 4.12]{Rindler}, namely, the barycenter of $\nu$ is the weak limit of $\{\Theta_n\}_n$. Without equiintegrability, lower semicontinuity still holds \cite[Proposition 4.7]{Rindler}, namely, for every continuous $f : \mathcal{G}_0 \to [0,\infty)$ and every nonnegative $\varphi \in L^\infty(\omega_T)$,
	\begin{equation}
		\label{YMlsc}
		\int_0^T \int_\omega \varphi \, \langle \nu_{t,x'} , f \rangle \; \d x' \, \d t \leq \liminf_{n\to\infty} \int_0^T \int_\omega \varphi \, f(\Theta_n) \; \d x' \, \d t.
	\end{equation}
	If $\{\Theta_n\}_n$ is bounded in $L^\infty(0,T;L^q(\omega;\mathcal{G}_0))$, then \eqref{YMlsc} with $f(\lambda) = \lvert \lambda \rvert^q$ and $\varphi$ depending only on $t$ gives
	$$
	\esssup_{t\in(0,T)} \int_\omega \langle \nu_{t,x'} , \lvert \lambda \rvert^q \rangle \; \d x' \leq \sup_n \lVert \Theta_n \rVert_{L^\infty(0,T;L^q(\omega))}^q,
	$$
	and, as $L^{q'}(\omega_T)$ is dense in $L^1(0,T;L^{q'}(\omega))$, the convergence of $\{\Theta_n\}_n$ to the barycenter holds weakly-$*$ in $L^\infty(0,T;L^q(\omega;\mathcal{G}_0))$.
	
	\begin{definition}[Measure-valued solutions to the membrane continuous-time problem]
		\label{MdefSolMV}
		Let $\kappa_0 : \mathcal{G}_0 \to \mathbb{R}$ be given by $\kappa_0(G,F) := \frac{1}{q}\bigl(\lvert G \rvert^2 + 2 \lvert F \rvert^2\bigr)^{q/2}$, so that $T(u,\beta) = \int_\omega \kappa_0\bigl((\nabla')^2 u,\nabla'\beta\bigr) \; \d x'$ and
		$$
		D_1\kappa_0(G,F) = \bigl(\lvert G \rvert^2 + 2 \lvert F \rvert^2\bigr)^{(q-2)/2} G, \qquad D_2\kappa_0(G,F) = 2\bigl(\lvert G \rvert^2 + 2 \lvert F \rvert^2\bigr)^{(q-2)/2} F.
		$$
		Given $\mathfrak{E}_0 \geq E(y_0,b_0)$, a measure-valued solution to the problem \eqref{Meq1S}--\eqref{Meq11S} with initial energy $\mathfrak{E}_0$ consists of a pair $(y,b)$, with $b \in C_w([0,T];W^{1,q}(\omega))$ and $y \in C_w([0,T];W^{2,q}(\omega))\cap H^1(0,T;L^2(\omega))$, and of a Young measure $\nu$ such that $(y(t),b(t)) \in B$ for every $t \in [0,T]$, $(\nabla' y \vert b)^T(\nabla' y \vert b) \in H^1(0,T;L^2(\omega))$, $y(0)=y_0$, $b(0)=b_0$,
		\begin{gather*}
			\esssup_{t\in(0,T)} \int_\omega \langle \nu_{t,x'} , \lvert \lambda \rvert^q \rangle \; \d x' < \infty,\\
			\langle \nu_{t,x'} , \lambda \rangle = \bigl((\nabla')^2 y(t,x'),\nabla' b(t,x')\bigr) \quad \text{for a.e. } (t,x') \in \omega_T,
		\end{gather*}
		and
		\begin{align}
			\label{MmvEq1}
			\int_0^T \biggl( &\langle D_1 K (y,b) , \Psi \rangle + \langle D_1 J (y,b) , \Psi \rangle + \int_\omega \langle \nu_{t,x'} , D_1\kappa_0 \rangle : (\nabla')^2 \Psi \; \d x' \nonumber\\
			&+ \langle D_2 R (y,\partial_t y,b,\partial_t b) , \Psi \rangle - (f_0, \Psi) - \rho (\partial_t y , \partial_t \Psi) \biggr) \; \d t - \rho (v_0,\Psi(0)) = 0,
		\end{align}
		for every $\Psi \in L^1(0,T;V_0^1) \cap H^1(0,T;H_0^1(\omega))$ such that $\Psi(T) = 0$, with $V_0^1$ as in Subsection \ref{SVar},
		\begin{align}
			\label{MmvEq2}
			\int_0^T \biggl( ( D_2 K (y,b) , \Psi ) + ( D_2 J (y,b) , \Psi ) &+ \int_\omega \langle \nu_{t,x'} , D_2\kappa_0 \rangle : \nabla' \Psi \; \d x' \nonumber\\
			&+ ( D_4 R (y,\partial_t y,b,\partial_t b) , \Psi ) \biggr) \; \d t = 0,
		\end{align}
		for every $\Psi \in L^1(0,T;W_0^{1,q}(\omega)) \cap L^2(0,T;L^2(\omega))$, and, for a.e. $t \in (0,T)$,
		\begin{align}
			\label{MmvEI}
			K(y(t),b(t)) + J(y(t),b(t)) &+ \int_\omega \langle \nu_{t,x'} , \kappa_0 \rangle \; \d x' + \frac{\rho}{2} \lVert \partial_t y(t) \rVert^2 + 2 \int_0^t R(y,\partial_t y,b,\partial_t b) \; \d s \nonumber\\
			&\leq \mathfrak{E}_0 + \frac{\rho}{2} \lVert v_0 \rVert^2 + \int_0^t (f_0,\partial_t y) \; \d s.
		\end{align}
\end{definition}

	\subsection{Main results}
	\label{MainRes}
	We are now in a position to state our main results. Throughout, the final time $T$ satisfies $T \leq T_0$, where $T_0 > 0$ depends only on $\rho$, $q$, $\Omega$, $\mathcal{C}$, $c_1 - c_0$, $c_2$, $v_0$, and $f_0$, and is determined in Subsections \ref{S42} and \ref{S411}.
	
	The restriction $T \leq T_0$ guarantees that the incremental minimizers never reach the averaged invertibility constraint in $A_\eps$, so that they satisfy the corresponding Euler-Lagrange equations. The time $T_0$ is not the first time of contact of a given solution, but a lower bound for the first time at which the constraint may become active, uniform in $\eps$, $h$, and $\tau$. In \cite{BenesovaJEMS}, for a fixed problem, the existence of solutions is established up to the first time of contact, and hence on the whole time interval if no contact occurs. For the treatment of contact by means of measure-valued contact forces, we refer to \cite{CGK,CGK2}.

	\begin{theorem}[Convergence of the bulk two-time-scale approximation scheme]
		\label{Theorem1}
		Assume {\rm (A0)--(A3)}. Then:
		\begin{enumerate}
			\item[{\rm i)}] There exist $\{y_\eps^{l,k} \}_{l=1, \; k = 0 }^{N,M}$ and $\{y_\eps^l\}_{l=1}^{N-1}$ solving the minimization problem \eqref{incrP}--\eqref{incrIC2} in the sense of Definition \ref{defSolIncr}.
			\item[{\rm ii)}] The sequence $\{ y_\eps^l\}_{l=1}^N$ is a weak solution of the sequential problem \eqref{sequential}--\eqref{initialSequentialI} in the sense of Definition \ref{defSolSeq}.
			\item[{\rm iii)}] The global interpolant $y_\eps^{(h)}$, given by $y_\eps^{(h)}(t) := y_\eps^l(t)$ for $t \in [(l-1)h,lh]$ and $l=1,\dots,N$, is a weak solution of the time-delayed problem \eqref{tDelay}--\eqref{initialVelocity} in the sense of Definition \ref{defSoltimedelayedT}.
			\item[{\rm iv)}] For any fixed $\eps >0$, up to a subsequence (not relabeled) as $h \to 0$, $y_{\eps}^{(h)} \to y_\eps$ in $C^0([0,T];C^{1,\alpha}(\bar{\Omega}))$, where $y_\eps$ solves the problem \eqref{eq1S}--\eqref{eq11S} in the sense of Definition \ref{defSolT}.
		\end{enumerate}	
		
	\end{theorem}

	\begin{theorem}[Membrane limit of the incremental problem]
		\label{Theorem2}
		Assume {\rm (A0)--(A3)}. Let the families $\{y_\eps^{l,k} \}_{l=1, \; k = 0 }^{N,M}$ and $\{y_\eps^l\}_{l=1}^{N-1}$ be a solution of the incremental problem \eqref{incrP}--\eqref{incrIC2} in the sense of Definition \ref{defSolIncr}, and let $\{y_\eps^l \}_{l=1}^N$ be the weak solution of the sequential problem \eqref{sequential}--\eqref{initialSequentialI} provided by Theorem \ref{Theorem1}.$(ii)$. Then, there exist $\{(y^{l,k},b^{l,k})\}_{l=1,\;k=0}^{N,M} \subset W^{2,q}(\omega) \times W^{1,q}(\omega)$, $\{d^{l,k}\}_{l=1,\;k=1}^{N,M} \subset L^q(\Omega)$, and $\{(y^l,b^l)\}_{l=1}^{N}$ such that:
		\begin{enumerate}
			\item[{\rm i)}] For every $l=1,\dots,N$ and every $k=1,\dots,M$, the following convergences hold (up to non-relabeled subsequences)
			\begin{alignat*}{4}
				&y_\eps^{l,k} &&\rightharpoonup y^{l,k} \quad &&\text{in} \quad W^{2,q}(\omega),\\
				&\nabla_\eps y_\eps^{l,k} &&\rightharpoonup (\nabla' y^{l,k} \vert b^{l,k}) \quad &&\text{in} \quad W^{1,q}(\omega),\\
				&\nabla_\eps^2 y_\eps^{l,k} &&\rightharpoonup \begin{pmatrix}
					(\nabla')^2 y^{l,k} & \nabla' b^{l,k}\\
					(\nabla'b^{l,k})^T & d^{l,k}
				\end{pmatrix}\quad &&\text{in} \quad L^q(\Omega).
			\end{alignat*}
			 \item[{\rm ii)}] For every $l=1,\dots,N$, the following convergence holds (up to non-relabeled subsequences)
			$$
			\nabla_\eps y_\eps^l \to (\nabla' y^l \vert b^l) \quad \text{in} \quad C^0([(l-1)h,lh];C^{0,\alpha}(\bar{\omega})).
			$$
			\item[{\rm iii)}] The families $\{(y^{l,k},b^{l,k}) \}_{l=1, \; k = 0 }^{N,M}$ and $\{(y^l,b^l)\}_{l=1}^{N-1}$ are a solution of the incremental problem \eqref{MincrP}--\eqref{MincrIC2} in the sense of Definition \ref{MdefSolIncr}.	
		\end{enumerate}

	\end{theorem}
	
	\begin{theorem}[Membrane limit of the sequential problem]
		\label{Theorem3}
		Assume {\rm (A0)--(A3)}. Let $\{ y_\eps^l\}_{l=1}^N$ be the weak solution of the sequential problem \eqref{sequential}--\eqref{initialSequentialI} provided by Theorem \ref{Theorem1}.$(ii)$. Then, up to a subsequence (not relabeled) as $\eps \to 0$, $\nabla_\eps y_\eps^l \to (\nabla' y^l \vert b^l) $ in $C^0([(l-1)h,lh];C^{0,\alpha}(\bar{\omega}))$ for all $l=1,\dots,N$, where $\{(y^l,b^l)\}_{l=1}^N$ is a weak solution of the membrane sequential problem \eqref{Msequential}--\eqref{MinitialSequentialI} in the sense of Definition \ref{MdefSolSeq}.
	\end{theorem}

	\begin{theorem}[Membrane limit of the time-delayed problem]
		\label{Theorem5}
		Assume {\rm (A0)--(A3)}. Let $\{ y_\eps^{(h)}\}_h$ be the family of weak solutions of the time-delayed problem \eqref{tDelay}--\eqref{initialVelocity} provided by Theorem \ref{Theorem1}.$(iii)$. Then, up to a subsequence (not relabeled) as $\eps \to 0$, $\nabla_\eps y_\eps^{(h)} \to  (\nabla' y^{(h)} \vert b^{(h)}) $ in $C^0([0,T];C^{0,\alpha}(\bar{\omega}))$, where, for every $h$, $(y^{(h)},b^{(h)})$ is a weak solution of the membrane time-delayed problem \eqref{MtDelay}--\eqref{MinitialVelocity} in the sense of Definition \ref{MdefSoltimedelayedT}.
	\end{theorem}
	
	\begin{corollary}[Consistency of the membrane time-delayed and sequential problems]
		\label{CorConcat}
		Assume {\rm (A0)--(A3)}. Let $\{y_\eps^l\}_{l=1}^N$ be the weak solution of the sequential problem \eqref{sequential}--\eqref{initialSequentialI} provided by Theorem \ref{Theorem1}.$(ii)$, and let $y_\eps^{(h)}$ be its global interpolant. Then, the convergences in Theorems \ref{Theorem3} and \ref{Theorem5} hold along a common subsequence, along which $(y^{(h)}(t),b^{(h)}(t)) = (y^l(t),b^l(t))$ for every $t \in [(l-1)h,lh]$ and $l=1,\dots,N$. Namely, the weak solution of the membrane time-delayed problem provided by Theorem \ref{Theorem5} is the concatenation of the weak solution of the membrane sequential problem provided by Theorem \ref{Theorem3}.
	\end{corollary}

	\begin{theorem}[Measure-valued limit of the membrane time-delayed problem]
		\label{Theorem6}
		Assume {\rm (A0)--(A3)}. Let $\{(y^{(h)},b^{(h)})\}_h$ be the family of weak solutions of the membrane time-delayed problem \eqref{MtDelay}--\eqref{MinitialVelocity} provided by Theorem \ref{Theorem5}, and set $\mathfrak{E}_0 := \limsup_{\eps\to0} E_\eps(y_{0\eps})$. Then, up to a subsequence (not relabeled) as $h \to 0$, $y^{(h)} \to y$ in $C^0([0,T];C^{1,\alpha}(\bar{\omega}))$ and $b^{(h)} \to b$ in $C^0([0,T];C^{0,\alpha}(\bar{\omega}))$, and $\bigl((\nabla')^2 y^{(h)},\nabla' b^{(h)}\bigr)$ generates a Young measure $\nu$ such that $(y,b,\nu)$ is a measure-valued solution of the problem \eqref{Meq1S}--\eqref{Meq11S} with initial energy $\mathfrak{E}_0$ in the sense of Definition \ref{MdefSolMV}.
	\end{theorem}
	
	\begin{remark}
		\label{RemMain}
			We collect some remarks on the results stated above.
		\begin{enumerate}
			\item[{\rm i)}] In Theorem \ref{Theorem3}, the membrane sequential solutions are obtained as limits of the bulk ones, and not by letting $\tau\to0$ in the membrane incremental problem. Indeed, for $l\geq2$, the initial data of the membrane incremental problem in Theorem \ref{Theorem2} are given by the membrane sequential solution on the previous interval obtained in Theorem \ref{Theorem3} (see Subsection \ref{S53}). In the absence of a uniqueness result, membrane sequential solutions obtained by letting $\tau\to0$ at the membrane level might differ from these, and would not be related to the bulk ones.
			\item[{\rm ii)}] In Theorems \ref{Theorem2}, \ref{Theorem3}, and \ref{Theorem5}, the bulk sequential and time-delayed solutions are those provided by Theorem \ref{Theorem1}. Indeed, the limit as $\eps\to0$ relies on the estimates of Subsection \ref{S410}, which are independent of $\eps$ and are proved for these solutions only.
			\item[{\rm iii)}] Corollary \ref{CorConcat} does not rely on any uniqueness result for the membrane problems. On the one hand, by Definitions \ref{MdefSolSeq} and \ref{MdefSoltimedelayedT}, the concatenation of any weak solution of the membrane sequential problem \eqref{Msequential}--\eqref{MinitialSequentialI} is a weak solution of the membrane time-delayed problem \eqref{MtDelay}--\eqref{MinitialVelocity}, and, conversely, the restrictions of any weak solution of the latter to the intervals $[(l-1)h,lh]$ form a weak solution of the former. On the other hand, both limits are taken on the same bulk family, since $y_\eps^{(h)}$ coincides with $y_\eps^l$ on $[(l-1)h,lh]$. The lack of uniqueness matters only when comparing membrane solutions which are not generated by the same bulk family, as for the limit $\tau\to0$ at the membrane level discussed in item $(i)$.
			\item[{\rm iv)}] In the measure-valued solutions of Definition \ref{MdefSolMV}, only the second-order stress is described through the Young measure $\nu$, whereas the first-order terms, the dissipation, and the inertia are evaluated at $(y,b)$. The Young measure accounts for possible oscillations of the second gradients, which cannot be excluded, since the membrane dissipation controls only the time derivative of the Cauchy-Green tensor and thus provides no spatial compactness of the velocities (see Remark \ref{RemKorn}). Young measure solutions of evolution problems obtained by time discretization go back to \cite{KinderlehrerPedregal}, and were constructed in \cite{Demoulini96,Demoulini97,Demoulini00} for forward-backward parabolic equations, systems of mixed type, and viscoelasticity, respectively, and in \cite{DST01} for polyconvex elastodynamics (see also \cite{CarstensenRieger,Rieger} for nonconvex elastodynamics). In \cite{Demoulini97,Demoulini00,DST01}, the solutions are defined through the equations, the barycenter condition, and the initial data. Following \cite{DST}, where the solutions of \cite{DST01} are shown to be dissipative measure-valued solutions, Definition \ref{MdefSolMV} also includes the energy inequality \eqref{MmvEI}. In Theorem \ref{Theorem6}, the initial energy satisfies $\mathfrak{E}_0 \geq E(y_0,b_0)$ by Subsection \ref{Membrane}, with equality if $E_\eps(y_{0\eps}) \to E(y_0,b_0)$. If $\nu_{t,x'}$ is the Dirac mass at $\Theta(t,x') := \bigl((\nabla')^2 y(t,x'),\nabla' b(t,x')\bigr)$ for a.e. $(t,x')$, then the stress terms in \eqref{MmvEq1} and \eqref{MmvEq2} reduce to $\langle D_1T(y,b),\Psi\rangle$ and $\langle D_2T(y,b),\Psi\rangle$, respectively, so that $(y,b)$ satisfies \eqref{Meq1S}--\eqref{Meq11S} in the weak sense. This is the case if $\int_0^T T(y^{(h)},b^{(h)}) \, \d t \to \int_0^T T(y,b) \, \d t$ along the subsequence of Theorem \ref{Theorem6}. Indeed, Jensen's inequality and \eqref{YMlsc} give
				$$
				\int_0^T T(y,b) \; \d t \leq \int_0^T \int_\omega \langle \nu_{t,x'} , \kappa_0 \rangle \; \d x' \, \d t \leq \liminf_{h \to 0} \int_0^T T(y^{(h)},b^{(h)}) \; \d t,
				$$
				so that $\langle \nu_{t,x'} , \kappa_0 \rangle = \kappa_0(\Theta(t,x'))$ for a.e. $(t,x')$, and the strict convexity of $\kappa_0$ forces $\nu_{t,x'} = \delta_{\Theta(t,x')}$. Similarly, if $\partial_t y^{(h)} \to \partial_t y$ in $L^2(0,T;L^2(\omega))$, the monotonicity argument of Subsection \ref{S412} applies at the membrane level and gives $\Theta^{(h)} \to \Theta$ in $L^q((0,T')\times\omega)$ for every $T' \in (0,T)$, so that again $\nu_{t,x'} = \delta_{\Theta(t,x')}$.
		\end{enumerate}
	\end{remark}

	\section{Proof of Theorem \ref{Theorem1}}
	\label{sec:4}
	This section is devoted to the proof of Theorem \ref{Theorem1}, which proceeds by iteration and induction, following the four items $(i)$--$(iv)$ in its statement. We start by proving the existence of a sequence $\{ y_\eps^{1,k} \}_{k=0}^M \in (A_\eps)^{M+1}$ which solves the minimization problem \eqref{incrP}--\eqref{incrIC2} for the time interval $[0,h]$, corresponding to the first iteration, i.e. $l = 1$ (item $(i)$). Notice that in the first time step the initial conditions read as $y_\eps^{1,0}=y_{0\eps}$ and $v_\eps^{0,k-1}=v_{0\eps}$ for every $k=1,\dots,M$. The sequence $\{y_\eps^{1,k} \}$ is then used to construct the interpolants $y_\eps^{1,(\tau)}$, $\underline{y_\eps}^{1,(\tau)}$, and $\tilde{y}_\eps^{1,(\tau)}$, whose limit, namely $y_\eps^1$, solves the sequential problem \eqref{sequential}--\eqref{initialSequentialI} on the time interval $[0,h]$. The function $y_\eps^1$ is then used to construct the initial data for the next interval $[h,2h]$ according to \eqref{incrIC1}--\eqref{incrIC2}, and the procedure is iterated over the outer time intervals. At each iteration, the incremental minimizers have to stay away from the averaged invertibility constraint in $A_\eps$, so that they satisfy the Euler-Lagrange equations. For $T \leq T_0$, this is proved by induction on $l$, by means of estimates on the sequential solutions which are independent of $h$. Items $(i)$ and $(ii)$ follow, with a solution $\{y_\eps^l\}_{l=1}^N$ of the sequential problem on the whole interval $[0,T]$.
	
	Item $(iii)$ is then a direct consequence of $(ii)$: since the sequential equations \eqref{sequential} hold on each subinterval $[(l-1)h,lh]$ with the matching conditions \eqref{initialSequential}, the global interpolant $y_\eps^{(h)}(t):=y_\eps^l(t)$ for $t\in[(l-1)h,lh]$, $l=1,\dots,N$, solves the time-delayed problem \eqref{tDelay}--\eqref{initialVelocity} in the sense of Definition \ref{defSoltimedelayedT}.
	
	Finally, item $(iv)$ is obtained by letting $h\to0$. The coercivity estimates of Lemmas \ref{W2qnormEps}, \ref{PositiveDet}, and \ref{H1normEps}, applied uniformly in $h$, guarantee compactness for the family $\{y_\eps^{(h)}\}_h$. Passing to the limit in the weak formulation \eqref{timeDweakEw} of the time-delayed problem then shows that any limit point $y_\eps$ is a weak solution of the continuous-time problem \eqref{eq1S}--\eqref{eq11S} in the sense of Definition \ref{defSolT}, which completes the proof of Theorem \ref{Theorem1}.

	\subsection{Existence of solutions to the incremental minimization problem}
	\label{S41}
	Fix $k \in \{1,\dots,M\}$ and assume $\{ y_{\eps}^{1,j}\}_{j=0}^{k-1} \in (A_{\eps})^{k}$ to be given. We prove that the minimization problem
	\begin{align*}
		\min \biggl\{ &E_{\eps}(y) + \tau R_{\eps} \biggl( y_{\eps}^{1,k-1} , \frac{y - y_{\eps}^{1,k-1}}{\tau}  \biggr) - \tau \biggl( f_\eps^{k-1} , \frac{y-y_{\eps}^{1,k-1}}{\tau} \biggr) \nonumber\\&+ \frac{\tau \rho}{2h} \biggl\lVert \frac{y - y_{\eps}^{1,k-1}}{\tau} - v_{\eps}^{0,k-1} 	\biggr\rVert^2 	 \biggr\},
	\end{align*}
	admits a solution $y \in A_{\eps}$. The assertion follows by the Direct Method. Let $\{ y^n\}_n$ be an infimizing sequence, namely, $y^n \in A_{\eps}$ with $W_{\eps}^{1,k-1} (y^n) \to \inf_{A_{\eps}} W_{\eps}^{1,k-1}$. Without loss of generality, we assume $\sup_n W_{\eps}^{1,k-1} (y^n) < \infty$. Set $z^n := (y^n - y_\eps^{1,k-1})/\tau$. The lemmas of Subsection \ref{Aux} require a bound on the energy $E_\eps(y^n)$, whereas the bound on $W_\eps^{1,k-1}(y^n)$ involves also the load term, which has no sign. Writing $z^n = (z^n - v_\eps^{0,k-1}) + v_\eps^{0,k-1}$ in the load term and using Young's inequality, we obtain
		$$
		\tau \bigl\lvert ( f_\eps^{k-1} , z^n ) \bigr\rvert \leq \frac{\tau\rho}{4h} \lVert z^n - v_\eps^{0,k-1} \rVert^2 + \frac{\tau h}{\rho} \lVert f_\eps^{k-1} \rVert^2 + \tau \lVert f_\eps^{k-1} \rVert \, \lVert v_\eps^{0,k-1} \rVert.
		$$
		The first term on the right-hand side is absorbed by the inertial term, and the last two are independent of $n$. Since $R_\eps \geq 0$, it follows that $\sup_n E_\eps(y^n) < \infty$. By Lemma \ref{W2qnormEps}, $y^n$ is bounded in $W^{2,q}(\Omega)$, so that, up to a subsequence (not relabeled), $y^n \to y$ in $C^0(\bar{\Omega})$ for some $y$. Hence, Lemma \ref{ClosA}.$(i)$ gives $y \in A_\eps$, $y^n \to y$ in $C^{1,\alpha}(\bar{\Omega})$, the convergence of $K_\eps(y^n)$ and $J_\eps(y^n)$, and the lower semicontinuity of $T_\eps$. The remaining terms of $W_\eps^{1,k-1}$ converge by the convergence in $C^{1,\alpha}(\bar{\Omega})$. Therefore,
			$$
			W_{\eps}^{1,k-1}(y) \leq \liminf_{n \to \infty} W_{\eps}^{1,k-1}(y^n) = \inf_{A_{\eps}} W_{\eps}^{1,k-1},
			$$
			namely, $y$ is a minimizer.
	
	\subsection{Discrete Euler-Lagrange equations and corresponding estimates}
	\label{S42}
	For every $k = 1,\dots,M$, let $y_\eps^{1,k} \in A_\eps$ be a minimizer of $W_\eps^{1,k-1}$, provided by Subsection \ref{S41}. We first derive the estimates which follow from the minimality of $y_\eps^{1,k}$, and then we compute the first variation of $W_\eps^{1,k-1}$ at $y_\eps^{1,k}$.

	Since $y_\eps^{1,k}$ minimizes $W_\eps^{1,k-1}$, one gets
	$
	W_\eps^{1,k-1}(y_\eps^{1,k}) \le W_\eps^{1,k-1}(y_\eps^{1,k-1}),
	$
	giving, for all $k = 1,\dots,M$,
	\begin{align}
		E_\eps(y_\eps^{1,k}) +& \tau R_\eps \biggl( y_\eps^{1,k-1}, \frac{y_\eps^{1,k} - y_\eps^{1,k-1}}{\tau} \biggr) - \tau \biggl( f_\eps^{k-1}, \frac{y_\eps^{1,k} - y_\eps^{1,k-1}}{\tau} \biggr) \nonumber\\&+ \frac{\tau \rho}{2 h} \biggl\rVert \frac{y_\eps^{1,k} - y_\eps^{1,k-1}}{\tau} - v_\eps^{0,k-1} \biggr\rVert^2
		\leq E_\eps(y_\eps^{1,k-1}) + \frac{\tau \rho}{2 h} \lVert v_\eps^{0,k-1} \rVert^2. \label{zeroethEstimate}
	\end{align}
	Summing \eqref{zeroethEstimate} over $k=1, \dots, n$ and using Young's inequality for the load and inertial terms, for any $n=1, \dots, M$, we obtain the following estimate
	\begin{align}
		E_{\eps}(y_{\eps}^{1,n}) + \sum_{k=1}^n \tau \biggl[ R_{\eps} \biggl(y_{\eps}^{1,k-1}, &\frac{y_{\eps}^{1,k} - y_{\eps}^{1,k-1}}{\tau}\biggr) + \frac{\rho}{8h} \biggl\lVert \frac{y_{\eps}^{1,k} - y_{\eps}^{1,k-1}}{\tau} \biggr\rVert^2 \biggr] \nonumber \\
		&\leq E_{\eps}(y_{0\eps}) + \rho \lVert v_{0\eps} \rVert^2 + \frac{2h}{\rho} \int_0^{n\tau} \lVert f_{\eps}(t) \rVert^2 \; \d t.\label{firstEstimate}
	\end{align}
	Therefore, as the right-hand side of \eqref{firstEstimate} is uniformly bounded due to Assumptions (A1)--(A3), it follows that
	\begin{equation*}
		E_\eps(y_\eps^{1,k}) \leq c, \qquad \forall k = 0,\dots,M.
	\end{equation*}
	Thus, in view of Lemma \ref{H1normEps}, inequality \eqref{firstEstimate} implies
	\begin{equation}
		\label{firstEstimateC}
		\lVert y_\eps^{1,k} - y_\eps^{1,k-1} \rVert_{H^1(\Omega)} \leq c \sqrt{\tau}, \qquad \forall k = 1,\dots,M.
	\end{equation}

	Estimate \eqref{firstEstimate} and Lemma \ref{InvStab} show that the incremental minimizers remain close to the initial datum $y_{0\eps}$, in the sense of the averaged invertibility condition. Indeed, we have the following lemma.

	\begin{lemma}[Stability of the averaged invertibility condition on the first outer time interval]
			\label{LocNC1}
			The following statements hold.
			\begin{enumerate}
				\item[{\rm i)}] For every $\kappa>0$, there exists a constant $C_\Lambda(\kappa)>0$, depending only on $\kappa$, $\rho$, $\Omega$, $q$, and $\mathcal{C}$, and nondecreasing in $\kappa$, such that, if
				$$
				E_\eps(y_{0\eps}) + \rho \lVert v_{0\eps} \rVert^2 + \frac{2h}{\rho} \int_0^h \lVert f_\eps \rVert^2 \; \d t \leq \kappa,
				$$
				then $E_\eps(y_\eps^{1,k}) \leq \kappa$ and $\Lambda(y_\eps^{1,k}) \geq \Lambda(y_{0\eps}) - C_\Lambda(\kappa) h^{2\gamma_0}$ for every $k=0,\dots,M$.
				\item[{\rm ii)}] There exists $T_0>0$, depending only on $\rho$, $q$, $\Omega$, $\mathcal{C}$, $c_1 - c_0$, $c_2$, $v_0$, and $f_0$, such that, if $T \leq T_0$, then
				\begin{equation}
					\label{NC1}
					\Lambda(y_\eps^{1,k}) \geq c_3 := \frac{c_0 + c_1}{2} > c_0, \qquad \forall k = 0,\dots,M.
				\end{equation}
			\end{enumerate}
		\end{lemma}
		\begin{proof}
			$(i)$ As $R_\eps \geq 0$, estimate \eqref{firstEstimate} gives
			$$
			E_\eps(y_\eps^{1,k}) + \frac{\rho}{8h} \sum_{j=1}^{k} \tau \biggl\lVert \frac{y_\eps^{1,j} - y_\eps^{1,j-1}}{\tau} \biggr\rVert^2 \leq \kappa, \qquad \forall k = 1,\dots,M.
			$$
			In particular, $E_\eps(y_\eps^{1,k}) \leq \kappa$ for every $k=0,\dots,M$, where, for $k=0$, this follows directly from the assumption, as $y_\eps^{1,0} = y_{0\eps}$. Moreover, as $k\tau \leq h$, the Cauchy-Schwarz inequality gives
			$$
			\lVert y_\eps^{1,k} - y_{0\eps} \rVert \leq \sum_{j=1}^{k} \tau \biggl\lVert \frac{y_\eps^{1,j} - y_\eps^{1,j-1}}{\tau} \biggr\rVert \leq (k\tau)^{1/2} \biggl( \frac{8 h \kappa}{\rho} \biggr)^{1/2} \leq \biggl( \frac{8\kappa}{\rho} \biggr)^{1/2} h.
			$$
			By Lemma \ref{W2qnormEps}, one has $\lVert y_\eps^{1,k} \rVert_{W^{2,q}(\Omega)} \leq M_\kappa$ for every $k=0,\dots,M$, and in particular $\lVert y_{0\eps} \rVert_{W^{2,q}(\Omega)} \leq M_\kappa$. Hence, Lemma \ref{InvStab}.$(ii)$, applied with $2M_\kappa$ in place of $M$, gives the claim with $C_\Lambda(\kappa) := C_{2M_\kappa} (8\kappa/\rho)^{\gamma_0}$, which is nondecreasing in $\kappa$.
			
			$(ii)$ By Subsection \ref{SubIn}, one has $E_\eps(y_{0\eps}) \leq c_2 \lvert \omega \rvert$, $\lVert v_{0\eps} \rVert \leq \lVert v_0 \rVert_{L^2(\omega)} + 1$, and $\lVert f_\eps(t) \rVert^2 \leq \lvert \omega \rvert \bigl( \lVert f_0 \rVert_{C^0([0,T]\times\bar{\omega})} + 1 \bigr)^2$ for every $t \in [0,T]$. Hence, as $h \leq T$, the assumption of $(i)$ holds with
			$$
			\kappa = \kappa_1 := c_2 \lvert \omega \rvert + \rho \bigl( \lVert v_0 \rVert_{L^2(\omega)} + 1 \bigr)^2 + \frac{2 T^2}{\rho} \lvert \omega \rvert \bigl( \lVert f_0 \rVert_{C^0([0,T]\times\bar{\omega})} + 1 \bigr)^2,
			$$
			which is nondecreasing in $T$. Since $\Lambda(y_{0\eps}) \geq c_1$ by Assumption (A1) and $h \leq T$, item $(i)$ delivers $\Lambda(y_\eps^{1,k}) \geq c_1 - C_\Lambda(\kappa_1) T^{2\gamma_0}$ for every $k$. As $\kappa_1$ is nondecreasing in $T$ and $C_\Lambda(\kappa)$ is nondecreasing in $\kappa$, the function $T \mapsto C_\Lambda(\kappa_1) T^{2\gamma_0}$ is nondecreasing and tends to zero as $T \to 0$. Hence, there exists $T_0 > 0$ such that $C_\Lambda(\kappa_1) T^{2\gamma_0} \leq (c_1 - c_0)/2$ for every $T \leq T_0$, which gives \eqref{NC1}.
		\end{proof}
		
		The value of $T_0$ is further reduced in Subsection \ref{S411}, in order to extend \eqref{NC1} to all outer time intervals.
	
	We now compute the first variation of $W_\eps^{1,k-1}$ at $y_\eps^{1,k}$. Fix $\phi \in V$. By \eqref{NC1} and Lemma \ref{InvStab}.$(i)$, one has $\Lambda(y_\eps^{1,k} + s \phi) \geq c_0$ whenever $C_\omega \lvert s \rvert \lVert \phi \rVert_{W^{1,\infty}(\Omega)} \leq c_3 - c_0$. Moreover, as $E_\eps(y_\eps^{1,k}) \leq c$, Lemma \ref{PositiveDet} guarantees that $\det \nabla_\eps y_\eps^{1,k}$ is bounded away from zero in $\bar{\Omega}$, so that $\det \nabla_\eps (y_\eps^{1,k} + s\phi) > 0$ in $\bar{\Omega}$ for $\lvert s \rvert$ small enough, whereas the boundary conditions in $A_\eps$ are preserved as $\phi \in V$. Hence, $y_\eps^{1,k} + s \phi \in A_\eps$ for $\lvert s \rvert$ small enough. Since the differentiable function $s \mapsto W_\eps^{1,k-1} (y_\eps^{1,k} + s \phi)$ attains its minimum at $s=0$, we obtain
	$$
	\frac{\d}{\d s} W_\eps^{1,k-1} (y_\eps^{1,k} + s \phi) \biggr\rvert_{s=0} = 0,
	$$
	namely,
	\begin{align}
		\label{N1}
		0=&\, \langle DE_\eps(y_\eps^{1,k}), \phi \rangle + \biggl\langle D_2 R_\eps\biggl( y_\eps^{1,k-1}, \frac{y_\eps^{1,k} - y_\eps^{1,k-1}}{\tau} \biggr), \phi \biggr\rangle
		- (f_\eps^{k-1}, \phi) \nonumber \\ &\,+ \frac{\rho}{h} \biggl( \frac{y_\eps^{1,k} - y_\eps^{1,k-1}}{\tau} - v_\eps^{0,k-1}, \phi \biggr).
	\end{align}
	
	\subsection{Interpolants and their estimates}
	\label{S43}
For any vector $(\eta^0,\dots,\eta^n)$ and a uniform partition of time step $\tau$, we denote by $\eta^{(\tau)}$, $\underline{\eta}^{(\tau)}$, and $\tilde{\eta}^{(\tau)}$ its backward constant, forward constant, and piecewise affine interpolants, namely,
\begin{align*}
	&\eta^{(\tau)}(t) = \eta^k, \qquad \underline{\eta}^{(\tau)}(t) = \eta^{k-1}, \qquad \tilde{\eta}^{(\tau)}(t) := \biggl( \frac{t}{\tau} - (k-1) \biggr) \eta^k + \biggl( k - \frac{t}{\tau} \biggr) \eta^{k-1}\\
	&\text{for } t \in ((k-1)\tau, k\tau], \; k = 1,\dots,n, \qquad \text{and } \eta^{(\tau)}(0) = \underline{\eta}^{(\tau)}(0) = \tilde{\eta}^{(\tau)}(0) = \eta^{0}.
\end{align*}
In particular, $\partial_t \tilde{\eta}^{(\tau)}(t) = (\eta^k - \eta^{k-1})/\tau$ for $t \in ((k-1)\tau, k\tau)$. Given the vector of incremental solutions $(y_\eps^{1,0},\dots,y_\eps^{1,M}) \in (A_\eps)^{M+1}$ to the problem \eqref{incrP}--\eqref{incrIC2}, we denote by $y_\eps^{1,(\tau)}$, $\underline{y_\eps}^{1,(\tau)}$, and $\tilde{y}_\eps^{1,(\tau)}$ its backward constant, forward constant, and piecewise affine interpolants. Note that, for all $t \in [0,h]$, the constant interpolants are admissible deformations, namely, $y_\eps^{1,(\tau)}(t),\underline{y_\eps}^{1,(\tau)}(t) \in A_\eps$. Similarly, we introduce the forward constant interpolants of the data, namely $\underline{f_\eps}^{(\tau)}(t) = f_\eps^{k-1}$ and $\underline{v_\eps}^{0,(\tau)}(t) = v_\eps^{0,k-1}$ for $t \in ((k-1)\tau,k\tau]$ with $k = 1, \dots , M$. Observe that, given the definition of $f_\eps$ and $v_\eps^0$, one has $\underline{f_\eps}^{(\tau)} \to f_\eps  $ in $L^{\infty}(0,h;C^0(\bar{\Omega}))$ and $\underline{v_\eps}^{0,(\tau)} = v_{0\eps}$.
	Then, owing to the Euler-Lagrange equation \eqref{N1}, the interpolants $y_\eps^{1,(\tau)}$, $\underline{y_\eps}^{1,(\tau)}$, and $\tilde{y}_\eps^{1,(\tau)}$ satisfy
	\begin{align}
		\label{PDE0hInt}
		\langle &DE_\eps(y_\eps^{1,(\tau)}(t)), \phi \rangle + \langle D_2 R_\eps( \underline{y_\eps}^{1,(\tau)}(t), \partial_t \tilde{y}_\eps^{1,(\tau)}(t)), \phi \rangle \nonumber\\
		&- (\underline{f_\eps}^{(\tau)}(t), \phi) + \frac{\rho}{h} ( \partial_t \tilde{y}_\eps^{1,(\tau)}(t)  - \underline{v_\eps}^{0,(\tau)}(t), \phi ) = 0, \quad \text{for a.e. }t \in (0,h).
	\end{align}
	
	We now provide a priori estimates on the interpolants.
	Estimate \eqref{firstEstimate} delivers
	\begin{align}
		&\sup_{t \in [0, h]} E_{\eps}(y_{\varepsilon}^{1,(\tau)}(t)) \leq E_{\eps}(y_{0\varepsilon}) + \rho \lVert v_{0\eps} \rVert^2 + \frac{2h}{\rho} \int_0^{h} \lVert f_{\eps}(s) \rVert^2 \; \d s,\label{secondEstimatea}
	\end{align}
	and
	\begin{align}
		&\int_0^h \biggl[ R_{\eps} \biggl( \underline{y_{\eps}}^{1,(\tau)}(t), \partial_t \tilde{y}_{\eps}^{1,(\tau)}(t) \biggr) + \frac{\rho}{8h} \biggl\lVert \partial_t \tilde{y}_{\eps}^{1,(\tau)}(t) \biggr\rVert^2 \biggr] \; \d t \nonumber\\
		&\qquad\leq E_{\eps}(y_{0\varepsilon}) + \rho \lVert v_{0\eps} \rVert^2 + \frac{2h}{\rho} \int_0^{h} \lVert f_{\eps}(s) \rVert^2 \; \d s.\label{secondEstimateb}
	\end{align}
	Observe that the right-hand side of \eqref{firstEstimate} is uniformly bounded due to Assumptions (A1)--(A3). Hence, in view of Lemma \ref{W2qnormEps}, we obtain
	\begin{equation}
		\label{thirdEstimate}
		\sup_{t \in [0, h]} \lVert y_{\eps}^{1,(\tau)}(t) \rVert_{W^{2,q}(\Omega)} \leq c,
	\end{equation}
	and, since $y_\eps^{1,0} = y_{0\eps}$ is independent of $\tau$, it follows that $\sup_{t \in [0,h]}E_\eps(\underline{y_\eps}^{1,(\tau)}(t))$ is uniformly bounded as well. Hence, by using again Lemma \ref{W2qnormEps}, we have
	\begin{equation}
		\label{fourthEstimate}
		\sup_{t\in [0,h]} \lVert \underline{y_{\eps}}^{1,(\tau)}(t) \rVert_{W^{2,q}(\Omega)} \leq c.
	\end{equation}
	Consequently, Lemma \ref{H1normEps} applies and from \eqref{secondEstimateb} one gets
	$$
	\int_0^h \lVert \nabla_\eps \partial_t \tilde{y}_{\eps}^{1,(\tau)} \rVert^2 \; \d t \leq c.
	$$
	Here, we also used the fact that $\partial_t \tilde{y}_{\eps}^{1,(\tau)}(t) \rvert_{P} = 0$. Indeed, $\partial_t \tilde{y}_\eps^{1,(\tau)}(t) = (y_\eps^{1,k} - y_\eps^{1,k-1})/\tau$, where $y_\eps^{1,k}$ and $y_\eps^{1,k-1}$ take the same values on the boundary $P$. The last estimate, combined with \eqref{secondEstimateb}, implies
	\begin{equation}
		\label{fifthEstimate}
		\lVert \nabla_\eps \tilde{y}_{\eps}^{1,(\tau)} \rVert_{H^{1}(0,h;L^{2}(\Omega))} +\lVert \tilde{y}_{\eps}^{1,(\tau)} \rVert_{H^{1}(0,h;H^{1}(\Omega))} \leq c.
	\end{equation}
	Observe that, as $\tilde{y}_\eps^{1,(\tau)}(0) = y_{0\eps}$ by definition, the bounds of $\tilde{y}_\eps^{1,(\tau)}$ in $L^2(0,h;H^1(\Omega))$ and of $\nabla_\eps \tilde{y}_\eps^{1,(\tau)}$ in $L^2(0,h;L^2(\Omega))$ are guaranteed by using the Poincar\'e inequality in time.
	Notice also that estimates \eqref{thirdEstimate} and \eqref{fourthEstimate} imply
	\begin{equation}
		\label{sixthEstimate}
		\sup_{t \in [0, h]}  \lVert \tilde{y}_{\eps}^{1,(\tau)}(t) \rVert_{W^{2,q}(\Omega)}  \leq c.
	\end{equation}
	Moreover, for $t \in ((k-1)\tau,k\tau]$ one has $y_\eps^{1,(\tau)}(t) - \tilde{y}_\eps^{1,(\tau)}(t) = \frac{k\tau - t}{\tau}(y_\eps^{1,k} - y_\eps^{1,k-1})$ and $y_\eps^{1,(\tau)}(t) - \underline{y_\eps}^{1,(\tau)}(t) = y_\eps^{1,k} - y_\eps^{1,k-1}$, so that estimate \eqref{firstEstimateC} gives
		\begin{equation}
			\label{seventhEstimate}
			\sup_{t \in [0,h]} \Bigl( \lVert y_\eps^{1,(\tau)}(t) - \tilde{y}_\eps^{1,(\tau)}(t) \rVert_{H^1(\Omega)} + \lVert y_\eps^{1,(\tau)}(t) - \underline{y_\eps}^{1,(\tau)}(t) \rVert_{H^1(\Omega)} \Bigr) \leq c \sqrt{\tau}.
	\end{equation}

	\subsection{Convergences of the interpolants}
	\label{S44}
Relying on estimates \eqref{thirdEstimate}--\eqref{seventhEstimate}, we prove, up to non-relabeled subsequences, the following convergences
\begin{alignat}{4}
	y_\eps^{1,(\tau)},\underline{y_\eps}^{1,(\tau)},\tilde{y}_\eps^{1,(\tau)} &\overset{*}{\rightharpoonup} y_\eps^1  \quad &&\text{in} \quad L^{\infty}(0,h;W^{2,q}(\Omega)),\label{c2}\\
	\tilde{y}_\eps^{1,(\tau)} &\rightharpoonup y_\eps^1  \quad &&\text{in} \quad H^1(0,h;H^1(\Omega)),\label{c3}\\
	\tilde{y}_\eps^{1,(\tau)} &\to y_\eps^1 \quad &&\text{in} \quad C^0([0,h];C^{1,\alpha}(\bar{\Omega})),\label{c4}\\
	y_\eps^{1,(\tau)},\underline{y_\eps}^{1,(\tau)} &\to y_\eps^1 \quad &&\text{in} \quad L^{\infty}(0,h;C^{1,\alpha}(\bar{\Omega})),\label{c5}
\end{alignat}
where $y_\eps^1 \in C^{0,1/2}([0,h];H^1(\Omega)) \cap C_w([0,h];W^{2,q}(\Omega)) \cap C^0([0,h];C^{1,\alpha}(\bar{\Omega}))$. Moreover, we show that the limit function $y_\eps^1(t) \in A_\eps$ for every $t \in [0,h]$, $y_\eps^1(0) = y_{0\eps}$, and $\partial_t y_\eps^1 \in L^2(0,h;H_0^1(\Omega))$.

$\bullet$ {\it Step 1: Identification of the limits and convergences \eqref{c2} and \eqref{c3}.}

Considering $\tau \to 0$ in \eqref{seventhEstimate},  we have that the three sequences $y_\eps^{1,(\tau)}$, $\underline{y_\eps}^{1,(\tau)}$, and $\tilde{y}_\eps^{1,(\tau)}$ have the same limit $y_\eps^1$. Convergences \eqref{c2} and \eqref{c3} follow then from estimates \eqref{thirdEstimate}--\eqref{sixthEstimate}.

$\bullet$ {\it Step 2: Convergence \eqref{c4} and regularity of $y_\eps^1$.}

By \eqref{fifthEstimate} and \eqref{sixthEstimate}, $\tilde{y}_\eps^{1,(\tau)}$ is bounded in $L^\infty(0,h;W^{2,q}(\Omega))$ and $\partial_t \tilde{y}_\eps^{1,(\tau)}$ is bounded in $L^2(0,h;H^1(\Omega))$. As $W^{2,q}(\Omega) \overset{c}{\hookrightarrow} C^{1,\alpha}(\bar{\Omega}) \hookrightarrow H^1(\Omega)$, Theorem \ref{cor4Simon} gives \eqref{c4}. Moreover, \eqref{c3} gives $y_\eps^1 \in H^1(0,h;H^1(\Omega)) \subset C^{0,1/2}([0,h];H^1(\Omega))$, and, as $y_\eps^1$ is bounded in $L^\infty(0,h;W^{2,q}(\Omega))$ by \eqref{c2} and continuous with values in $C^{1,\alpha}(\bar{\Omega})$, \cite[Lemma II.5.9, p.~95]{boyer} gives $y_\eps^1 \in C_w([0,h];W^{2,q}(\Omega))$.

	$\bullet$ {\it Step 3: Convergence \eqref{c5}.}

By \eqref{seventhEstimate} and \eqref{c4}, $y_\eps^{1,(\tau)}$ and $\underline{y_\eps}^{1,(\tau)}$ converge to $y_\eps^1$ in $L^\infty(0,h;H^1(\Omega))$. As they are bounded in $L^\infty(0,h;W^{2,q}(\Omega))$ by \eqref{thirdEstimate} and \eqref{fourthEstimate}, Simon's interpolation lemma (Theorem \ref{lemma9Simon}) gives \eqref{c5}.

$\bullet$ {\it Step 4: $y_\eps^1(t) \in A_\eps$ for every $t \in [0,h]$, $y_\eps^1(0) = y_{0\eps}$, and $\partial_t y_\eps^1 \in L^2(0,h;H_0^1(\Omega))$.}
		
	By \eqref{c5}, $y_\eps^{1,(\tau)}(t) \to y_\eps^1(t)$ in $C^{1,\alpha}(\bar{\Omega})$ for a.e. $t \in (0,h)$. For such $t$, since $y_\eps^{1,(\tau)}(t) \in A_\eps$ and its energy is bounded by \eqref{secondEstimatea}, Lemma \ref{ClosA}.$(i)$ gives $y_\eps^1(t) \in A_\eps$ and $E_\eps(y_\eps^1(t)) \leq c$. Any other $t \in [0,h]$ is the limit of a sequence $\{t_k\}_k$ of such times, and $y_\eps^1(t_k) \to y_\eps^1(t)$ in $C^{1,\alpha}(\bar{\Omega})$ by \eqref{c4}, so that Lemma \ref{ClosA}.$(i)$ applies again. Hence, $y_\eps^1(t) \in A_\eps$ for every $t \in [0,h]$,
		\begin{equation}
			\label{1b}
			\sup_{t \in [0,h]} E_\eps(y_\eps^1(t)) \leq c,
		\end{equation}
	and, by the same lemma,	$\det \nabla_\eps y_\eps^1 \geq \mu$ in $[0,h] \times \bar{\Omega}$.
	
	Concerning the initial condition, we know that $\tilde{y}_\eps^{1,(\tau)}(0) = y_{0\eps}$ and that, due to the strong convergence \eqref{c4}, one has $\tilde{y}_\eps^{1,(\tau)}(0) \to y_\eps^1(0)$ in $C^{1,\alpha}(\bar{\Omega})$, hence $y_\eps^1(0,x) = y_{0\eps}(x)$ for every $x \in \bar{\Omega}$.
	
	Finally, as $y_\eps^1(t) = {\rm id}_\eps$ on $P$ for every $t \in [0,h]$ and the trace operator is linear and continuous from $H^1(\Omega)$ to $L^2(P)$, the trace of $\partial_t y_\eps^1 \in L^2(0,h;H^1(\Omega))$ on $P$ vanishes, namely, $\partial_t y_\eps^1 \in L^2(0,h;H_0^1(\Omega))$.

	\subsection{Improved integrability}
	\label{S45}
	From equation \eqref{PDE0hInt},  it follows that, for every $\psi \in L^2(0,h;\mathbb{R})$ and every $\phi \in V$, the interpolants $y_\eps^{1,(\tau)}$, $\underline{y_\eps}^{1,(\tau)}$, and $\tilde{y}_\eps^{1,(\tau)}$ satisfy
	\begin{align}
		\label{eq0h1}
		\int_0^h \biggl( \langle DE_\eps(y_\eps^{1,(\tau)}(t)), &\phi \rangle + \langle D_2 R_\eps( \underline{y_\eps}^{1,(\tau)}(t), \partial_t \tilde{y}_\eps^{1,(\tau)}(t)), \phi \rangle \nonumber \\
		&- (\underline{f_\eps}^{(\tau)}(t), \phi) + \frac{\rho}{h} ( \partial_t \tilde{y}_\eps^{1,(\tau)}  - \underline{v_\eps}^{0,(\tau)}(t), \phi ) \biggr) \psi(t) \; \d t = 0.
	\end{align}
	As $L^2(0,h) \otimes V $ is dense in $L^2(0,h;V)$, \eqref{eq0h1} holds with $\phi\,\psi$ replaced by any $\phi \in L^2(0,h;V)$.
	In view of definition \eqref{defE} of the energy $E_\eps$, we define the auxiliary function
	$$
	\mathcal{D}_\eps^{1,(\tau)} := DJ_\eps(y_\eps^{1,(\tau)}) + DK_\eps(y_\eps^{1,(\tau)}) + D_2 R_\eps( \underline{y_\eps}^{1,(\tau)}, \partial_t \tilde{y}_\eps^{1,(\tau)}) - \underline{f_\eps}^{(\tau)} + \frac{\rho}{h} ( \partial_t \tilde{y}_\eps^{1,(\tau)}  - \underline{v_\eps}^{0,(\tau)}),
	$$
	which is bounded in $L^2(0,h;H^{-1}(\Omega))$ by \eqref{thirdEstimate}--\eqref{fifthEstimate}, uniformly in $\tau$. Hence, Lemma \ref{LemC1}, applied with $X = V$, $Z = H_0^1(\Omega)$, $\xi = DT_\eps(y_\eps^{1,(\tau)})$, and $\zeta = \mathcal{D}_\eps^{1,(\tau)}$, extends \eqref{eq0h1} to every $\phi \in L^2(0,h;H_0^1(\Omega))$, namely,
	\begin{align}
		\label{eq0h2}
		\int_0^h \biggl(& \langle DT_\eps(y_\eps^{1,(\tau)}), \phi \rangle_{H_0^1} + \langle DJ_\eps(y_\eps^{1,(\tau)}), \phi \rangle  + \langle DK_\eps(y_\eps^{1,(\tau)}), \phi \rangle \nonumber\\
		&+ \langle D_2 R_\eps( \underline{y_\eps}^{1,(\tau)}, \partial_t \tilde{y}_\eps^{1,(\tau)}), \phi \rangle 
		- (\underline{f_\eps}^{(\tau)}, \phi) + \frac{\rho}{h} ( \partial_t \tilde{y}_\eps^{1,(\tau)}  - \underline{v_\eps}^{0,(\tau)}, \phi ) \biggr) \; \d t = 0.
	\end{align}
	Here, the notation $\langle \cdot , \cdot \rangle_{H_0^1}$ is retained explicit only in the first term to emphasize the improved integrability established above.
	
	Moreover, Lemma \ref{LemC1} gives $DT_\eps(y_\eps^{1,(\tau)}) = -\mathcal{D}_\eps^{1,(\tau)}$ in $L^2(0,h;H^{-1}(\Omega))$, so that
	\begin{equation}
		\label{estDT0h}
		\int_0^h \lVert DT_\eps (y_\eps^{1,(\tau)}) \rVert_{H^{-1}}^2 \; \d t \leq c,
	\end{equation}
	with $c$ independent of $\tau$.
	
	\subsection{Passage to the limit as $\tau \to 0$}
	\label{S46b}
	We address the limit $\tau \to 0$ in equation \eqref{eq0h2} and we prove that the limit function $y_\eps^1$ established in Subsection \ref{S44} solves, for every $\phi \in L^2(0,h;H_0^1(\Omega))$, the following equation
	\begin{align}
		\int_0^h \biggl(& \langle DT_\eps(y_\eps^1), \phi \rangle_{H_0^1} + \langle DJ_\eps(y_\eps^1), \phi \rangle  + \langle DK_\eps(y_\eps^1), \phi \rangle \nonumber \\
		&+ \langle D_2 R_\eps( y_\eps^1, \partial_t y_\eps^1), \phi \rangle 
		- (f_\eps, \phi) + \frac{\rho}{h} ( \partial_t y_\eps^1  - v_\eps^0, \phi ) \biggr) \; \d t = 0. \label{PDE0h}
	\end{align}
	
	$\bullet$	 {\it Step 1: Convergence of the linear terms.}
	
	Given the definitions of $f_\eps$ and $v_{0\eps}$, one has $\underline{f_\eps}^{(\tau)} \to f_\eps$ in $L^{\infty}(0,h;C^0(\bar{\Omega}))$ and $\underline{v_\eps}^{0,(\tau)} = v_{0\eps}$. Hence, using also convergence \eqref{c3}, we get
	\begin{align*}
		\int_0^h \biggl( - (\underline{f_\eps}^{(\tau)}, \phi) + \frac{\rho}{h} ( \partial_t \tilde{y}_\eps^{1,(\tau)}  - \underline{v_\eps}^{0,(\tau)}, \phi ) \biggr) \; \d t\to 	\int_0^h \biggl( - (f_\eps, \phi) + \frac{\rho}{h} ( \partial_t y_\eps^1  - v_\eps^0, \phi ) \biggr) \; \d t.
	\end{align*}

	$\bullet$ {\it Step 2: Convergence of the dissipation term.}
	
	In view of convergences \eqref{c3} and \eqref{c5}, we deduce
	\begin{equation}
		\label{conv0h4}
		\int_0^h \langle D_2 R_\eps( \underline{y_\eps}^{1,(\tau)}, \partial_t \tilde{y}_\eps^{1,(\tau)}), \phi \rangle \; \d t \to \int_0^h \langle D_2 R_\eps( y_\eps^1, \partial_t y_\eps^1), \phi \rangle \; \d t.
	\end{equation}
	Indeed, by virtue of \eqref{defDR}, one has
	\begin{align*}
		&\int_0^h \langle D_2 R_\eps( \underline{y_\eps}^{1,(\tau)}, \partial_t \tilde{y}_\eps^{1,(\tau)}), \phi \rangle \; \d t \\&= \int_0^h \int_\Omega 2 \biggl( (\nabla_\eps \partial_t \tilde{y}_\eps^{1,(\tau)})^T \nabla_\eps \underline{y_\eps}^{1,(\tau)} + (\nabla_\eps \underline{y_\eps}^{1,(\tau)} )^T \nabla_\eps \partial_t \tilde{y}_\eps^{1,(\tau)}	\biggr)\\&\qquad\qquad	: \biggl( \nabla_\eps \phi^T \nabla_\eps \underline{y_\eps}^{1,(\tau)} + (\nabla_\eps \underline{y_\eps}^{1,(\tau)})^T \nabla_\eps \phi \biggr) \; \d x \; \d t.
	\end{align*}
	Then, convergence \eqref{conv0h4} follows by observing that the first factor of the integrand converges weakly in $L^2(0,h;L^2(\Omega))$, while the second one converges strongly in $L^{2}(0,h;L^2(\Omega))$.
	
	$\bullet$ {\it Step 3: Convergence of the elastic energy terms.}
	
	Owing to the strong convergence \eqref{c5} and recalling \eqref{defDE}, we deduce
	\begin{equation}
		\label{conv0h1}
		\int_0^h \langle DK_\eps(y_\eps^{1,(\tau)}), \phi \rangle \; \d t \to \int_0^h \langle DK_\eps(y_\eps^1), \phi \rangle \; \d t.
	\end{equation}
	
	Next, we consider the term involving $D J_\eps$. Thanks to estimate \eqref{secondEstimatea} and Lemma \ref{PositiveDet}, we get that $\det \nabla_\eps y_\eps^{1,(\tau)}(t,x) > \mu$ for every $(t,x) \in [0,h] \times \bar{\Omega}$ with $\mu$ independent of the deformation.
	Observe that the strong convergence \eqref{c5} implies the pointwise convergence $\nabla_\eps y_\eps^{1,(\tau)}(t,x) \to \nabla_\eps y_\eps^1(t,x)$ for almost every $t \in (0,h)$ and for all $x \in \bar{\Omega}$.
	Then, as the map $F \mapsto {\rm cof}\, F/(\det F)^{a+1}$ is locally Lipschitz continuous on the set of matrices with positive determinant, we get
	$$
	\frac{{\rm cof}\,\nabla_\eps y_\eps^{1,(\tau)}(t,x)}{(\det \nabla_\eps y_\eps^{1,(\tau)}(t,x))^{a+1}} \to \frac{{\rm cof}\,\nabla_\eps y_\eps^1(t,x)}{(\det \nabla_\eps y_\eps^1(t,x))^{a+1}}, \qquad \text{for a.e. } t \in (0,h), \qquad \forall x \in \bar{\Omega},
	$$
	as well as
	$$
	\biggl\lVert \frac{{\rm cof}\,\nabla_\eps y_\eps^{1,(\tau)}}{(\det \nabla_\eps y_\eps^{1,(\tau)})^{a+1}} \biggr\rVert_{L^{\infty}(0,h;C^{0,\alpha}(\bar{\Omega}))}\leq c.
	$$
	We are now in a position to apply the Vitali convergence theorem and obtain
	\begin{equation*}
		\frac{{\rm cof}\,\nabla_\eps y_\eps^{1,(\tau)}}{(\det \nabla_\eps y_\eps^{1,(\tau)})^{a+1}} \to \frac{{\rm cof}\,\nabla_\eps y_\eps^1}{(\det \nabla_\eps y_\eps^1)^{a+1}}\quad \text{in} \quad L^p((0,h)\times \Omega), \qquad \forall p < \infty.
	\end{equation*}
	This, in turn, guarantees
	\begin{equation}
		\label{conv0h2}
		\int_0^h \langle DJ_\eps (y_\eps^{1,(\tau)}) , \phi \rangle \; \d t \to \int_0^h \langle D J_\eps (y_\eps^1) , \phi \rangle \; \d t.
	\end{equation}
	
	Finally, we consider the term involving $D T_\eps$.
	First, estimate \eqref{estDT0h} implies that, up to a non-relabeled subsequence, the following convergence holds
	\begin{equation}
		\label{conv0h3}
		DT_\eps(y_\eps^{1,(\tau)}) \rightharpoonup \mathcal{P}_\eps^1 \quad \text{in} \quad L^2(0,h;H^{-1}(\Omega)).
	\end{equation}
	Also, from estimate \eqref{thirdEstimate} we infer, along a further subsequence,
	\begin{equation}
		\label{conv0h5}
		DT_\eps(y_\eps^{1,(\tau)}) \rightharpoonup \mathcal{P}_\eps^1 \quad \text{in} \quad L^{q'}(0,h;(W^{2,q}(\Omega))').
	\end{equation}
	Observe that, by the uniqueness of the limit, the two limits coincide in $(L^q(0,h;V))'$ and we denote both by $\mathcal{P}_\eps^1$.
	Then, we establish the identity $\mathcal{P}_\eps^1 = DT_\eps(y_\eps^1)$ by exploiting the monotonicity of $DT_\eps$ on $W^{2,q}(\Omega)$. In particular, for every $v\in L^q(0,h;W^{2,q}(\Omega))$, one has
	\begin{equation}
		\label{conv0h7}
		0 \leq \int_0^h \langle DT_\eps (y_\eps^{1,(\tau)}) - DT_\eps(v) , y_\eps^{1,(\tau)} - v \rangle \; \d t.
	\end{equation}
	We show that
	\begin{equation}
		\label{conv0h6}
		\int_0^h \langle DT_\eps(y_\eps^{1,(\tau)}) , y_\eps^{1,(\tau)} \rangle \; \d t \to 		\int_0^h \langle \mathcal{P}_\eps^1 , y_\eps^1 \rangle \; \d t.
	\end{equation}
	By adding and subtracting the boundary condition for $y_\eps^{1,(\tau)}$, namely ${\rm id}_\eps$, we get
	$$
	\int_0^h \langle DT_\eps(y_\eps^{1,(\tau)}) , y_\eps^{1,(\tau)} \rangle \; \d t = \int_0^h \langle DT_\eps(y_\eps^{1,(\tau)}) , y_\eps^{1,(\tau)} - {\rm id}_\eps \rangle \; \d t + \int_0^h \langle DT_\eps(y_\eps^{1,(\tau)}) , {\rm id}_\eps \rangle \; \d t.
	$$
	Thus, as $y_\eps^{1,(\tau)}(t) - {\rm id}_\eps \in V $ for every $t$, the identity $\int_0^h \langle DT_\eps(y_\eps^{1,(\tau)}) , y_\eps^{1,(\tau)} - {\rm id}_\eps \rangle \; \d t = \int_0^h \langle DT_\eps(y_\eps^{1,(\tau)}) , y_\eps^{1,(\tau)} - {\rm id}_\eps \rangle_{H_0^1} \; \d t$ holds, and convergences \eqref{c5}, \eqref{conv0h3}, and \eqref{conv0h5} allow us to pass to the limit and obtain
	\begin{align*}
		\int_0^h \langle DT_\eps(y_\eps^{1,(\tau)}) , y_\eps^{1,(\tau)} \rangle \; \d t \to \int_0^h \langle \mathcal{P}_\eps^1 , y_\eps^1 - {\rm id}_\eps \rangle_{H_0^1} \; \d t + \int_0^h \langle \mathcal{P}_\eps^1 , {\rm id}_\eps \rangle \; \d t.
	\end{align*}
	As $y_\eps^1 - {\rm id}_\eps \in L^q(0,h;V)$ by Step 4 of Subsection \ref{S44}, we recover \eqref{conv0h6}. Also, using convergences \eqref{c2} and \eqref{conv0h5} for the remaining terms in \eqref{conv0h7}, we get
	\begin{align}
		&\int_0^h \biggl( \langle DT_\eps (y_\eps^{1,(\tau)}), v \rangle + \langle DT_\eps(v) , y_\eps^{1,(\tau)} - v \rangle \biggr) \; \d t \nonumber\\
		&\qquad \to \int_0^h \biggl( \langle \mathcal{P}_\eps^1 , v \rangle + \langle DT_\eps (v) , y_\eps^1 - v \rangle \biggr) \; \d t.\label{N3}
	\end{align}
	Using \eqref{conv0h6} and \eqref{N3} in \eqref{conv0h7}, it follows that
	\begin{equation*}
		0 \leq \int_0^h \langle \mathcal{P}_\eps^1 - DT_\eps(v) , y_\eps^1 - v \rangle \; \d t, \qquad \forall v \in L^q(0,h;W^{2,q}(\Omega)).
	\end{equation*}
	By the continuity of the map $X \mapsto \lvert X \rvert^{q-2} X$, the bound $\bigl\lvert \lvert X \rvert^{q-2} X \bigr\rvert \leq \lvert X \rvert^{q-1}$, and the dominated convergence theorem, $DT_\eps$ is hemicontinuous on $L^q(0,h;W^{2,q}(\Omega))$. Hence, Lemma \ref{LemMinty} gives $\mathcal{P}_\eps^1 = DT_\eps (y_\eps^1)$ in $L^{q'}(0,h;(W^{2,q}(\Omega))')$, and $\mathcal{P}_\eps^1 \in L^2(0,h;H^{-1}(\Omega))$ by \eqref{conv0h3}. Therefore, Step 1 and \eqref{conv0h4}--\eqref{conv0h3} allow us to pass to the limit in \eqref{eq0h2}, which gives \eqref{PDE0h}.

	\subsection{Chain rule for $T_\eps$}
	\label{S46}
	Since the limit function $y_\eps^1$ belongs to $C_w([0,h];W^{2,q}(\Omega))$, the term $T_\eps \circ y_\eps^1$ is well-defined pointwise in every $t \in [0,h]$. Then, exploiting the convexity of $T_\eps$, one has
	$$
	T_\eps(y_\eps^1(t + \lambda)) - T_\eps (y_\eps^1(t) ) \geq \langle D T_\eps (y_\eps^1(t)) , y_\eps^1(t + \lambda) - y_\eps^1(t) \rangle , \qquad \forall t \in [0,h-\lambda].
	$$
	Integrating this inequality over $(t_1,t_2)$ and dividing by $\lambda$, we obtain for all $0 \leq t_1 < t_2 \leq h-\lambda$
	\begin{equation}
		\label{ineqCR}
		\int_{t_1}^{t_2} \frac{T_\eps(y_\eps^1(t + \lambda)) - T_\eps (y_\eps^1(t) )}{\lambda} \; \d t\geq \int_{t_1}^{t_2} \biggl\langle D T_\eps (y_\eps^1(t)) , \frac{y_\eps^1(t + \lambda) - y_\eps^1(t)}{\lambda} \biggr\rangle \; \d t.
	\end{equation}
	As $y_\eps^1(t) \in A_\eps$ for every $t$, the boundary conditions in $A_\eps$ give $y_\eps^1(\cdot + \lambda) - y_\eps^1 \in L^2(0,h-\lambda;V)$. Also, recalling that $D T_\eps (y_\eps^1) \in L^2(0,h;H^{-1}(\Omega))$, the right-hand side of inequality \eqref{ineqCR} can then be written by means of the $H^{-1},H_0^1$ duality pairing as
	$$
	\int_{t_1}^{t_2} \biggl\langle D T_\eps (y_\eps^1(t)) , \frac{y_\eps^1(t + \lambda) - y_\eps^1(t)}{\lambda} \biggr\rangle_{H_0^1}\; \d t.
	$$
	Owing to the strong convergence $(y_\eps^1(t + \lambda) - y_\eps^1(t))/\lambda \to \partial_t y_\eps^1(t)$ in $L^2(0,h;H_0^1(\Omega))$, this integral converges to
	$$
	\int_{t_1}^{t_2} \langle D T_\eps (y_\eps^1(t)) , \partial_t y_\eps^1(t) \rangle_{H_0^1}\; \d t.
	$$
	Concerning the left-hand side of \eqref{ineqCR}, the Lebesgue differentiation theorem guarantees the convergence to $T_\eps(y_\eps^1(t_2)) - T_\eps(y_\eps^1(t_1))$ whenever $t_1$ and $t_2$ are Lebesgue points of the bounded function $t \mapsto T_\eps(y_\eps^1(t))$. For such $t_1 < t_2$, passing to the limit as $\lambda \to 0$ in \eqref{ineqCR} gives us
	\begin{equation}
		\label{ineqCR2}
		T_\eps(y_\eps^1(t_2)) - T_\eps(y_\eps^1(t_1)) \geq \int_{t_1}^{t_2} \langle D T_\eps (y_\eps^1(t)) , \partial_t y_\eps^1(t) \rangle_{H_0^1} \; \d t.
	\end{equation}
	Similarly, by considering
	$$
	T_\eps(y_\eps^1(t - \lambda)) - T_\eps (y_\eps^1(t) ) \geq \langle D T_\eps (y_\eps^1(t)) , y_\eps^1(t - \lambda) - y_\eps^1(t) \rangle , \qquad \forall t \in [\lambda,h],
	$$ 
	an analogous argument leads to inequality \eqref{ineqCR2}, for the same $t_1 < t_2$, with opposite sign. Combining both bounds, $T_\eps \circ y_\eps^1$ coincides a.e. in $(0,h)$ with an absolutely continuous function $F$ such that $F(t_2) - F(t_1) = \int_{t_1}^{t_2} \langle D T_\eps (y_\eps^1) , \partial_t y_\eps^1 \rangle_{H_0^1} \, \d t$ for all $t_1 < t_2$. We show that $T_\eps(y_\eps^1(t_0)) = F(t_0)$ for every $t_0 \in [0,h]$. On the one hand, approaching $t_0$ along Lebesgue points, the weak continuity of $y_\eps^1$ in $W^{2,q}(\Omega)$ and the weak lower semicontinuity of $T_\eps$ give $T_\eps(y_\eps^1(t_0)) \leq F(t_0)$. On the other hand, as $y_\eps^1(t_0) - y_\eps^1(s) \in V$, the convexity of $T_\eps$ and Step 2 of Subsection \ref{S44} give, for a.e. $s \in (0,h)$,
		$$
		T_\eps(y_\eps^1(t_0)) \geq T_\eps(y_\eps^1(s)) - c \, \lVert DT_\eps(y_\eps^1(s)) \rVert_{H^{-1}} \lvert t_0 - s \rvert^{1/2}.
		$$
		Since $\lVert DT_\eps(y_\eps^1) \rVert_{H^{-1}}^2$ is integrable, there exist Lebesgue points $s_n \to t_0$ such that $\lVert DT_\eps(y_\eps^1(s_n)) \rVert_{H^{-1}}^2 \lvert t_0 - s_n \rvert \to 0$, and hence $T_\eps(y_\eps^1(t_0)) \geq F(t_0)$. Therefore, we establish the chain rule for $T_\eps \circ y_\eps^1$, namely,
	\begin{equation}
		\label{ChainRule}
		T_\eps(y_\eps^1(t_2)) - T_\eps (y_\eps^1(t_1)) = \int_{t_1}^{t_2} \langle D T_\eps (y_\eps^1(t)) , \partial_t y_\eps^1(t) \rangle_{H_0^1} \; \d t, \qquad \forall 0\leq t_1 < t_2 \leq h.
	\end{equation}

	The absolute continuity of $T_\eps \circ y_\eps^1$ established by means of the chain rule \eqref{ChainRule} allows us to show that $y_\eps^1 \in C^0([0,h];W^{2,q}(\Omega))$. In particular, as $y_\eps^1 \in C^0([0,h];C^{1,\alpha}(\bar{\Omega}))$ is already proved, it remains to prove $\nabla_\eps^2 y_\eps^1 \in C^0([0,h];L^q(\Omega))$. First, observe that $y_\eps^1 \in C_w([0,h];W^{2,q}(\Omega))$ ensures that $y_\eps^1(t) \rightharpoonup y_\eps^1(t_0)$ in $W^{2,q}(\Omega)$ whenever $t \to t_0$. Next, the absolute continuity of $T_\eps \circ y_\eps^1$ guarantees the pointwise convergence
	$
	T_\eps (y_\eps^1(t)) \to T_\eps(y_\eps^1(t_0)),
	$
	namely,
	$$
	\int_\Omega \lvert \nabla_\eps^2 y_\eps^1(t) \rvert^q \; \d x \to \int_\Omega \lvert \nabla_\eps^2 y_\eps^1(t_0) \rvert^q \; \d x.
	$$
	In view of \cite[Proposition 3.32, p.~78]{Brezis}, combining the convergence of the norm with the weak convergence implies the strong convergence $\nabla_\eps^2 y_\eps^1(t) \to \nabla_\eps^2 y_\eps^1(t_0)$ in $L^q(\Omega)$ as $t \to t_0$. Consequently, $\nabla_\eps^2 y_\eps^1 \in C^0([0,h];L^q(\Omega))$.

	\subsection{Iterating the procedure}
	\label{S49}
	The procedure developed so far for the first interval provides the basis for constructing the solution on each interval $[(l-1)h,lh]$, with the time-dependent initial velocity $v_\eps^{l-1}$ replacing the initial datum $v_{0\eps}$ and $y_\eps^{l-1}((l-1)h)$ replacing $y_{0\eps}$ as initial condition. The only point which requires a new argument is the lower bound $\Lambda(y_\eps^{l,k}) > c_0$ for every $k$, which guarantees the Euler-Lagrange equations at every incremental step.
	
	Consider the time interval $[h,2h]$. While the first iteration exhibits the initial condition $y_{0\eps}$ and the time-independent datum $v_{0\eps}$, the current time step involves $y_\eps^1(h)$ and the time-dependent term $v_\eps^1(t)$ (see \eqref{initialSequential}--\eqref{initialSequentialI}). The corresponding terms in the discrete formulation \eqref{incrP}--\eqref{incrIC2} are given by
	\begin{align*}
		y_\eps^{2,0} = y_\eps^1(h) \quad \text{and} \quad v_\eps^{1,k-1} = \frac{1}{\tau} \int_{(k-1)\tau + h}^{k\tau + h} v_\eps^1(t) \; \d t \qquad \text{for } k = 1, \dots, M.
	\end{align*}
	First, we observe that $y_\eps^{2,0} = y_\eps^1(h)$ and $v_\eps^1 = \partial_t y_\eps^1(\cdot-h)$ are admissible initial data, in the sense that $y_\eps^{2,0} \in A_\eps$ has bounded energy and $v_\eps^1$ is bounded in $L^2(h,2h;L^2(\Omega))$, uniformly in $\tau$ and $\eps$. In particular, for every $t \in [0,h]$, $y_\eps^1(t) \in A_\eps$, hence $y_\eps^1(h) = y_\eps^{2,0} \in A_\eps$. Next, by taking $t = h$ in \eqref{1b}, one has
	\begin{equation}
		\label{N4}
		E_\eps (y_\eps^{2,0}) \leq c,	
	\end{equation}
	whereas, by taking the liminf as $\tau \to 0$ in inequality \eqref{secondEstimateb}, in view of the weak convergence \eqref{c3}, we get
	\begin{equation}
		\label{iter1}
		\int_0^h \lVert \partial_t y_\eps^1 (t) \rVert^2 \; \d t \leq c,
	\end{equation}
	where $c$ is independent of $\tau$ and $\eps$. Therefore, by \eqref{iter1} and Jensen's inequality, one has $v_\eps^{1,k-1} \in L^2(\Omega)$ for each $k = 1,\dots, M$, and $\tau \sum_{k=1}^M \lVert v_\eps^{1,k-1} \rVert^2 \leq \int_h^{2h} \lVert v_\eps^1 \rVert^2 \, \d t \leq c$, uniformly in $\tau$ and $\eps$.
	Then, as the structure of the minimization problem remains the same, minimizers $y_\eps^{2,k} \in A_\eps$ of $W_\eps^{2,k-1}$ exist for each $k$, and we recover an energy inequality analogous to \eqref{firstEstimate}. In this case, the terms $E_\eps(y_{0\eps})$ and $\rho \lVert v_{0\eps} \rVert^2$ are replaced by
	\begin{equation}
		\label{iter2}
		E_\eps(y_\eps^1(h)) \quad \text{and} \quad \frac{\rho}{h} \int_h^{2h} \lVert v_\eps^1(t) \rVert^2 \; \d t.
	\end{equation}
	As, due to \eqref{N4} and \eqref{iter1}, these terms are uniformly bounded in $\tau$, the right-hand side of the energy inequality remains bounded independently of $\tau$. As on the first outer time interval, the Euler-Lagrange equations hold, provided that $\Lambda(y_\eps^{2,k}) > c_0$ for every $k$. This lower bound is proved in Proposition \ref{NoContact} below, together with the analogous bound on all outer time intervals.
	
	Under this bound, the construction of the interpolants $y_\eps^{2,(\tau)}$, $\underline{y_\eps}^{2,(\tau)}$, and $\tilde{y}_\eps^{2,(\tau)}$ follows the same logic as in the first iteration, allowing us to deduce
	\begin{align*}
		\langle DE_\eps(y_\eps^{2,(\tau)}(t)), &\phi \rangle + \langle D_2 R_\eps( \underline{y_\eps}^{2,(\tau)}(t), \partial_t \tilde{y}_\eps^{2,(\tau)}(t)), \phi \rangle\\
		&- (\underline{f_\eps}^{2,(\tau)}(t), \phi) + \frac{\rho}{h} ( \partial_t \tilde{y}_\eps^{2,(\tau)}(t)  - \underline{v_\eps}^{1,(\tau)}(t), \phi ) = 0, \quad \text{for a.e. }t \in (h,2h),
	\end{align*}
	where $\underline{f_\eps}^{2,(\tau)}(t) = f_\eps^{2,k-1}$ for $t \in ((k-1)\tau+h,k\tau+h]$, and the forward constant interpolant of the initial velocity $\underline{v_\eps}^{1,(\tau)}$ satisfies $\underline{v_\eps}^{1,(\tau)} \to v_\eps^{1}$ in $L^2(h,2h;H^1(\Omega))$. The resulting energy estimate, again with \eqref{iter2} replacing the terms $E_\eps(y_{0\eps})$ and $\rho \lVert v_{0\eps} \rVert^2$, is uniform in $\tau$. This allows the limit $\tau \to 0$ to be carried out on $[h,2h]$ exactly as for $[0,h]$. The resulting limit deformation $y_\eps^2$ thus solves the problem \eqref{sequential}--\eqref{initialSequential} in the sense of Definition \ref{defSolSeq} on $[h,2h]$ and provides the initial data for the next iteration.
	
	For $l \geq 3$, the same construction applies on $[(l-1)h,lh]$, with $y_\eps^{l-1}((l-1)h)$ and $v_\eps^{l-1}$ in place of $y_\eps^1(h)$ and $v_\eps^1$, as soon as $\Lambda(y_\eps^{l,k}) > c_0$ for every $k$. This lower bound cannot be obtained by estimating the decrease of $\Lambda$ interval by interval, as in the proof of Lemma \ref{LocNC1}.$(i)$. Indeed, this would produce a sum of $l-1$ terms of order $h^{2\gamma_0}$, which is not bounded uniformly in $h$, as $2\gamma_0 < 1$. Instead, in Subsection \ref{S410} we derive estimates independent of $h$ on the sequential solutions constructed so far, and in Subsection \ref{S411} we use them to prove the lower bound by induction on $l$.

		\subsection{Estimates independent of $h$}
	\label{S410}
	In this subsection, we derive estimates on the solutions of the sequential problem which are independent of $h$. Let $m \in \{1,\dots,N\}$, and assume that the solutions $y_\eps^1,\dots,y_\eps^m$ of the sequential problem \eqref{sequential}--\eqref{initialSequentialI} on the intervals $[(l-1)h,lh]$, $l = 1,\dots,m$, have been constructed as in Subsections \ref{S41}--\ref{S49}. We denote by $y_\eps^{(h)}$ their concatenation, namely $y_\eps^{(h)}(t) := y_\eps^l(t)$ for $t \in [(l-1)h,lh]$ and $l = 1,\dots,m$, and we set $\partial_t y_\eps^{(h)} := v_{0\eps}$ in $(-h,0)$. By \eqref{initialSequential}--\eqref{initialSequentialI}, $y_\eps^{(h)}$ is well defined at the nodes $lh$, and $v_\eps^{l-1} = \partial_t y_\eps^{(h)}(\cdot - h)$ in $((l-1)h,lh)$ for every $l = 1,\dots,m$.
	
	Fix $l \in \{1,\dots,m\}$ and $t \in [(l-1)h,lh]$, and consider $\partial_t y_\eps^l \chi_{[(l-1)h,t]}$ as a test function in equation \eqref{seqPweakEq}. In view of the chain rule \eqref{ChainRule} for $T_\eps \circ y_\eps^l$ (the chain rules for $K_\eps \circ y_\eps^l$ and $J_\eps \circ y_\eps^l$ are elementary, as $\nabla_\eps y_\eps^l$ is bounded, $\det \nabla_\eps y_\eps^l \geq \mu$, and $\partial_t \nabla_\eps y_\eps^l \in L^2((l-1)h,lh;L^2(\Omega))$) and using Young's inequality, we obtain
	\begin{align}
		E_\eps& (y_\eps^l(t)) + 2 \int_{(l-1)h}^t R_\eps (y_\eps^l , \partial_t y_\eps^l) \;  \d s + \frac{\rho}{2h} \int_{(l-1)h}^t \lVert \partial_t y_\eps^l \rVert^2 \; \d s \nonumber \\
		&\leq E_\eps (y_\eps^l((l-1)h)) + \int_{(l-1)h}^t ( f_\eps , \partial_t y_\eps^l) \; \d s + \frac{\rho}{2h} \int_{(l-1)h}^t \lVert  v_\eps^{l-1} \rVert^2 \; \d s. \label{L36}
	\end{align}
	As $v_\eps^{l-1} = \partial_t y_\eps^{(h)}(\cdot - h)$ in $((l-1)h,lh)$, this results in the following inequality for $y_\eps^{(h)}$, for every $t \in [(l-1)h,lh]$ and $l = 1,\dots,m$,
	\begin{align}
		\label{est0T1}
		E_\eps& (y_\eps^{(h)}(t)) + 2 \int_{(l-1)h}^{t} R_\eps (y_\eps^{(h)} , \partial_t y_\eps^{(h)}) \;  \d s + \frac{\rho}{2h} \int_{(l-1)h}^{t} \lVert \partial_t y_\eps^{(h)} \rVert^2 \; \d s \nonumber\\
		&\leq E_\eps (y_\eps^{(h)}((l-1)h)) + \int_{(l-1)h}^{t} ( f_\eps , \partial_t y_\eps^{(h)}) \; \d s + \frac{\rho}{2h} \int_{(l-2)h}^{t-h} \lVert  \partial_t y_\eps^{(h)} \rVert^2 \; \d s.
	\end{align}
	Taking $t=lh$ in \eqref{est0T1} and summing over $l = 1,\dots,j$, with $j \in \{1,\dots,m\}$, we deduce
	\begin{align}
		\label{est0T2}
		E_\eps (y_\eps^{(h)}(jh)) +& 2 \int_0^{jh} R_\eps (y_\eps^{(h)} , \partial_t y_\eps^{(h)}) \;  \d s + \frac{\rho}{2h} \int_{(j-1)h}^{jh} \lVert \partial_t y_\eps^{(h)} \rVert^2 \; \d s \nonumber \\
		&\leq E_\eps (y_{0\eps}) + \int_{0}^{jh} ( f_\eps , \partial_t y_\eps^{(h)}) \; \d s + \frac{\rho}{2} \lVert  v_{0\eps} \rVert^2.
	\end{align}
	Consider now inequality \eqref{est0T1} with $l = j$ and $t \in [(j-1)h,jh]$. Adding \eqref{est0T2} with $j-1$ in place of $j$ if $j \geq 2$, and using $\partial_t y_\eps^{(h)} = v_{0\eps}$ in $(-h,0)$ if $j = 1$, it follows that
		\begin{align}
			\label{EIeps}
			E_\eps& (y_\eps^{(h)}(t)) + 2 \int_0^t R_\eps (y_\eps^{(h)} , \partial_t y_\eps^{(h)}) \;  \d s + \frac{\rho}{2h} \int_{t-h}^t \lVert \partial_t y_\eps^{(h)} \rVert^2 \; \d s \nonumber\\
			&\leq E_\eps (y_{0\eps}) + \int_0^t ( f_\eps , \partial_t y_\eps^{(h)}) \; \d s + \frac{\rho}{2} \lVert  v_{0\eps} \rVert^2 ,\qquad \forall t \in [0,mh].
		\end{align}
		By Young's inequality, this gives
	\begin{align}
		\label{est0T3}
		E_\eps& (y_\eps^{(h)}(t)) + 2 \int_0^t R_\eps (y_\eps^{(h)} , \partial_t y_\eps^{(h)}) \;  \d s + \frac{\rho}{2h} \int_{t-h}^t \lVert \partial_t y_\eps^{(h)} \rVert^2 \; \d s \nonumber\\
		&\leq E_\eps (y_{0\eps}) + \frac{1}{2\delta} \int_0^t \lVert f_\eps  \rVert^2 \; \d s + \frac{\delta}{2} \int_0^t \lVert \partial_t y_\eps^{(h)} \rVert^2 \; \d s + \frac{\rho}{2} \lVert  v_{0\eps} \rVert^2 ,\qquad \forall t \in [0,mh],
	\end{align}
	where $\delta >0$ is arbitrary. Singling out the third term on the left-hand side, taking $t = lh$, multiplying by $h$, and summing over $l = 1,\dots, m$, we get, with the choice $\delta = \rho/(2T)$ and as $mh \leq T$,
	\begin{equation}
		\label{est0T4}
		\frac{\rho}{4} \int_0^{mh} \lVert \partial_t y_\eps^{(h)} \rVert^2 \; \d t \leq \biggl( E_\eps (y_{0\eps}) + \frac{\rho}{2} \lVert v_{0\eps} \rVert^2 \biggr) T + \biggl(\frac{1}{\rho} \int_0^T \lVert f_\eps \rVert^2  \; \d t \biggr) T^2.
	\end{equation}
	Using this in inequality \eqref{est0T3}, we obtain that its left-hand side is bounded by a constant $c$ for every $t \in [0,mh]$, where, by Assumptions (A1)--(A3), $c$ depends only on $\rho$, $\Omega$, $T$, $c_2$, $v_0$, and $f_0$, and is nondecreasing in $T$. In particular, $c$ is independent of $h$, $\eps$, and $m$, and
	\begin{equation}
		\label{est0T5}
		\sup_{t \in [0,mh]} E_\eps(y_\eps^{(h)}(t)) + \int_0^{mh} R_\eps (y_\eps^{(h)},\partial_t y_\eps^{(h)}) \; \d t \leq c.
	\end{equation}
	
		\subsection{Inactivity of the averaged invertibility constraint}
	\label{S411}
	In this subsection, we extend the lower bound of Lemma \ref{LocNC1}.$(ii)$, obtained on the first outer time interval, to all outer time intervals, to all outer time intervals, provided that $T$ is small enough. This guarantees the validity of the Euler-Lagrange equations of Subsection \ref{S42} at every incremental step. The proof is by induction on $l$, and it relies on the estimates of Subsection \ref{S410}.
	
	\begin{proposition}
		\label{NoContact}
		There exists $T_0>0$, depending only on $\rho$, $q$, $\Omega$, $\mathcal{C}$, $c_1 - c_0$, $c_2$, $v_0$, and $f_0$, such that, if $T \leq T_0$, then $\Lambda(y_\eps^{l,k}) \geq c_3$ for every $\eps$, $h$, and $\tau$, and for all $l=1,\dots,N$ and $k=0,\dots,M$. Moreover, $\Lambda(y_\eps^l(t)) \geq c_3$ for every $t \in [(l-1)h,lh]$ and $l = 1,\dots,N$.
	\end{proposition}
	\begin{proof}
		By the bound on the left-hand side of \eqref{est0T3} and by \eqref{est0T4}, there exists $\kappa>0$, depending only on $\rho$, $\Omega$, $T$, $c_2$, $v_0$, and $f_0$, and nondecreasing in $T$, such that the following holds. If, for some $m \in \{1,\dots,N\}$, the sequential solutions $y_\eps^1,\dots,y_\eps^m$ have been constructed, then their concatenation $y_\eps^{(h)}$ satisfies
		\begin{equation}
			\label{NCbound}
			\sup_{t\in[0,mh]} \biggl( E_\eps(y_\eps^{(h)}(t)) + \frac{\rho}{2h}\int_{t-h}^{t} \lVert \partial_t y_\eps^{(h)} \rVert^2 \; \d s \biggr) + \int_0^{mh} \lVert \partial_t y_\eps^{(h)} \rVert^2 \; \d t \leq \kappa.
		\end{equation}
		Let $C_* := C_{2M_\kappa + \kappa} (2M_\kappa + \kappa)^{\gamma_0}$, with $M_\kappa$ and $C_M$ as in Lemmas \ref{W2qnormEps} and \ref{InvStab}.$(ii)$, respectively, and let $\kappa' := 2\kappa + \kappa_1$, with $\kappa_1$ as in the proof of Lemma \ref{LocNC1}.$(ii)$. As $\kappa$ and $\kappa_1$ are nondecreasing in $T$, $M_\kappa$ and $C_\Lambda(\kappa)$ are nondecreasing in $\kappa$, and $C_M$ is nondecreasing in $M$, the function $T \mapsto C_* T^{\gamma_0} + C_\Lambda(\kappa') T^{2\gamma_0}$ is nondecreasing and tends to zero as $T \to 0$. Hence, up to reducing the time $T_0$ of Lemma \ref{LocNC1}.$(ii)$, one has
		\begin{equation}
			\label{NCT0}
			C_* T^{\gamma_0} + C_\Lambda(\kappa') T^{2\gamma_0} \leq \frac{c_1 - c_0}{2}, \qquad \forall T \leq T_0.
		\end{equation}
		
		Let $T \leq T_0$. We argue by induction on $l$. For $l=1$, the claim on $y_\eps^{1,k}$ is given by Lemma \ref{LocNC1}.$(ii)$. Fix $l \in \{2,\dots,N\}$, and assume that $\Lambda(y_\eps^{j,k}) \geq c_3$ for all $j \leq l-1$ and all $k$. Then, the Euler-Lagrange equations of Subsection \ref{S42} hold at every incremental step on the intervals $[(j-1)h,jh]$ with $j \leq l-1$, so that the construction of Subsections \ref{S42}--\ref{S49} provides the sequential solutions $y_\eps^1,\dots,y_\eps^{l-1}$. Hence, \eqref{NCbound} holds with $m = l-1$.
		
		By \eqref{NCbound} and Lemma \ref{W2qnormEps}, one has $\lVert y_\eps^{(h)}(t) \rVert_{W^{2,q}(\Omega)} \leq M_\kappa$ for every $t \in [0,(l-1)h]$, and $\int_0^{(l-1)h} \lVert \partial_t y_\eps^{(h)} \rVert^2 \, \d t \leq \kappa$. Hence, Lemma \ref{InvStab}.$(ii)$, applied on $[0,(l-1)h]$ with $2M_\kappa + \kappa$ in place of $M$, and Assumption (A1) deliver
		$$
		\Lambda(y_\eps^{l,0}) = \Lambda(y_\eps^{(h)}((l-1)h)) \geq \Lambda(y_{0\eps}) - C_* T^{\gamma_0} \geq c_1 - C_* T^{\gamma_0}.
		$$
		Next, since $v_\eps^{l-1} = \partial_t y_\eps^{(h)}(\cdot-h)$ in $((l-1)h,lh)$, the bound \eqref{NCbound} with $t = (l-1)h$, together with the bound on $f_\eps$ in the proof of Lemma \ref{LocNC1}.$(ii)$, gives
		$$
		E_\eps(y_\eps^{l,0}) + \frac{\rho}{h} \int_{(l-1)h}^{lh} \lVert v_\eps^{l-1} \rVert^2 \; \d t + \frac{2h}{\rho} \int_{(l-1)h}^{lh} \lVert f_\eps \rVert^2 \; \d t \leq \kappa'.
		$$
		Since the derivation of \eqref{firstEstimate} relies only on the minimality of the incremental solutions, it applies on $[(l-1)h,lh]$, with $E_\eps(y_{0\eps})$ and $\rho \lVert v_{0\eps} \rVert^2$ replaced by $E_\eps(y_\eps^{l,0})$ and $(\rho/h) \int_{(l-1)h}^{lh} \lVert v_\eps^{l-1} \rVert^2 \, \d t$, by Jensen's inequality. Therefore, the proof of Lemma \ref{LocNC1}.$(i)$, with $y_\eps^{l,0}$ in place of $y_{0\eps}$ and $\kappa'$ in place of $\kappa$, and $h \leq T$ give
		$$
		\Lambda(y_\eps^{l,k}) \geq \Lambda(y_\eps^{l,0}) - C_\Lambda(\kappa') T^{2\gamma_0} \geq c_1 - C_* T^{\gamma_0} - C_\Lambda(\kappa') T^{2\gamma_0}, \qquad k = 0,\dots,M.
		$$
		By \eqref{NCT0}, we obtain $\Lambda(y_\eps^{l,k}) \geq c_3$ for every $k$, which completes the induction. Note that $T_0$ is chosen before the induction, and it does not depend on $l$, as $\kappa$ is independent of $h$ and $m$.
		
		Finally, for every $l$, the backward constant interpolant $y_\eps^{l,(\tau)}$ takes values in $\{y_\eps^{l,k}\}_k$ and converges to $y_\eps^l$ in $L^\infty((l-1)h,lh;C^{1,\alpha}(\bar{\Omega}))$, as in \eqref{c5}. Hence, by Lemma \ref{InvStab}.$(i)$, $\Lambda(y_\eps^l(t)) \geq c_3$ for a.e. $t$, and, by the continuity of $y_\eps^l$ with values in $C^{1,\alpha}(\bar{\Omega})$, for every $t \in [(l-1)h,lh]$.
	\end{proof}
	
	Consequently, if $T \leq T_0$, the construction of Subsections \ref{S42}--\ref{S49} applies on every outer time interval, and it provides the incremental minimizers $\{y_\eps^{l,k}\}_{l=1,\,k=0}^{N,M}$ and the solution $\{y_\eps^l\}_{l=1}^N$ of the sequential problem \eqref{sequential}--\eqref{initialSequentialI}. This proves Theorem \ref{Theorem1}.$(i)$--$(ii)$.

	\subsection{Global interpolant, corresponding estimates, and convergences}
	\label{SGlob}
	Theorem \ref{Theorem1}.$(iii)$ is readily obtained by introducing the global interpolant $y_\eps^{(h)}: [0,T] \times \Omega \to \mathbb{R}^3$, namely the concatenation of Subsection \ref{S410} with $m = N$. Thus, $y_\eps^{(h)}(t):=  y_\eps^l(t) $ for $t \in [(l-1)h,lh]$ with $l=1,\dots,N$, where $\{y_\eps^l\}_{l=1}^N$ is the solution to the sequential problem \eqref{sequential}--\eqref{initialSequentialI} provided by Subsections \ref{S41}--\ref{S411}. By construction, $y_\eps^{(h)}$ is a weak solution to the time-delayed problem \eqref{tDelay}--\eqref{initialVelocity} in the sense of Definition \ref{defSoltimedelayedT}. In particular, for every $\Psi \in L^2(0,T;H_0^{1}(\Omega))$, $y_\eps^{(h)}$ satisfies
	\begin{align}
		&0=\int_0^T \biggl( \langle D E_\eps (y_\eps^{(h)}) , \Psi \rangle_{H_0^1} + \langle D_2 R_\eps (y_\eps^{(h)},\partial_t y_\eps^{(h)}) , \Psi \rangle - (f_\eps, \Psi) \nonumber\\&\qquad\qquad+ \frac{\rho}{h} (\partial_t y_\eps^{(h)} - \partial_t y_\eps^{(h)} (\cdot - h) , \Psi) \biggr) \; \d t, \label{m1}\\
		&y_\eps^{(h)}(0)=y_{0\eps}, \qquad \partial_t y_\eps^{(h)}(t-h) = v_{0\eps} \; \text{a.e. in } (0,h) . \label{m2}
	\end{align}	
		To address the limit as $h \to 0$, we observe that estimates \eqref{est0T3}--\eqref{est0T5} hold with $m = N$, namely on the whole interval $[0,T]$, with constants independent of both $h$ and $\eps$. Thus, for fixed $\eps$, in view of Lemmas \ref{W2qnormEps}, \ref{PositiveDet}, and \ref{H1normEps}, estimate \eqref{est0T5} implies
		\begin{align}
			\sup_{t \in [0,T]} \lVert y_\eps^{(h)} (t) \rVert_{W^{2,q}(\Omega)} + \lVert y_\eps^{(h)} \rVert_{H^1(0,T;H^1(\Omega))}  \leq c \label{est0T6},
		\end{align}
		where the constant $c$ is independent of $h$ but, through Lemma \ref{H1normEps}, depends on $\eps$ (see Remark \ref{RemKorn}). The control of the first term in \eqref{est0T6} delivers
		\begin{equation}
			\label{est0T8}
			\lVert DT_\eps(y_\eps^{(h)}) \rVert_{L^{\infty}(0,T;(W^{2,q}(\Omega))')} \leq c.
		\end{equation}
		
		Thanks to estimates \eqref{est0T6} and \eqref{est0T8} we find $y_\eps$ and $\mathcal{P}_\eps$ such that, up to non-relabeled subsequences, the following convergences hold
		\begin{alignat}{4}
			y_\eps^{(h)} &\overset{*}{\rightharpoonup} y_\eps \quad &&\text{in} \quad L^{\infty}(0,T;W^{2,q}(\Omega)),\label{d1}\\
			y_\eps^{(h)} &\rightharpoonup y_\eps \quad &&\text{in} \quad H^1(0,T;H^1(\Omega)),\label{d2}\\
			y_\eps^{(h)} &\to y_\eps \quad &&\text{in} \quad C^0([0,T];C^{1,\alpha}(\bar{\Omega})),\label{d3}\\
			DT_\eps(y_\eps^{(h)}) &\overset{*}{\rightharpoonup} \mathcal{P}_\eps \quad &&\text{in} \quad L^{\infty}(0,T;(W^{2,q}(\Omega))'), \label{d4}
		\end{alignat}
		where \eqref{d3} is obtained via \eqref{d1}--\eqref{d2} by means of Simon's compactness argument (Theorem \ref{cor4Simon}). Next, for every $t \in [0,T]$, $y_\eps^{(h)}(t) \in A_\eps$ and its energy is bounded by \eqref{est0T5}. Hence, convergence \eqref{d3} and Lemma \ref{ClosA}.$(i)$ give $y_\eps(t) \in A_\eps$ and $\det \nabla_\eps y_\eps(t) \geq \mu$ in $\bar{\Omega}$, and, by Proposition \ref{NoContact}, $\Lambda(y_\eps(t)) \geq c_3$. Moreover, as in Step 2 of Subsection \ref{S44}, \cite[Lemma II.5.9, p.~95]{boyer} gives $y_\eps \in C_w([0,T];W^{2,q}(\Omega))$.
		Finally, convergence \eqref{d3} guarantees $y_\eps(0,x) = y_{0\eps}$ for every $x \in \bar{\Omega}$ and, as $\partial_t y_\eps \in L^2(0,T;H^1(\Omega))$ and $y_\eps^{(h)}\rvert_P = {\rm id}_\eps$, from the continuity of the trace operator, we deduce $\partial_t y_\eps \in L^2(0,T;H_0^1(\Omega))$.

	\subsection{Passage to the limit as $h \to 0$}
	\label{S412}
	As $y_\eps^{(h)}$ satisfies \eqref{m1}--\eqref{m2}, in particular,
	\begin{align}
		0=\int_0^T \biggl( &\langle D K_\eps (y_\eps^{(h)}) , \phi \rangle + \langle D J_\eps (y_\eps^{(h)}) , \phi \rangle + \langle D T_\eps (y_\eps^{(h)}) , \phi \rangle + \langle D_2 R_\eps (y_\eps^{(h)},\partial_t y_\eps^{(h)}) , \phi \rangle \nonumber\\ &- (f_\eps, \phi)+ \frac{\rho}{h} (\partial_t y_\eps^{(h)} - \partial_t y_\eps^{(h)} (\cdot - h) , \phi) \biggr) \psi \; \d t,\label{d5}
	\end{align}
		for every $\phi \in L^1(0,T;V) \cap H^1(0,T;H_0^1(\Omega))$ and every $\psi$ as in Lemma \ref{TOT}, with associated $h_0 \in (0,T)$. As $\phi$ takes values in $V$, the duality pairings in \eqref{d5} are given by the explicit expressions \eqref{defDE}--\eqref{defDR}. In view of convergences \eqref{d1}--\eqref{d4}, an adaptation of the argument carried out in Subsection \ref{S46b} results in the following convergence
	\begin{align*}
		\int_0^T &\biggl( \langle D K_\eps (y_\eps^{(h)}) , \phi \rangle + \langle D J_\eps (y_\eps^{(h)}) , \phi \rangle+ \langle D T_\eps (y_\eps^{(h)}) , \phi \rangle + \langle D_2 R_\eps (y_\eps^{(h)},\partial_t y_\eps^{(h)}) , \phi \rangle \biggr) \psi \; \d t\\ 
		&\to \int_0^T \biggl( \langle D K_\eps (y_\eps) , \phi \rangle + \langle D J_\eps (y_\eps) , \phi \rangle + \langle \mathcal{P}_\eps , \phi \rangle + \langle D_2 R_\eps (y_\eps,\partial_t y_\eps) , \phi \rangle \biggr) \psi \; \d t.
	\end{align*}
	As for the inertia term, we proceed as follows. For the sake of notational simplicity, we set $\Psi = \phi\, \psi$. By splitting the two terms while shifting the second one, one gets 
	\begin{align*}
		\int_0^T& \biggl(\frac{\partial_t y_\eps^{(h)}(t) - \partial_t y_\eps^{(h)} (t - h)}{h} , \Psi(t)\biggr) \; \d t\\
		&=\frac{1}{h}\int_0^T (\partial_t y_\eps^{(h)}(t) , \Psi(t)) \; \d t - \frac{1}{h}\int_0^T (\partial_t y_\eps^{(h)} (t - h) , \Psi(t)) \; \d t\\
		&=\frac{1}{h}\int_0^T (\partial_t y_\eps^{(h)}(t) , \Psi(t)) \; \d t -  \frac{1}{h}\int_{-h}^0 (\partial_t y_\eps^{(h)} (t) , \Psi(t+h)) \; \d t \\&\quad- \frac{1}{h}\int_0^{T-h} (\partial_t y_\eps^{(h)} (t) , \Psi(t+h))  \; \d t.
	\end{align*}
	Thus, by exploiting $\partial_t y_\eps^{(h)}(t) = v_{0\eps}$ almost everywhere in $(-h,0)$,
	\begin{align*}
		\int_0^T& \biggl(\frac{\partial_t y_\eps^{(h)}(t) - \partial_t y_\eps^{(h)} (t - h)}{h} , \Psi(t)\biggr) \; \d t\\
		&=-\int_0^{T-h} \biggl(\partial_t y_\eps^{(h)}(t) , \frac{\Psi(t+h) - \Psi(t)}{h}\biggr) \; \d t -  \frac{1}{h}\int_{-h}^0 (v_{0\eps} , \Psi(t+h)) \; \d t \\&\quad+ \frac{1}{h}\int_{T-h}^T (\partial_t y_\eps^{(h)} (t) , \Psi(t))  \; \d t.
	\end{align*}
	As $\psi$ vanishes for $t \geq T-h_0$, the first integral can be extended to the whole interval $(0,T)$, whereas the third term vanishes. By the fundamental theorem of calculus and the Lebesgue differentiation theorem, convergence \eqref{d2} guarantees
	\begin{align*}
		\int_0^T& \biggl(\frac{\partial_t y_\eps^{(h)}(t) - \partial_t y_\eps^{(h)} (t - h)}{h} , \Psi(t)\biggr) \; \d t\to - \int_0^T (\partial_t y_\eps (t),\partial_t \Psi(t)) \; \d t - (v_{0\eps},\Psi(0)). 
	\end{align*}
	Consequently, the limit deformation $y_\eps$ satisfies
	\begin{align}
		\int_0^T \biggl( &\langle D K_\eps (y_\eps) , \Psi \rangle + \langle D J_\eps (y_\eps) , \Psi \rangle + \langle \mathcal{P}_\eps , \Psi \rangle \nonumber\\ &+ \langle D_2 R_\eps (y_\eps,\partial_t y_\eps) , \Psi \rangle - (f_\eps, \Psi) - \rho (\partial_t y_\eps,\partial_t \Psi) \biggr) \; \d t - \rho (v_{0\eps},\Psi(0)) = 0,\label{d9}
	\end{align}
	for every $\Psi = \phi \, \psi$ with $\phi \in L^1(0,T;V) \cap H^1(0,T;H_0^1(\Omega))$ and $\psi$ as in Lemma \ref{TOT}.
	
	To identify $\mathcal{P}_\eps$ we pick the sequence $\phi^{(h)} = y_\eps^{(h)} - {\rm id}_\eps$ as test function in \eqref{d5}, which is allowed by the regularity of $y_\eps^{(h)}$, and we use again the simplified notation $\Psi^{(h)} = \psi\,\phi^{(h)}$. The convergence $\phi^{(h)} \to \phi = y_\eps - {\rm id}_\eps$ in $C^0([0,T];C^{1,\alpha}(\bar{\Omega}))$ allows us to pass to the limit in all the terms except the inertial one. Since the test function now depends on $h$, the latter requires the convergence of the difference quotients of $y_\eps^{(h)}$. Let
	$$
	\mathcal{Q}_\eps^{(h)}(t) := \frac{y_\eps^{(h)}(t) - y_\eps^{(h)}(t-h)}{h} = \frac{1}{h}\int_{t-h}^t \partial_t y_\eps^{(h)}(s) \; \d s, \qquad t \in [0,T],
	$$
	where $y_\eps^{(h)}(s) := y_{0\eps} + s\,v_{0\eps}$ for $s \in (-h,0)$, in accordance with \eqref{m2}. By Jensen's inequality and estimates \eqref{est0T3}--\eqref{est0T4}, one has $\lVert \mathcal{Q}_\eps^{(h)}(t) \rVert^2 \leq \frac{1}{h}\int_{t-h}^t \lVert \partial_t y_\eps^{(h)} \rVert^2 \, \d s \leq c$ for every $t \in [0,T]$, so that $\mathcal{Q}_\eps^{(h)}$ is bounded in $L^\infty(0,T;L^2(\Omega))$. Moreover, $\partial_t \mathcal{Q}_\eps^{(h)} = (\partial_t y_\eps^{(h)} - \partial_t y_\eps^{(h)}(\cdot-h))/h$, which is bounded in $L^2(0,T;V')$ by equation \eqref{m1}, tested with $V$-valued functions, and estimates \eqref{est0T6}--\eqref{est0T8}. Since $L^2(\Omega) \overset{c}{\hookrightarrow} H^{-1}(\Omega) \hookrightarrow V'$, Theorem \ref{cor4Simon} guarantees, up to a subsequence (not relabeled), $\mathcal{Q}_\eps^{(h)} \to \mathcal{Q}_\eps$ in $C^0([0,T];H^{-1}(\Omega))$ for some $\mathcal{Q}_\eps$. To identify $\mathcal{Q}_\eps$, let $\vartheta \in C_c^\infty((0,T)\times\Omega)$. For $h$ small enough, shifting the integration variable as above and using \eqref{d3}, we obtain
	\begin{align*}
		\int_0^T (\mathcal{Q}_\eps^{(h)},\vartheta) \; \d t &= - \int_0^T \Bigl( y_\eps^{(h)}(t) , \frac{\vartheta(t+h) - \vartheta(t)}{h} \Bigr) \; \d t \\
		&\to - \int_0^T (y_\eps,\partial_t \vartheta) \; \d t = \int_0^T (\partial_t y_\eps , \vartheta) \; \d t,
	\end{align*}
	whence $\mathcal{Q}_\eps = \partial_t y_\eps$. As $\mathcal{Q}_\eps$ is uniformly continuous with values in $H^{-1}(\Omega)$, also $\mathcal{Q}_\eps^{(h)}(\cdot + h) \to \partial_t y_\eps$ in $C^0([0,T-h_0];H^{-1}(\Omega))$.
	
	Now, shifting the second part of the inertial term as above, one has
	\begin{align*}
		\frac{1}{h}\int_0^T \bigl( \partial_t y_\eps^{(h)} - \partial_t y_\eps^{(h)}(\cdot-h) , \Psi^{(h)} \bigr) \; \d t = &- \int_0^T \Bigl( \partial_t y_\eps^{(h)}(t) , \frac{\Psi^{(h)}(t+h) - \Psi^{(h)}(t)}{h} \Bigr) \; \d t \\
		&- \frac{1}{h}\int_0^h (v_{0\eps},\Psi^{(h)}(t)) \; \d t,
	\end{align*}
	where
	$$
	\frac{\Psi^{(h)}(t+h) - \Psi^{(h)}(t)}{h} = \psi(t+h) \, \mathcal{Q}_\eps^{(h)}(t+h) + \frac{\psi(t+h) - \psi(t)}{h} \bigl( y_\eps^{(h)}(t) - {\rm id}_\eps \bigr).
	$$
	The contribution of the first summand is quadratic in the velocity. Since $\partial_t y_\eps^{(h)} \rightharpoonup \partial_t y_\eps$ in $L^2(0,T;H_0^1(\Omega))$ by \eqref{d2} and the boundary conditions, and $\mathcal{Q}_\eps^{(h)}(\cdot+h) \to \partial_t y_\eps$ in $C^0([0,T-h_0];H^{-1}(\Omega))$, the duality between $H_0^1(\Omega)$ and $H^{-1}(\Omega)$ gives
	\begin{equation}
		\label{quadterm}
		\int_0^T \psi(t+h) \, \bigl( \partial_t y_\eps^{(h)}(t) , \mathcal{Q}_\eps^{(h)}(t+h) \bigr) \; \d t \to \int_0^T \psi \, \lVert \partial_t y_\eps \rVert^2 \; \d t.
\end{equation}
	The contribution of the second summand converges by \eqref{d2} and \eqref{d3}, and so does the last term. Altogether, the inertial term converges to $- \rho \int_0^T (\partial_t y_\eps , \partial_t \Psi) \, \d t - \rho (v_{0\eps},\Psi(0))$, with $\Psi = \psi\,\phi$. Hence, we get
	\begin{align}
		&\int_0^T \biggl( \langle D K_\eps (y_\eps^{(h)}) , \Psi^{(h)} \rangle + \langle D J_\eps (y_\eps^{(h)}) , \Psi^{(h)} \rangle + \langle D_2 R_\eps (y_\eps^{(h)},\partial_t y_\eps^{(h)}) , \Psi^{(h)} \rangle - (f_\eps, \Psi^{(h)})\nonumber \\ &\qquad\;+ \frac{\rho}{h} (\partial_t y_\eps^{(h)} - \partial_t y_\eps^{(h)} (\cdot - h),\Psi^{(h)})\biggr) \; \d t\nonumber\\
		&\to \int_0^T \biggl( \langle D K_\eps (y_\eps) , \Psi \rangle + \langle D J_\eps (y_\eps) , \Psi \rangle  + \langle D_2 R_\eps (y_\eps,\partial_t y_\eps) , \Psi \rangle - (f_\eps , \Psi)\nonumber \\ &\qquad\; - \rho (\partial_t y_\eps , \partial_t \Psi) \biggr) \; \d t - \rho (v_{0\eps} , \Psi(0)), \label{d8}
	\end{align}
	for every $\psi$ as in Lemma \ref{TOT}.
	Next, we exploit the monotonicity of $DT_\eps$ on $W^{2,q}(\Omega)$. For all $v \in L^{q}(0,T;W^{2,q}(\Omega))$ and any $\psi$ as above, one has
	\begin{equation}
		\label{d6}
		0 \leq \int_0^T \langle DT_\eps (y_\eps^{(h)}) - DT_\eps(v) , y_\eps^{(h)} - v \rangle  \psi\; \d t.
	\end{equation}
	First, we establish the following convergence
	\begin{equation}
		\label{d7}
		\int_0^T \langle DT_\eps(y_\eps^{(h)}) , y_\eps^{(h)} \rangle \psi \; \d t \to 		\int_0^T \langle \mathcal{P}_\eps , y_\eps \rangle  \psi \; \d t.
	\end{equation}
	To this aim, we introduce the boundary datum ${\rm id}_\eps$ inside the duality pairing. Exploiting equation \eqref{d5}, we can expand the integral as follows
	\begin{align*}
		&\int_0^T \langle DT_\eps(y_\eps^{(h)}) , y_\eps^{(h)} \rangle \psi \; \d t = \int_0^T \langle DT_\eps(y_\eps^{(h)}) , y_\eps^{(h)} - {\rm id}_\eps \rangle \psi \; \d t + \int_0^T \langle DT_\eps(y_\eps^{(h)}) , {\rm id}_\eps \rangle \psi \; \d t \\
		&= - \int_0^T \biggl( \langle D K_\eps (y_\eps^{(h)}) , \Psi^{(h)}  \rangle  + \langle D J_\eps (y_\eps^{(h)}) , \Psi^{(h)}  \rangle+ \langle D_2 R_\eps (y_\eps^{(h)},\partial_t y_\eps^{(h)}) , \Psi^{(h)}  \rangle\\ &\qquad \qquad\quad- (f_\eps, \Psi^{(h)} )+ \frac{\rho}{h} (\partial_t y_\eps^{(h)} - \partial_t y_\eps^{(h)} (\cdot - h) , \Psi^{(h)} ) \biggr) \; \d t + \int_0^T \langle DT_\eps(y_\eps^{(h)}) , {\rm id}_\eps \rangle \psi \; \d t,
	\end{align*}
	where we have set $\Psi^{(h)} = \psi (y_\eps^{(h)} - {\rm id}_\eps )$. Passing to the limit as $h \to 0$, in view of convergences \eqref{d4} and \eqref{d8} and using the already established limit equation \eqref{d9}, we recover \eqref{d7}. We then consider the limit in \eqref{d6}. Convergences \eqref{d1}, \eqref{d4}, and \eqref{d7} result in
		\begin{equation*}
			0 \leq \int_0^T \langle \mathcal{P}_\eps - DT_\eps(v) , y_\eps - v \rangle \psi \; \d t.
		\end{equation*}
	As in Subsection \ref{S46b}, the map $v \mapsto \psi\, DT_\eps(v)$ is hemicontinuous on $L^q(0,T;W^{2,q}(\Omega))$, so that Lemma \ref{LemMinty}, applied with $\mathcal{P} = \psi\,\mathcal{P}_\eps$, gives $\psi\,\mathcal{P}_\eps = \psi\, DT_\eps(y_\eps)$. Since $\psi > 0$ in $(0,T-h_0)$ and $h_0$ is arbitrary, this delivers $\mathcal{P}_\eps = DT_\eps(y_\eps)$ in $L^{q'}(0,T;(W^{2,q}(\Omega))')$.
		
	Finally, Lemma \ref{TOT}, applied with $\mathcal{D} = \Omega$ and $X = V$, closes the proof.
	
		\section{Proof of Theorem \ref{Theorem2}}
		\label{sec:5}
		This section is devoted to the proof of Theorem \ref{Theorem2}.
		We show that any family of solutions $(\{y_\eps^{l,k}\}_{l=1 , \; k=0}^{N,M} , \{ y_\eps^l \}_{l=1}^{N-1})$ to the minimization problem \eqref{incrP}--\eqref{incrIC2} converges to a solution $(\{(y^{l,k},b^{l,k}) \}_{l=1, \; k = 0 }^{N,M},\{(y^l,b^l)\}_{l=1}^{N-1})$ to the membrane minimization problem \eqref{MincrP}--\eqref{MincrIC2}. Here, $\{y_\eps^l\}_{l=1}^{N-1}$ is the solution of the sequential problem provided by Theorem \ref{Theorem1}.$(ii)$. To this aim, in Subsection \ref{S51} we state and prove our $\Gamma$-convergence result, namely, Theorem \ref{Gamma0h}. For all $l = 1,\dots,N$ and $k = 1,\dots, M$, the functionals $W_\eps^{l,k-1}$ are proved to $\Gamma$-converge to the corresponding membrane functionals $W^{l,k-1}$, provided that the data converge strongly. This allows us to prove, in Subsection \ref{S52}, that the minimizers $\{ y_\eps^{1,k} \}_{k=1}^M$ of $W_\eps^{1,k-1}$ converge to the minimizers $\{(y^{1,k},b^{1,k}) \}_{k = 1 }^{M}$ of the corresponding membrane functionals $W^{1,k-1}$ as $\eps \to 0$. Finally, in Subsection \ref{S53} we close the proof by iterating the arguments above to all $l = 1, \dots, N$. 
		
		\subsection{$\Gamma$-convergence}
		\label{S51}
		We state and prove the following $\Gamma$-convergence theorem.
		\begin{theorem}
			\label{Gamma0h}
			Assume {\rm (A0)} and {\rm (A3)}. Let $l \in \{1,\dots,N\}$ and $k \in \{1,\dots,M\}$, and assume that the data of $W_\eps^{l,k-1}$ satisfy $y_\eps^{l,k-1} \in A_\eps$ and
			\begin{subequations}
				\label{AssDataGen}
				\begin{alignat}{4}
					&\nabla_\eps y_\eps^{l,k-1} &&\to (\nabla' y^{l,k-1}\vert b^{l,k-1}) \quad &&\text{in} \quad C^{0,\alpha}(\bar{\omega}),\label{AssDataGena}\\
					&v_\eps^{l-1,k-1} &&\to v^{l-1,k-1} \quad &&\text{in} \quad L^2(\omega),\label{AssDataGenb}
				\end{alignat}
			\end{subequations}
			for some $(y^{l,k-1},b^{l,k-1}) \in B$ and $v^{l-1,k-1} \in L^2(\omega)$. Let $\{y_\eps\}_\eps\subset A_\eps$ with $W_\eps^{l,k-1}(y_\eps) \leq c$. Then, there exist $(y,b) \in B$ and $d \in L^q(\Omega)$ such that, up to non-relabeled subsequences, the following convergences hold
			\begin{alignat}{4}
				&y_\eps &&\rightharpoonup y \quad &&\text{in} \quad W^{2,q}(\omega)\label{e1},\\
				&\nabla_\eps y_\eps &&\rightharpoonup (\nabla' y \vert b) \quad &&\text{in} \quad W^{1,q}(\omega),\label{e2}\\
				&\nabla_\eps^2 y_\eps &&\rightharpoonup \begin{pmatrix}
					(\nabla')^2 y & \nabla' b\\
					(\nabla'b)^T & d
				\end{pmatrix}\quad &&\text{in} \quad L^q(\Omega)\label{e3},
			\end{alignat}
			and
			\begin{equation}
				\label{e7}
				\liminf_{\eps \to 0} W_\eps^{l,k-1}(y_\eps) \geq W^{l,k-1}(y,b).
			\end{equation}
			
			Also, for every $(\bar{y},\bar{b}) \in B$ with $E(\bar{y},\bar{b}) < \infty$, there exists a recovery sequence $\{\bar{y}_{\eps}\}_{\eps} \subset A_\eps$ such that
			\begin{alignat}{4}
				&\bar{y}_\eps &&\rightarrow \bar{y} \quad &&\text{in} \quad W^{2,q}(\omega),\label{e4}\\
				&\nabla_\eps \bar{y}_\eps &&\rightarrow (\nabla' \bar{y} \vert \bar{b}) \quad &&\text{in} \quad W^{1,q}(\omega),\label{e5}\\
				&\nabla_\eps^2 \bar{y}_\eps &&\rightharpoonup \begin{pmatrix}
					(\nabla')^2 \bar{y} & \nabla' \bar{b}\\
					(\nabla'\bar{b})^T & 0
				\end{pmatrix}\quad &&\text{in} \quad L^q(\omega)\label{e6},
			\end{alignat}
			and
			\begin{equation}
				\label{e8}
				\limsup_{\eps \to 0} W_\eps^{l,k-1}(\bar{y}_\eps) \leq W^{l,k-1}(\bar{y},\bar{b}).
			\end{equation}
		\end{theorem}
		\begin{proof}
		Arguing as in Subsection \ref{S41}, with $z_\eps := (y_\eps - y_\eps^{l,k-1})/\tau$ in place of $z^n$, and using Assumption (A3) and \eqref{AssDataGen} to bound the data independently of $\eps$, we obtain $E_\eps(y_\eps) \leq c$. Hence, Lemma \ref{ClosAB} provides $(y,b) \in B$ and $d \in L^q(\Omega)$ such that, up to non-relabeled subsequences, the convergences \eqref{e1}--\eqref{e3} hold, together with
				\begin{alignat}{4}
					&y_\eps &&\to y \quad &&\text{in} \quad C^{1,\alpha}(\bar{\omega}),\label{e9}\\ 
					&\nabla_\eps y_\eps &&\to (\nabla' y \vert b) \quad &&\text{in} \quad C^{0,\alpha}(\bar{\omega}),\label{e10}
				\end{alignat}
				and $K_\eps(y_\eps) \to K(y,b)$, $J_\eps(y_\eps) \to J(y,b)$, and $T(y,b) \leq \liminf_{\eps\to0} T_\eps(y_\eps)$.
				
				Next, we prove \eqref{e7}. In view of the above, it remains to consider the dissipation, load, and inertial terms. By \eqref{AssDataGen} and the boundary conditions in $A_\eps$, one has $y_\eps^{l,k-1} \to y^{l,k-1}$ uniformly in $\bar{\Omega}$. Hence, convergences \eqref{e9}, \eqref{e10}, and \eqref{AssDataGen} guarantee that $\tau R_\eps(y_\eps^{l,k-1},z_\eps)$ converges to the corresponding membrane term. By Assumption (A3), namely convergence \eqref{f0conv}, the load term converges as well. As for the inertial term, \eqref{e9} and the uniform convergence of $y_\eps^{l,k-1}$ give $z_\eps \to (y - y^{l,k-1})/\tau$ uniformly in $\bar{\Omega}$. Hence, by \eqref{AssDataGen}, $z_\eps - v_\eps^{l-1,k-1}$ converges strongly in $L^2(\Omega)$, and the inertial term converges to the corresponding membrane term. This proves \eqref{e7}.
			
			As regards the recovery sequence condition \eqref{e8}, let $(\bar{y},\bar{b}) \in B$ with $E(\bar{y},\bar{b}) < \infty$. We first note that the expected choice $\bar{y}_\eps (x',x_3) = \bar{y}(x') + \eps x_3 \bar{b}(x')$ is not allowed as $\bar{b}$ is not twice differentiable. Thus, we construct an approximate recovery sequence (see \cite{Bhattacharya,Braides,DavoliFonseca}).
			Let $w := \bar{b} - (0,0,1)^T$, so that $w \in W_0^{1,q}(\omega)$. By density of $C_c^\infty(\omega)$ in $W_0^{1,q}(\omega)$, for every $\lambda > 0$ there exists $w_\lambda \in C_c^\infty(\omega)$ with $\lVert w_\lambda - w \rVert_{W^{1,q}(\omega)} \leq \lambda$. As $w_\lambda$ vanishes together with its gradient in a neighborhood of $\partial\omega$, the functions $\bar{b}_\lambda := w_\lambda + (0,0,1)^T \in C^\infty(\bar{\omega})$ satisfy $\bar{b}_\lambda\rvert_P = (0,0,1)^T$, $\nabla' \bar{b}_\lambda\rvert_P = 0$, and $\bar{b}_\lambda \to \bar{b}$ in $W^{1,q}(\omega)$ as $\lambda \to 0$.
			We introduce the sequence $\bar{y}_{\eps\lambda}(x',x_3) = \bar{y}(x') + \eps x_3 \bar{b}_\lambda(x')$, which belongs to $W^{2,q}(\Omega)$, as $\bar{b}_\lambda$ is smooth. We show that $\bar{y}_{\eps\lambda} \in A_\eps$ for $\lambda$ and $\eps$ small enough. First, since $\bar{b}_\lambda\rvert_{\partial\omega} = (0,0,1)^T$ and $\nabla' \bar{b}_\lambda \rvert_{\partial\omega} = \nabla' w_\lambda\rvert_{\partial\omega} = 0$, together with $\nabla' \bar{y}\rvert_{\partial\omega} = \nabla'{\rm id}_0$ from $(\bar{y},\bar{b}) \in B$, we get $\nabla_\eps \bar{y}_{\eps\lambda}(x)\rvert_P = (\nabla' \bar{y}(x') + \eps x_3 \nabla' \bar{b}_\lambda(x') \vert \bar{b}_\lambda(x'))\rvert_{x' \in \partial\omega} = {\rm Id}$, and similarly $\bar{y}_{\eps\lambda}\rvert_P = {\rm id}_\eps$. Next, as $\int_{-1/2}^{1/2} x_3 \, \d x_3 = 0$, one has $\int_{-1/2}^{1/2} \bar{y}_{\eps\lambda}(x',x_3) \, \d x_3 = \bar{y}(x')$ for every $x' \in \omega$. Hence, the averaged invertibility condition in $A_\eps$ follows from the condition $\lvert \bar{y}(x') - \bar{y}(z') \rvert \geq c_0 \lvert x'-z' \rvert$ in $B$. Finally, as $E(\bar{y},\bar{b}) < \infty$, one has $J(\bar{y},\bar{b}) < \infty$. Since $q>2$, the space $W^{1,q}(\omega)$ embeds in $C^{0,\gamma}(\bar{\omega})$ with $\gamma = 1 - 2/q$, so that $\det(\nabla'\bar{y} \vert \bar{b}) \in C^{0,\gamma}(\bar{\omega})$. Moreover, Assumption (A0) gives $a > 3q/(q-3) > 2q/(q-2) = 2/\gamma$. Hence, Theorem \ref{thHealeyKromer}, applied in $\omega$ with $d = 2$ and $s = a$, provides $\bar{\mu}>0$ such that $\det(\nabla'\bar{y} \vert \bar{b}) \geq \bar{\mu}$ in $\bar{\omega}$. As $\bar{b}_\lambda \to \bar{b}$ uniformly in $\bar{\omega}$, there exists $\lambda_0>0$ such that $\det(\nabla'\bar{y} \vert \bar{b}_\lambda) \geq \bar{\mu}/2$ in $\bar{\omega}$ for every $\lambda \leq \lambda_0$. For such $\lambda$, as $\nabla'\bar{b}_\lambda$ is bounded, $\det \nabla_\eps \bar{y}_{\eps\lambda} \to \det(\nabla'\bar{y} \vert \bar{b}_\lambda)$ uniformly in $\bar{\Omega}$ as $\eps \to 0$, so that $\det \nabla_\eps \bar{y}_{\eps\lambda} \geq \bar{\mu}/4$ in $\bar{\Omega}$ for $\eps$ small enough, depending on $\lambda$. Therefore, $\bar{y}_{\eps\lambda} \in A_\eps$ for every $\lambda \leq \lambda_0$ and $\eps$ small enough. Moreover, for every fixed $\lambda \leq \lambda_0$, one has
				\begin{equation*}
					\bar{y}_{\eps\lambda} \to \bar{y} \quad \text{in} \quad W^{2,q}(\omega) \quad \text{as} \quad \eps \to 0,
				\end{equation*}
				and the scaled second gradient $\nabla_\eps^2 \bar{y}_{\eps\lambda}$ has blocks $(\nabla')^2 \bar{y} + \eps x_3 (\nabla')^2 \bar{b}_\lambda$, $\nabla' \bar{b}_\lambda$, and $0$, so that $T_\eps(\bar{y}_{\eps\lambda}) \to T(\bar{y},\bar{b}_\lambda)$ as $\eps \to 0$, and $T(\bar{y},\bar{b}_\lambda) \to T(\bar{y},\bar{b})$ as $\lambda \to 0$. Arguing as in the proof of \eqref{e7} for the remaining terms, where the convergence of $J_\eps(\bar{y}_{\eps\lambda})$, first as $\eps \to 0$ and then as $\lambda \to 0$, follows from the lower bounds on the determinants above, we obtain $\lim_{\lambda \to 0} \lim_{\eps \to 0} a_{\eps\lambda} = 0$, where
					$$
					a_{\eps\lambda} := \bigl\lvert W_\eps^{l,k-1}(\bar{y}_{\eps\lambda}) - W^{l,k-1}(\bar{y},\bar{b}) \bigr\rvert + \lVert \bar{y}_{\eps\lambda} - \bar{y} \rVert_{W^{2,q}(\Omega)} + \lVert \nabla_\eps \bar{y}_{\eps\lambda} - (\nabla' \bar{y} \vert \bar{b}) \rVert_{W^{1,q}(\Omega)}.
					$$
					Thus, in view of Attouch's diagonalization argument \cite[Corollary 1.16, p.~83]{attouch}, there exists $\lambda(\eps) \to 0$ such that $a_{\eps\lambda(\eps)} \to 0$. In particular, $W_\eps^{l,k-1}(\bar{y}_{\eps\lambda(\eps)})$ is finite, and hence $\bar{y}_{\eps\lambda(\eps)} \in A_\eps$, for $\eps$ small enough. Setting $\bar{y}_\eps = \bar{y}_{\eps\lambda(\eps)}$, we obtain \eqref{e4}, \eqref{e5}, and \eqref{e8}, whereas \eqref{e6} follows from the structure of $\nabla_\eps^2 \bar{y}_{\eps\lambda}$ above.
			
		\end{proof}

		\subsection{Convergence of minimizers}
		\label{S52}
		Let $\{y_\eps^{1,k}\}_{k=0}^M$ be a solution of the incremental problem \eqref{incrP}--\eqref{incrIC2} on $[0,h]$ in the sense of Definition \ref{defSolIncr}. Consider the first incremental step $k=1$ and the corresponding family of functionals, namely, $\{W_\eps^{1,0}\}_\eps$. An application of Theorem \ref{Gamma0h} shows that $y_\eps^{1,1}$ converges to a solution of the corresponding membrane incremental problem \eqref{MincrP}--\eqref{MincrIC2} as $\eps \to 0$. First, observe that the assumption \eqref{AssDataGen} corresponds, for $l=1$ and $k=1$, to $\nabla_\eps y_{0\eps} \to (\nabla' y_0 \vert b_0)$ in $C^{0,\alpha}(\bar{\omega})$ and $v_{0\eps} \to v_0$ in $L^2(\omega)$, which are satisfied by Assumptions (A1)--(A2), namely convergences \eqref{b0conv} and \eqref{v0conv}. Also, since $y_\eps^{1,1}$ minimizes $W_\eps^{1,0}$, one has $W_\eps^{1,0}(y_\eps^{1,1}) \leq W_\eps^{1,0}(y_{0\eps}) = E_\eps(y_{0\eps}) + \frac{\tau\rho}{2h}\lVert v_{0\eps} \rVert^2 \leq c$. Therefore, inequality \eqref{e7} and convergence \eqref{e10} hold for $y_\eps^{1,1}$. In particular, 
		$$
		W^{1,0}(y^{1,1},b^{1,1}) \leq \liminf_{\eps \to 0} W_\eps^{1,0}(y_\eps^{1,1}),
		$$
		and
		\begin{equation}
			\label{e11}
			\nabla_\eps y_\eps^{1,1} \to (\nabla' y^{1,1} \vert b^{1,1}) \quad \text{in} \quad C^{0,\alpha}(\bar{\omega}).
		\end{equation}
		Moreover, for any $(\bar{y},\bar{b}) \in B$ with $E(\bar{y},\bar{b}) < \infty$, the existence of a recovery sequence $\bar{y}_\eps \in A_\eps$ such that
		$$
		\limsup_{\eps \to 0} W_\eps^{1,0}(\bar{y}_\eps) \leq W^{1,0}(\bar{y},\bar{b})
		$$
		follows from \eqref{e8}. Combining the last two inequalities with the minimality of $y_\eps^{1,1}$, we get that
		$$
		W^{1,0}(y^{1,1},b^{1,1}) \leq W^{1,0}(\bar{y},\bar{b}) 
		$$
		holds for every $(\bar{y},\bar{b}) \in B$, trivially if $E(\bar{y},\bar{b}) = \infty$. Thus, $(y^{1,1},b^{1,1})$ is a minimizer of $W^{1,0}$. By exploiting the strong convergence \eqref{e11}, we iterate the procedure for all $k$. In particular, we get that
		$$
		W^{1,k-1}(y^{1,k},b^{1,k}) \leq W^{1,k-1}(\bar{y},\bar{b})
		$$
		holds for every $(\bar{y},\bar{b}) \in B$ and all $k=1,\dots,M$. Thus, $\{(y^{1,k},b^{1,k})\}_{k=0}^M$ is a solution to the incremental problem \eqref{MincrP}--\eqref{MincrIC2} on $[0,h]$ in the sense of Definition \ref{MdefSolIncr}.

		\subsection{Inductive iteration}
		\label{S53}
		According to \eqref{incrIC1}--\eqref{incrIC2}, the initial data for the second iteration are provided by the solution $y_\eps^1$ of the sequential problem \eqref{sequential}--\eqref{initialSequentialI} on $[0,h]$. Theorem \ref{Theorem1} guarantees the existence of such a $y_\eps^1$ which, together with $\{y_\eps^{2,k}\}_{k=0}^M$, is a solution to the problem \eqref{incrP}--\eqref{incrIC2} for $l=2$ in the sense of Definition \ref{defSolIncr}. To apply Theorem \ref{Gamma0h} for $l=2$, we prove the convergences
		\begin{align}
			&\nabla_\eps y_\eps^1(h) \to (\nabla' y^1(h) \vert b^1(h)) \quad \text{in} \quad C^{0,\alpha}(\bar{\omega}),\label{f1}\\
			&\frac{1}{\tau}\int_{(k-1)\tau}^{k\tau} \partial_t y_\eps^1(t) \; \d t  \to \frac{1}{\tau}\int_{(k-1)\tau}^{k\tau} \partial_t y^1(t) \; \d t \quad \text{in} \quad L^2(\omega), \qquad \forall k=1,\dots,M.\label{f2}
		\end{align}
		Inequality \eqref{L36} for $l=1$, combined with Young's inequality for the load term and with Assumptions (A1)--(A3), implies the bound
		\begin{equation}
			\label{f3}
			\sup_{t \in [0,h]} E_\eps(y_\eps^1(t)) + \int_0^h \Bigl( R_\eps(y_\eps^1,\partial_t y_\eps^1) + \lVert \partial_t y_\eps^1 \rVert^2 \Bigr) \; \d t \leq c,
		\end{equation}
		with $c$ independent of $\eps$. Hence, Lemma \ref{TimeComp} applies to $y_\eps^1$ on $[0,h]$ and gives, up to a subsequence (not relabeled),
		\begin{equation}
			\label{f4}
			\nabla_\eps y_\eps^1 \rightarrow (\nabla' y^1 \vert b^1) \quad \text{in} \quad C^0([0,h];C^{0,\alpha}(\bar{\omega})),
		\end{equation}
		and $y_\eps^1 \to y^1$ in $C^0([0,h];C^{1,\alpha}(\bar{\omega}))$. In particular, \eqref{f4} gives \eqref{f1}. Since $\tau$ is fixed and the averaged velocities in \eqref{f2} are the difference quotients $(y_\eps^1(k\tau) - y_\eps^1((k-1)\tau))/\tau$, the uniform convergence of $y_\eps^1$ gives \eqref{f2} as well. Moreover, Lemma \ref{TimeComp} guarantees the regularity required by Definition \ref{MdefSolIncr}, namely, $y^1 \in C_w([0,h];W^{2,q}(\omega)) \cap H^1(0,h;L^2(\omega))$, $b^1 \in C_w([0,h];W^{1,q}(\omega))$, $(\nabla' y^1 \vert b^1)^T(\nabla' y^1 \vert b^1) \in H^1(0,h;L^2(\omega))$, and $(y^1(t),b^1(t)) \in B$ for every $t \in [0,h]$.

		By \eqref{f1} and \eqref{f2}, the data of $W_\eps^{2,0}$ satisfy \eqref{AssDataGen}. Therefore, the argument of Subsection \ref{S52} shows that the limits $\{(y^{2,k},b^{2,k})\}_{k=0}^M$ solve the membrane incremental problem for $l=2$. For $l \geq 3$, the bound \eqref{f3} for $y_\eps^{l-1}$ on $[(l-2)h,(l-1)h]$ follows from \eqref{est0T4}--\eqref{est0T5}, whose constants are independent of $\eps$. Hence, the same argument applies on every outer time interval, and Theorem \ref{Theorem2} follows. 
		
		Observe that, since the proof of Theorem \ref{Theorem3} in Section \ref{S6} also starts from Lemma \ref{TimeComp}, it can be carried out along the subsequences above. Hence, the limits $(y^{l-1},b^{l-1})$ coincide with the solutions of the membrane sequential problem provided by Theorem \ref{Theorem3}, as stated in Remark \ref{RemMain}.$(i)$.

	\section{Proof of Theorem \ref{Theorem3}}
	\label{S6}
	In this section, we consider the limit as $\eps \to 0$ in equation \eqref{seqPweakEq} for the family of solutions $\{y_\eps^l\}_{l=1}^N$, whose existence is guaranteed by Theorem \ref{Theorem1}.$(ii)$. Performing this limit is subject to the convergence of the initial data $y_\eps^{l-1}((l-1)h)$ and $v_\eps^{l-1}$. By exploiting Assumptions (A1) and (A2), namely convergences \eqref{y0conv}--\eqref{v0conv}, we first consider the limit $\eps \to 0$ in equation \eqref{seqPweakEq} for $l=1$. Next, establishing the convergences $\nabla_\eps y_\eps^1(h) \to (\nabla' y^1(h) \vert b^1(h))$ in $C^{0,\alpha}(\bar{\omega})$ and $v_\eps^{1} \rightharpoonup v^1$ in $L^2(h,2h;L^2(\omega))$ as $\eps \to 0$, the procedure is repeated for $l=2$ and hence iterated for all $l$. 
	
	\subsection{Estimates and convergences of the deformations}
	\label{S61}
	By Subsection \ref{S410}, $y_\eps^1$ satisfies inequality \eqref{L36} for $l=1$. By Young's inequality, the load term is bounded by $\frac{\rho}{4h}\int_0^t \lVert \partial_t y_\eps^1 \rVert^2 \; \d s + \frac{h}{\rho}\int_0^t \lVert f_\eps \rVert^2 \; \d s$, and the first contribution can be absorbed in the left-hand side. Since the right-hand side of \eqref{L36} is bounded independently of $\eps$ by Assumptions (A1)--(A3), we obtain
	\begin{equation}
		\label{g2}
		\sup_{t \in [0,h]} E_\eps(y_\eps^1(t)) + \int_0^h \Bigl( R_\eps(y_\eps^1,\partial_t y_\eps^1) + \lVert \partial_t y_\eps^1 \rVert^2 \Bigr) \; \d t \leq c,
	\end{equation}
	with $c$ independent of $\eps$. In view of Remark \ref{RemKorn}, the bound on the dissipation in \eqref{g2} does not provide a bound on $\partial_t \nabla_\eps y_\eps^1$ independent of $\eps$. It provides, however, a bound on the time derivative of the Cauchy-Green tensor $C_\eps^1 := (\nabla_\eps y_\eps^1)^T \nabla_\eps y_\eps^1$, since $R_\eps(y_\eps^1,\partial_t y_\eps^1) = \lVert \partial_t C_\eps^1 \rVert^2$. By Lemma \ref{W2qnormEps}, estimate \eqref{g2} implies
	\begin{equation}
		\label{g1}
		\sup_{t \in [0,h]} \biggl( \lVert y_\eps^1 (t) \rVert_{W^{2,q}(\Omega)} + \lVert \nabla_\eps y_\eps^1 (t) \rVert_{W^{1,q}(\Omega)} + \lVert \nabla_\eps^2 y_\eps^1 (t) \rVert_{L^q(\Omega)}  \biggr) \leq c.
	\end{equation}
  By Lemma \ref{TimeComp}, applied on $[0,h]$, there exist $y^1$, $b^1$, and $d^1$ such that, setting $C^1 := (\nabla' y^1 \vert b^1)^T(\nabla' y^1 \vert b^1)$, convergences \eqref{conv311a}--\eqref{conv311g} read, up to non-relabeled subsequences,
	\begin{alignat}{4}
		&y_\eps^1 &&\overset{*}{\rightharpoonup} y^1 \quad &&\text{in} \quad L^{\infty}(0,h;W^{2,q}(\omega))\label{g4},\\
		&\nabla_\eps y_\eps^1 &&\overset{*}{\rightharpoonup} (\nabla' y^1 \vert b^1) \quad &&\text{in} \quad L^{\infty}(0,h;W^{1,q}(\omega)),\label{g5}\\
		&\nabla_\eps^2 y_\eps^1 &&\overset{*}{\rightharpoonup} \begin{pmatrix}
			(\nabla')^2 y^1 & \nabla' b^1\\
			(\nabla'b^1)^T & d^1
		\end{pmatrix}\quad &&\text{in} \quad L^{\infty}(0,h;L^q(\Omega)),\label{g6}\\
		&\partial_t C_\eps^1 &&\rightharpoonup \partial_t C^1 \quad &&\text{in} \quad L^2(0,h;L^2(\omega)),\label{g9}\\
		&\partial_t y_\eps^1 &&\rightharpoonup \partial_t y^1 \quad &&\text{in} \quad L^2(0,h;L^2(\omega)), \label{g10}\\
		&y_\eps^1 &&\to  y^1 \quad &&\text{in} \quad C^0([0,h];C^{1,\alpha}(\bar{\omega})),\label{g12}\\
		&\nabla_\eps y_\eps^1 &&\to (\nabla' y^1 \vert b^1) \quad &&\text{in} \quad C^0([0,h];C^{0,\alpha}(\bar{\omega})).\label{g13}
	\end{alignat}
	Moreover, by the same lemma, $y^1 \in C_w([0,h];W^{2,q}(\omega)) \cap H^1(0,h;L^2(\omega))$, $b^1 \in C_w([0,h];W^{1,q}(\omega))$, $C^1 \in H^1(0,h;L^2(\omega))$, and $(y^1(t),b^1(t)) \in B$ for every $t \in [0,h]$. Finally, the strong convergences \eqref{g12} and \eqref{g13}, together with \eqref{y0conv} and \eqref{b0conv}, guarantee $y^1(0)=y_0$ and $b^1(0)=b_0$.

	\subsection{Convergence of the equation}
	\label{S63}
	Let $(\phi_1,\phi_2) \in L^2(0,h;V_0)$, where $V_0$ is the tangent space introduced in Subsection \ref{SVar}, and let $\phi_3 \in L^2(0,h;L^q(\Omega))$. We choose $\phi_2^\eps \in L^2(0,h;C_c^\infty(\omega))$ and $\phi_3^\eps \in L^2(0,h;C_c^\infty(\Omega))$ such that, as $\eps \to 0$, one has $\phi_2^\eps \to \phi_2$ in $L^2(0,h;W^{1,q}(\omega))$, $\phi_3^\eps \to \phi_3$ in $L^2(0,h;L^q(\Omega))$, and
		$$
		\eps \bigl( \lVert \phi_2^\eps \rVert_{L^2(0,h;W^{2,q}(\omega))} + \lVert \phi_3^\eps \rVert_{L^2(0,h;W^{2,q}(\Omega))} \bigr) \to 0,
		$$
		which is possible by letting the approximation parameter tend to zero slowly enough with respect to $\eps$. We set $\Phi_3^\eps(x',x_3) := \int_{-1/2}^{x_3} (x_3 - s)\, \phi_3^\eps(x',s) \, \d s$, so that $\partial_{33}\Phi_3^\eps = \phi_3^\eps$, and
		$$
		\phi^\eps (x',x_3) := \phi_1(x') + \eps x_3 \phi^\eps_2(x') + \eps^2 \Phi^\eps_3(x',x_3).
		$$
		Since $\nabla'\phi_1 = 0$ on $P$ and $\phi_2^\eps$ and $\Phi_3^\eps$ vanish in a neighborhood of $P$, one has $\phi^\eps \in L^2(0,h;V)$. Hence, as $y_\eps^1$ solves \eqref{seqPweakEq},
		\begin{equation}
			\label{g14}
			\int_0^h \biggl( \langle D E_\eps (y_\eps^1) , \phi^\eps \rangle + \langle D_2 R_\eps (y_\eps^1,\partial_t y_\eps^1) , \phi^\eps \rangle - (f_\eps, \phi^\eps) + \frac{\rho}{h} (\partial_t y_\eps^1 - v_{0\eps} , \phi^\eps) \biggr) \; \d t = 0,
		\end{equation}
		where, as $\phi^\eps \in L^2(0,h;V)$, the duality pairings are given by the explicit expressions \eqref{defDE}--\eqref{defDR}.
	Observe that, by construction of $\phi^\eps$, one has
	\begin{alignat}{4}
		&\phi^\eps &&\to \phi_1 \quad &&\text{in} \quad L^2(0,h;W^{2,q}(\omega)\cap H_0^1(\omega))\label{g15},\\
		&\nabla_\eps \phi^\eps &&\to (\nabla' \phi_1 \vert \phi_2) \quad &&\text{in} \quad L^2(0,h;W_0^{1,q}(\omega)),\label{g16}\\
		&\nabla_\eps^2 \phi^\eps &&\to \begin{pmatrix}
			(\nabla')^2 \phi_1 & \nabla' \phi_2\\
			(\nabla'\phi_2)^T & \phi_3
		\end{pmatrix}\quad&&\text{in} \quad L^2(0,h;L^q(\Omega))\label{g17}.
	\end{alignat}
	Also, estimate \eqref{g1} guarantees
	\begin{align*}
		&\lVert \lvert \nabla_\eps^2 y_\eps^1 \rvert^{q-2} (\nabla')^2 y_\eps^1 \rVert_{L^{\infty}(0,h;L^{q'}(\Omega))} + \lVert \lvert \nabla_\eps^2 y_\eps^1 \rvert^{q-2} \eps^{-1} \nabla' \partial_3 y_\eps^1 \rVert_{L^{\infty}(0,h;L^{q'}(\Omega))}\nonumber \\ &+ \lVert \lvert \nabla_\eps^2 y_\eps^1 \rvert^{q-2} \eps^{-2}\partial_{33} y_\eps^1 \rVert_{L^{\infty}(0,h;L^{q'}(\Omega))} \leq c,
	\end{align*}
	from which we find $\pi_1^1,\pi_2^1,\pi_3^1 \in L^{\infty}(0,h;L^{q'}(\Omega))$ such that, up to non-relabeled subsequences, the following convergences hold
	\begin{alignat}{4}
		&\lvert \nabla_\eps^2 y_\eps^1 \rvert^{q-2} (\nabla')^2 y_\eps^1 &&\overset{*}{\rightharpoonup} \pi^1_1 \quad &&\text{in} \quad L^{\infty}(0,h;L^{q'}(\Omega)) \label{g19},\\
		&\lvert \nabla_\eps^2 y_\eps^1 \rvert^{q-2} \eps^{-1} \nabla' \partial_3 y_\eps^1 &&\overset{*}{\rightharpoonup} \pi^1_2 \quad &&\text{in} \quad L^{\infty}(0,h;L^{q'}(\Omega)) \label{g20},\\
		&\lvert \nabla_\eps^2 y_\eps^1 \rvert^{q-2} \eps^{-2}\partial_{33} y_\eps^1 &&\overset{*}{\rightharpoonup} \pi^1_3 \quad &&\text{in} \quad L^{\infty}(0,h;L^{q'}(\Omega)) \label{g21}.
	\end{alignat}
	
	We consider the limit as $\eps \to 0$ of each term in \eqref{g14} separately. As for the first one, we write the duality pairing explicitly as
	\begin{align*}
		&\int_0^h \langle D K_\eps (y_\eps^1), \phi^\eps \rangle \; \d t \\&= \int_0^h \int_\Omega \frac{1}{4} \biggl\{ \nabla_\eps y_\eps^1 \biggl[ \mathcal{C}\bigl((\nabla_\eps y_\eps^1)^T\nabla_\eps y_\eps^1 - {\rm Id}\bigr)^T + \mathcal{C}\bigl((\nabla_\eps y_\eps^1)^T\nabla_\eps y_\eps^1 - {\rm Id}\bigr)  \biggr]	\biggr\} \\&\qquad\qquad\quad: \biggl\{ (\nabla' \phi_1 + \eps x_3 \nabla'\phi_2^\eps + \eps^2 \nabla'\Phi_3^\eps \vert \phi_2^\eps + \eps \partial_3 \Phi_3^\eps ) \biggr\} \; \d x\; \d t.
	\end{align*}
	Then, the strong convergences \eqref{g13} and \eqref{g16} give the convergence to
	\begin{align*}
		&\int_0^h \int_\omega \frac{1}{4} \biggl\{ (\nabla' y^1\vert b^1) \biggl[ \mathcal{C}\bigl((\nabla' y^1\vert b^1)^T(\nabla' y^1 \vert b^1) - {\rm Id}\bigr)^T + \mathcal{C}\bigl((\nabla' y^1\vert b^1)^T(\nabla' y^1 \vert b^1) - {\rm Id}\bigr)  \biggr]	\biggr\} \\&\qquad\qquad\quad: \biggl\{ (\nabla' \phi_1 \vert \phi_2 ) \biggr\} \; \d x' \; \d t,
	\end{align*}
	namely,
	\begin{equation}
		\label{g25}
		\int_0^h \langle D K_\eps (y_\eps^1), \phi^\eps \rangle \; \d t \to \int_0^h \langle D_1 K (y^1,b^1), \phi_1 \rangle \; \d t + \int_0^h ( D_2 K (y^1,b^1), \phi_2 ) \; \d t.
	\end{equation}
	Analogously, we write the second term as
	\begin{align*}
		\int_0^h \langle DJ_\eps (y_\eps^1) , \phi^\eps \rangle \; \d t = -\int_0^h \int_\Omega a \frac{{\rm cof}\, \nabla_\eps y_\eps^1}{(\det \nabla_\eps y_\eps^1)^{a+1}} : (\nabla' \phi^\eps \vert \eps^{-1} \partial_3 \phi^\eps)\; \d x \; \d t.
	\end{align*}
	Since, by \eqref{g2} and Lemma \ref{PositiveDet}, $\det \nabla_\eps y_\eps^1 > \mu$ in $[0,h]\times\bar{\Omega}$ with $\mu$ independent of $\eps$, convergences \eqref{g13} and \eqref{g16} imply
	\begin{align*}
		&\int_0^h \int_\Omega a \frac{{\rm cof}\, \nabla_\eps y_\eps^1}{(\det \nabla_\eps y_\eps^1)^{a+1}} : (\nabla' \phi^\eps \vert \eps^{-1} \partial_3 \phi^\eps) \; \d x\; \d t \\&\quad\quad\to \int_0^h \int_\omega a \frac{{\rm cof}\, (\nabla' y^1 \vert b^1)}{(\det (\nabla' y^1 \vert b^1))^{a+1}} : (\nabla' \phi_1 \vert \phi_2)\; \d x' \; \d t,
	\end{align*}
	that is
	\begin{equation}
		\label{g26}
		\int_0^h \langle DJ_\eps (y_\eps^1) , \phi^\eps \rangle \; \d t \to \int_0^h \langle D_1J (y^1,b^1) , \phi_1 \rangle \; \d t + \int_0^h ( D_2J (y^1,b^1) , \phi_2 ) \; \d t.
	\end{equation}
	As for the third term, we write
		$$
		\int_0^h \langle DT_\eps (y_\eps^1) , \phi^\eps \rangle \; \d t = \int_0^h \int_\Omega \lvert \nabla_\eps^2 y_\eps^1 \rvert^{q-2} \, \nabla_\eps^2 y_\eps^1 : \nabla_\eps^2 \phi^\eps \; \d x \; \d t.
		$$
		Hence, convergences \eqref{g17}--\eqref{g21} result in
		\begin{equation}
			\label{g26b}
			\int_0^h \langle DT_\eps (y_\eps^1) , \phi^\eps \rangle \; \d t \to \int_0^h \int_\Omega \bigl( \pi_1^1 : (\nabla')^2 \phi_1 + 2 \pi_2^1 : \nabla' \phi_2 + \pi_3^1 \cdot \phi_3 \bigr) \; \d x \; \d t.
	\end{equation}
	To address the limit in the dissipation term, we observe that $\partial_t C_\eps^1 = (\nabla_\eps \partial_t y_\eps^1)^T \nabla_\eps y_\eps^1 + (\nabla_\eps y_\eps^1)^T \nabla_\eps \partial_t y_\eps^1$ is symmetric. Consequently, one gets
	\begin{align*}
		\int_0^h \langle D_2 R_\eps (y_\eps^1, \partial_t y_\eps^1) , \phi^\eps \rangle \; \d t =  \int_0^h \int_\Omega 4 \, \nabla_\eps y_\eps^1 \, \partial_t C_\eps^1 : (\nabla' \phi^\eps \vert \eps^{-1} \partial_3 \phi^\eps) \; \d x \; \d t.
	\end{align*}
	Then, the weak convergence \eqref{g9}, combined with the strong convergences \eqref{g13} and \eqref{g16}, guarantees the convergence to
	\begin{align*}
		\int_0^h \int_\omega 4 \, (\nabla' y^1 \vert b^1) \, \partial_t C^1 : (\nabla' \phi_1 \vert \phi_2) \; \d x' \; \d t,
	\end{align*}
	namely, with the convention introduced after \eqref{MdefDR},
	\begin{align}
		\label{g27}
		&\int_0^h \langle D_2 R_\eps (y_\eps^1, \partial_t y_\eps^1) , \phi^\eps \rangle \; \d t\nonumber\\ &\to \int_0^h \langle D_2 R (y^1, \partial_t y^1,b^1,\partial_t b^1) , \phi_1 \rangle \; \d t + \int_0^h ( D_4 R (y^1, \partial_t y^1,b^1,\partial_t b^1) , \phi_2 ) \; \d t.
	\end{align}
	Finally, Assumptions (A2) and (A3), namely convergences \eqref{v0conv} and \eqref{f0conv}, together with \eqref{g10} and \eqref{g15}, give
	\begin{equation}
		\label{g28}
	  \int_0^h \biggl( -(f_\eps , \phi^\eps) + \frac{\rho}{h} (\partial_t y_\eps^1 - v_{0\eps} , \phi^\eps ) \biggr) \; \d t \to \int_0^h \biggl( -(f_0 , \phi_1) + \frac{\rho}{h} (\partial_t y^1 - v_0 , \phi_1 ) \biggr) \; \d t.
	\end{equation} 
	Using convergences \eqref{g25}--\eqref{g28} in equation \eqref{g14} delivers the following triple of equations
	\begin{align}
		&\int_0^h \biggl( \langle D_1 K (y^1,b^1) , \phi_1 \rangle + \langle D_1 J (y^1,b^1) , \phi_1 \rangle + \int_\Omega \pi^1_1 : (\nabla')^2 \phi_1 \;  \d x \nonumber\\ &\qquad\;+ \langle D_2 R (y^1,\partial_t y^1,b^1,\partial_t b^1) , \phi_1 \rangle - (f_0, \phi_1)+ \frac{\rho}{h} (\partial_t y^1 - v_0 , \phi_1) \biggr) \; \d t = 0,\label{g22}\\
		&\int_0^h \biggl( ( D_2 K(y^1,b^1) , \phi_2 ) + ( D_2 J(y^1,b^1) , \phi_2 ) + 2 \int_\Omega \pi^1_2 : \nabla' \phi_2 \; \d x \nonumber\\&\qquad\;+ ( D_4 R (y^1,\partial_t y^1,b^1,\partial_t b^1) , \phi_2 ) \biggr) \; \d t = 0,\label{g23}\\
		&\int_0^h  \int_\Omega \pi^1_3 \cdot \phi_3   \; \d x \; \d t = 0,\label{g24}
	\end{align}
	which hold for every $(\phi_1,\phi_2) \in L^2(0,h;V_0)$ and $\phi_3 \in L^2(0,h;L^q(\Omega))$.

	\subsection{Identification of the limit}
	\label{S64}
	First, observe that, from the arbitrariness of $\phi_3 \in L^2(0,h;L^{q}(\Omega))$, equation \eqref{g24} implies that
	$$
	\pi_3^1(t,x) = 0 \quad \text{almost everywhere in} \quad (0,h) \times \Omega.
	$$
	Next, we proceed by identifying the limits $\pi_1^1$ and $\pi_2^1$. Consider the product space $\mathcal{G}= \mathbb{M}^{3 \times 2 \times 2} \times \mathbb{M}^{3 \times 2} \times \mathbb{R}^3$ and introduce the auxiliary function $\kappa: \mathcal{G}  \to \mathbb{R}$ given by
		$$
		\kappa(\mathbb{G},F,w) = \frac{1}{q} \bigl( \lvert \mathbb{G} \rvert^2 + 2 \lvert F \rvert^2 + \lvert w \rvert^2 \bigr)^{q/2}.
		$$
		The function $\kappa$ is convex and continuously differentiable on $\mathcal{G}$, and its partial derivatives read
		\begin{align*}
			&D_1 \kappa(\mathbb{G},F,w) = \bigl( \lvert \mathbb{G} \rvert^2 + 2 \lvert F \rvert^2 + \lvert w \rvert^2 \bigr)^{(q-2)/2}\, \mathbb{G},\\
			&D_2 \kappa(\mathbb{G},F,w) = 2\bigl( \lvert \mathbb{G} \rvert^2 + 2 \lvert F \rvert^2 + \lvert w \rvert^2 \bigr)^{(q-2)/2}\, F,\\
			&D_3 \kappa(\mathbb{G},F,w) = \bigl( \lvert \mathbb{G} \rvert^2 + 2 \lvert F \rvert^2 + \lvert w \rvert^2 \bigr)^{(q-2)/2}\, w,
		\end{align*}
		so that, for every direction $(\tilde{\mathbb{G}},\tilde{F},\tilde{w}) \in \mathcal{G}$,
		$$
		D \kappa(\mathbb{G},F,w) \cdot (\tilde{\mathbb{G}},\tilde{F},\tilde{w})  = D_1 \kappa(\mathbb{G},F,w) : \tilde{\mathbb{G}}  +  D_2 \kappa (\mathbb{G},F,w) : \tilde{F}  + D_3 \kappa(\mathbb{G},F,w) \cdot \tilde{w}.
		$$
		For brevity, we set
		$$
		\Theta_\eps^1 = \bigl((\nabla')^2 y_\eps^1,\eps^{-1}\nabla' \partial_3 y_\eps^1, \eps^{-2}\partial_{33}y_\eps^1\bigr) \quad \text{and} \quad \Theta^1 = \bigl((\nabla')^2 y^1,\nabla' b^1, d^1\bigr).
		$$
		Since $\lvert \nabla_\eps^2 y_\eps^1 \rvert^2 = \lvert (\nabla')^2 y_\eps^1 \rvert^2 + 2 \lvert \eps^{-1}\nabla'\partial_3 y_\eps^1 \rvert^2 + \lvert \eps^{-2}\partial_{33}y_\eps^1 \rvert^2$, one has
		$$
		T_\eps(y_\eps^1)= \int_\Omega \kappa(\Theta_\eps^1) \; \d x
		$$
		and
		$$
		D\kappa(\Theta_\eps^1) = \lvert \nabla_\eps^2 y_\eps^1 \rvert^{q-2} \bigl( (\nabla')^2 y_\eps^1 ,\, 2 \eps^{-1}\nabla'\partial_3 y_\eps^1 ,\, \eps^{-2}\partial_{33}y_\eps^1 \bigr).
		$$
		Hence, convergences \eqref{g19}--\eqref{g21} can be written as $D\kappa(\Theta_\eps^1) \overset{*}{\rightharpoonup} (\pi_1^1, 2\pi_2^1, \pi_3^1)$ in $L^{\infty}(0,h;L^{q'}(\Omega;\mathcal{G}))$, whereas convergence \eqref{g6} gives $\Theta_\eps^1 \overset{*}{\rightharpoonup} \Theta^1$ in $L^{\infty}(0,h;L^{q}(\Omega;\mathcal{G}))$.
		By the convexity of $\kappa$, its gradient $D\kappa$ is monotone. Consequently, for every $(\mathbb{G},F,w) \in L^q(0,h;L^q(\Omega;\mathcal{G}))$, one has
		\begin{align}
				\label{g29}
				&\int_0^h \int_\Omega  D\kappa (\Theta_\eps^1) \cdot \bigl( \Theta_\eps^1 - (\mathbb{G},F,w) \bigr) \; \d x \; \d t \nonumber\\
				&\qquad \geq \int_0^h \int_\Omega D\kappa(\mathbb{G},F,w) \cdot \bigl( \Theta_\eps^1 - (\mathbb{G},F,w) \bigr) \; \d x \;  \d t.
		\end{align}
		
		We first consider the term
		\begin{equation}
			\label{tempor}
			\int_0^h \int_\Omega D \kappa (\Theta_\eps^1) \cdot \Theta_\eps^1 \; \d x \; \d t = \int_0^h \langle DT_\eps(y_\eps^1) , y_\eps^1 \rangle \; \d t.
		\end{equation}
		As ${\rm id}_\eps$ is affine, one has $\nabla_\eps^2 {\rm id}_\eps = 0$, hence $\langle DT_\eps(y_\eps^1) , {\rm id}_\eps \rangle = 0$, and \eqref{tempor} is equal to $\int_0^h \langle DT_\eps(y_\eps^1) , y_\eps^1 - {\rm id}_\eps \rangle \; \d t$. Since $y_\eps^1 - {\rm id}_\eps$ can be used as test function in equation \eqref{seqPweakEq}, one has
		\begin{align*}
			&\int_0^h \langle DT_\eps(y_\eps^1) , y_\eps^1  - {\rm id}_\eps\rangle \; \d t  \\
			&= - \int_0^h \bigl( \langle D K_\eps(y_\eps^1) , y_\eps^1 - {\rm id}_\eps \rangle + \langle DJ_\eps(y_\eps^1), y_\eps^1 - {\rm id}_\eps \rangle \\&\qquad\quad+  \langle D_2 R_\eps (y_\eps^1,\partial_t y_\eps^1 ) , y_\eps^1 - {\rm id}_\eps \rangle - (f_\eps , y_\eps^1 - {\rm id}_\eps)+ \frac{\rho}{h} (\partial_t y_\eps^1 - v_{0\eps} , y_\eps^1 - {\rm id}_\eps) \bigr) \; \d t.
		\end{align*}
		Then, due to the strong convergences \eqref{g12} and \eqref{g13}, arguing along the lines of Subsection \ref{S63}, it follows that
		\begin{align*}
			&\int_0^h \biggl( \langle D K_\eps(y_\eps^1) , y_\eps^1 - {\rm id}_\eps \rangle + \langle DJ_\eps(y_\eps^1), y_\eps^1 - {\rm id}_\eps \rangle \nonumber\\&\qquad\quad+  \langle D_2 R_\eps (y_\eps^1,\partial_t y_\eps^1 ) , y_\eps^1 - {\rm id}_\eps \rangle - (f_\eps , y_\eps^1 - {\rm id}_\eps)+ \frac{\rho}{h} (\partial_t y_\eps^1 - v_{0\eps} , y_\eps^1 - {\rm id}_\eps) \biggr) \; \d t\nonumber\\
			&\to \int_0^h \biggl( \langle D_1 K(y^1,b^1) , y^1 - {\rm id}_0 \rangle + \langle D_1 J (y^1,b^1), y^1 - {\rm id}_0 \rangle\nonumber \\&\qquad\quad+  \langle D_2 R (y^1,\partial_t y^1,b^1,\partial_t b^1 ) , y^1 - {\rm id}_0 \rangle - (f_0 , y^1 - {\rm id}_0)+ \frac{\rho}{h} (\partial_t y^1 - v_0 , y^1 - {\rm id}_0) \biggr) \; \d t\nonumber\\
			&\quad + \int_0^h \biggl( ( D_2 K(y^1,b^1) , b^1 - (0,0,1)^T ) + ( D_2 J (y^1,b^1), b^1 - (0,0,1)^T )\nonumber \\&\qquad\quad+  ( D_4 R (y^1,\partial_t y^1,b^1,\partial_t b^1 ) , b^1 - (0,0,1)^T ) \biggr) \; \d t.
		\end{align*}
		Therefore, using equations \eqref{g22} and \eqref{g23} with $\phi_1 = y^1 - {\rm id}_0$ and $\phi_2 = b^1 - (0,0,1)^T$, and exploiting the fact that $(\nabla')^2 {\rm id}_0$ and $\nabla' (0,0,1)^T$ are null, we obtain
		\begin{equation}
			\label{g33}
			\int_0^h \langle DT_\eps(y_\eps^1) , y_\eps^1 - {\rm id}_\eps \rangle \; \d t \to \int_0^h \int_\Omega \bigl( \pi_1^1 : (\nabla')^2  y^1 + 2\pi_2^1 : \nabla' b^1 \bigr) \;  \d x \; \d t.
		\end{equation}
		Using convergences \eqref{g6}, \eqref{g19}--\eqref{g21}, and \eqref{g33} in inequality \eqref{g29}, and recalling that $\pi_3^1 = 0$, we obtain
		\begin{align*}
			0 \leq \int_0^h \int_\Omega \Bigl( &\bigl[\pi_1^1 - D_1\kappa(\mathbb{G},F,w) \bigr]  : \bigl[ (\nabla')^2 y^1- \mathbb{G} \bigr]  \nonumber\\ &+ \bigl[2\pi_2^1 - D_2\kappa(\mathbb{G},F,w) \bigr] : \bigl[ \nabla' b^1- F \bigr] \nonumber\\ &+ \bigl[\pi_3^1 - D_3\kappa(\mathbb{G},F,w)\bigr] \cdot \bigl[ d^1- w \bigr]  \Bigr) \; \d x \; \d t.
		\end{align*}
		The integrand on the right-hand side equals $\bigl((\pi_1^1,2\pi_2^1,\pi_3^1) - D\kappa(\mathbb{G},F,w)\bigr) \cdot \bigl(\Theta^1 - (\mathbb{G},F,w)\bigr)$. As $D\kappa$ is continuous and satisfies $\lvert D\kappa(\lambda) \rvert \leq c \lvert \lambda \rvert^{q-1}$ for every $\lambda \in \mathcal{G}$, the dominated convergence theorem shows that $D\kappa$ is hemicontinuous from $L^q(0,h;L^q(\Omega;\mathcal{G}))$ into its dual. Hence, Lemma \ref{LemMinty} gives $(\pi_1^1,2\pi_2^1,\pi_3^1) = D\kappa(\Theta^1)$ almost everywhere in $(0,h) \times \Omega$.
		Since $\pi_3^1 = 0$, the expression of $D_3\kappa$ gives $d^1 = 0$ almost everywhere in $(0,h) \times \Omega$. Then, as $y^1$ and $b^1$ are independent of $x_3$, the expressions of $D_1\kappa$ and $D_2\kappa$ result in
		$$
		\pi_1^1 = \bigl( \lvert (\nabla')^2 y^1 \rvert^2 + 2 \lvert \nabla' b^1 \rvert^2 \bigr)^{(q-2)/2} \, (\nabla ')^2 y^1
		$$
		and
		$$
		\pi_2^1 = \bigl( \lvert (\nabla')^2 y^1 \rvert^2 + 2 \lvert \nabla' b^1 \rvert^2 \bigr)^{(q-2)/2} \, \nabla ' b^1
		$$
		almost everywhere in $(0,h)\times \omega$. Comparing with \eqref{MdefDE}, we conclude that $\int_\Omega \pi_1^1 : (\nabla')^2 \phi_1 \; \d x = \langle D_1 T(y^1,b^1) , \phi_1 \rangle$ and $2\int_\Omega \pi_2^1 : \nabla' \phi_2 \; \d x = \langle D_2 T(y^1,b^1) , \phi_2 \rangle$ for all $\phi_1$ and $\phi_2$ as in \eqref{g22}--\eqref{g23}.

	\subsection{Improved integrability}
	\label{S65}
	The membrane deformation $(y^1,b^1)$ satisfies, for every $(\phi_1,\phi_2) \in L^2(0,h;V_0)$, the following equations
	\begin{align}
		\label{g34}
		&\int_0^h \biggl( \langle D_1 K (y^1,b^1) , \phi_1 \rangle + \langle D_1 J (y^1,b^1) , \phi_1 \rangle + \langle D_1 T(y^1,b^1) , \phi_1 \rangle \nonumber\\ &\qquad\;+ \langle D_2 R (y^1,\partial_t y^1,b^1,\partial_t b^1) , \phi_1 \rangle - (f_0, \phi_1)+ \frac{\rho}{h} (\partial_t y^1 - v_0 , \phi_1) \biggr) \; \d t = 0,
	\end{align}
	and
	\begin{align}
		\label{g35}
		&\int_0^h \biggl( ( D_2 K(y^1,b^1) , \phi_2 ) + ( D_2 J(y^1,b^1) , \phi_2 ) + \langle D_2 T(y^1,b^1) , \phi_2 \rangle \nonumber\\&\qquad\;+ ( D_4 R (y^1,\partial_t y^1,b^1,\partial_t b^1) , \phi_2 ) \biggr) \; \d t = 0.
	\end{align}
	
	Set
	\begin{align*}
		&\mathcal{D}_1^1 := D_1 K(y^1,b^1) + D_1 J(y^1,b^1) + D_2 R(y^1,\partial_t y^1, b^1 , \partial_t b^1)-f_0 + \frac{\rho}{h} (\partial_t y^1 - v_0),\\
		&\mathcal{D}_2^1 := D_2 K(y^1,b^1) + D_2 J(y^1,b^1) + D_4 R(y^1,\partial_t y^1, b^1 , \partial_t b^1),
	\end{align*}
	so that $\mathcal{D}_1^1 \in L^2(0,h;H^{-1}(\omega))$ and $\mathcal{D}_2^1 \in L^2(0,h;L^2(\omega))$. As $V_0^1 \supset C_c^\infty(\omega)$ is dense in $H_0^1(\omega)$ and $W_0^{1,q}(\omega)$ is dense in $L^2(\omega)$, Lemma \ref{LemC1}, applied with $X = V_0^1$, $Z = H_0^1(\omega)$, $\xi = D_1 T(y^1,b^1)$, and $\zeta = \mathcal{D}_1^1$, and with $X = W_0^{1,q}(\omega)$, $Z = L^2(\omega)$, $\xi = D_2 T(y^1,b^1)$, and $\zeta = \mathcal{D}_2^1$, extends \eqref{g34} to every $\phi_1 \in L^2(0,h;H_0^1(\omega))$ and \eqref{g35} to every $\phi_2 \in L^2(0,h;L^2(\omega))$, namely,
	\begin{align}
		\label{g36}
		&\int_0^h \biggl( \langle D_1 K (y^1,b^1) , \phi_1 \rangle + \langle D_1 J (y^1,b^1) , \phi_1 \rangle + \langle D_1 T(y^1,b^1) , \phi_1 \rangle_{H_0^1} \nonumber\\ &\qquad\;+ \langle D_2 R (y^1,\partial_t y^1,b^1,\partial_t b^1) , \phi_1 \rangle - (f_0, \phi_1)+ \frac{\rho}{h} (\partial_t y^1 - v_0 , \phi_1) \biggr) \; \d t = 0,
	\end{align}
	and
	\begin{align}
		\label{g37}
		&\int_0^h \biggl( ( D_2 K(y^1,b^1) , \phi_2 ) + ( D_2 J(y^1,b^1) , \phi_2 ) + ( D_2 T(y^1,b^1) , \phi_2 ) \nonumber\\&\qquad\;+ ( D_4 R (y^1,\partial_t y^1,b^1,\partial_t b^1) , \phi_2 ) \biggr) \; \d t = 0,
	\end{align}
	which are \eqref{MseqPweakEq} and \eqref{MseqPweakEq2} for $l=1$.
		
	\subsection{Inductive iteration}
	\label{S66}
	
	Consider the sequential problem \eqref{sequential}--\eqref{initialSequentialI} for $l=2$. Theorem \ref{Theorem1}.$(ii)$ guarantees the existence of a weak solution $y_\eps^2$ in the sense of Definition \ref{defSolSeq}. In particular, $y_\eps^2$ satisfies, for any $\phi \in L^2(h,2h;H_0^1(\Omega))$,
	\begin{align}
		&\int_h^{2h} \biggl( \langle D E_\eps (y_\eps^2) , \phi \rangle_{H_0^1} + \langle D_2 R_\eps (y_\eps^2,\partial_t y_\eps^2) , \phi \rangle \nonumber \\&\qquad\qquad \qquad\qquad - (f_\eps, \phi) + \frac{\rho}{h} (\partial_t y_\eps^2 - v_\eps^1 , \phi) \biggr) \; \d t = 0,\label{g41c}\\
		&y_\eps^2(h)=y_\eps^1(h) \qquad \text{and} \qquad v_\eps^1(\cdot) = \partial_t y_\eps^1(\cdot - h) . \label{g41d}
	\end{align}
	Passing to the limit $\eps \to 0$ in \eqref{g41c}--\eqref{g41d} then requires the convergence of $y_\eps^1(h)$ and $\partial_t y_\eps^1(\cdot - h)$, which is guaranteed by \eqref{g13} and \eqref{g10}. Indeed, one has
	\begin{align}
		&\nabla_\eps y_\eps^1(h) \to (\nabla' y^1(h)\vert b^1(h)) \quad \text{in} \quad C^{0,\alpha}(\bar{\omega}),\label{g41b}\\
		&\partial_t y_\eps^1(\cdot - h) \rightharpoonup  \partial_t y^1(\cdot - h) \quad \text{in} \quad L^2(h,2h;L^2(\omega))\label{g41}.
	\end{align}
	
	Next, inequality \eqref{L36} for $l=2$, together with $E_\eps(y_\eps^2(h)) = E_\eps(y_\eps^1(h)) \leq c$ and $\int_h^{2h} \lVert v_\eps^1 \rVert^2 \; \d t \leq c$ by \eqref{g2}, delivers the analogue of \eqref{g2} on $[h,2h]$. Hence, Lemma \ref{TimeComp} applies to $y_\eps^2$ on $[h,2h]$ and gives the analogues of \eqref{g4}--\eqref{g13}, with limits $y^2$, $b^2$, and $d^2$. In particular, since $y_\eps^2(h) = y_\eps^1(h)$, the uniform convergences \eqref{g12}--\eqref{g13} and their analogues on $[h,2h]$ give $y^2(h) = y^1(h)$ and $b^2(h) = b^1(h)$. These convergences, combined with \eqref{g41b}--\eqref{g41}, allow us to repeat the arguments of Subsections \ref{S63}--\ref{S65} on $[h,2h]$. Observe that, for $l=1$, the initial velocity $v_{0\eps}$ converges strongly by \eqref{v0conv}, whereas, for $l=2$, it corresponds to the time-dependent term $v_\eps^1$, which converges only weakly in $L^2(h,2h;L^2(\omega))$. Since $v_\eps^1$ enters \eqref{g41c} linearly, the weak convergence is enough, also in the monotonicity argument of Subsection \ref{S64}, where it is tested with the strongly converging sequence $y_\eps^2 - {\rm id}_\eps$.
	
	Finally, Theorem \ref{Theorem3} follows by iterating this procedure for all $l$.
	
	\section{Proof of Theorem \ref{Theorem5}}
		\label{ProofT5}
		This section is devoted to the proof of Theorem \ref{Theorem5}. We consider the limit as $\eps \to 0$ in equation \eqref{m1}, satisfied together with \eqref{m2} by the family of solutions $\{ y_\eps^{(h)} \}_h$ provided by Theorem \ref{Theorem1}.$(iii)$. Since no iteration over $l$ is involved here, this amounts to repeating, over the whole interval $[0,T]$, the argument used in the proof of Theorem \ref{Theorem3} (Section \ref{S6}) for the first sequential step $y_\eps^1$ on $[0,h]$. We therefore keep the presentation brief, referring back to Section \ref{S6} for the details of each step.
		
		\subsection{Estimates and convergences of the deformations}
		\label{S81}
		By estimates \eqref{est0T4} and \eqref{est0T5}, whose right-hand sides are bounded independently of $\eps$ by Assumptions (A1)--(A3), the family $\{y_\eps^{(h)}\}_\eps$ satisfies \eqref{TCbound} on $[0,T]$, with a constant independent of $\eps$ and $h$. Hence, Lemma \ref{W2qnormEps} gives
		\begin{align*}
			& \sup_{t \in [0,T]} \biggl( \lVert y_\eps^{(h)} (t) \rVert_{W^{2,q}(\Omega)} + \lVert \nabla_\eps y_\eps^{(h)} (t) \rVert_{W^{1,q}(\Omega)} + \lVert \nabla_\eps^2 y_\eps^{(h)} (t) \rVert_{L^q(\Omega)}  \biggr) \leq c,
		\end{align*}
		and Lemma \ref{TimeComp}, applied on $[0,T]$, provides $y^{(h)}$, $b^{(h)}$, and $d^{(h)}$ such that, setting $C^{(h)} := (\nabla' y^{(h)} \vert b^{(h)})^T(\nabla' y^{(h)} \vert b^{(h)})$ and $C_\eps^{(h)} := (\nabla_\eps y_\eps^{(h)})^T\nabla_\eps y_\eps^{(h)}$, the following convergences hold up to non-relabeled subsequences
		\begin{alignat}{4}
			&y_\eps^{(h)} &&\overset{*}{\rightharpoonup} y^{(h)} \quad &&\text{in} \quad L^{\infty}(0,T;W^{2,q}(\omega))\label{i4},\\
			&\nabla_\eps y_\eps^{(h)} &&\overset{*}{\rightharpoonup} (\nabla' y^{(h)} \vert b^{(h)}) \quad &&\text{in} \quad L^{\infty}(0,T;W^{1,q}(\omega)),\label{i5}\\
			&\nabla_\eps^2 y_\eps^{(h)} &&\overset{*}{\rightharpoonup} \begin{pmatrix}
				(\nabla')^2 y^{(h)} & \nabla' b^{(h)}\\
				(\nabla'b^{(h)})^T & d^{(h)}
			\end{pmatrix}\quad&&\text{in} \quad L^{\infty}(0,T;L^q(\Omega))\label{i6},\\
			&\partial_t y_\eps^{(h)} &&\rightharpoonup \partial_t y^{(h)} \quad &&\text{in} \quad L^2(0,T;L^2(\omega)),\label{i7}\\
			&\partial_t C_\eps^{(h)} &&\rightharpoonup \partial_t C^{(h)} \quad &&\text{in} \quad L^2(0,T;L^2(\omega)),\label{i9}\\
			&y_\eps^{(h)} &&\to  y^{(h)} \quad &&\text{in} \quad C^0([0,T];C^{1,\alpha}(\bar{\omega})),\label{i12}\\
			&\nabla_\eps y_\eps^{(h)} &&\to (\nabla' y^{(h)} \vert b^{(h)}) \quad &&\text{in} \quad C^0([0,T];C^{0,\alpha}(\bar{\omega})).\label{i13}
		\end{alignat}
		Moreover, the bound on $\nabla_\eps^2 y_\eps^{(h)}$ provides $\pi_1^{(h)},\pi_2^{(h)},\pi_3^{(h)} \in L^{\infty}(0,T;L^{q'}(\Omega))$ such that, up to non-relabeled subsequences,
		  \begin{alignat}{4}
			&\lvert \nabla_\eps^2 y_\eps^{(h)} \rvert^{q-2} (\nabla')^2 y_\eps^{(h)} &&\overset{*}{\rightharpoonup} \pi^{(h)}_1 \quad &&\text{in} \quad L^{\infty}(0,T;L^{q'}(\Omega)),\label{i19}\\
			&\lvert \nabla_\eps^2 y_\eps^{(h)} \rvert^{q-2} \eps^{-1} \nabla' \partial_3 y_\eps^{(h)} &&\overset{*}{\rightharpoonup} \pi^{(h)}_2 \quad &&\text{in} \quad L^{\infty}(0,T;L^{q'}(\Omega)),\label{i20}\\
			&\lvert \nabla_\eps^2 y_\eps^{(h)} \rvert^{q-2} \eps^{-2}\partial_{33} y_\eps^{(h)} &&\overset{*}{\rightharpoonup} \pi^{(h)}_3 \quad &&\text{in} \quad L^{\infty}(0,T;L^{q'}(\Omega)).\label{i21}
		\end{alignat}
		Finally, Lemma \ref{TimeComp} guarantees $b^{(h)} \in C_w([0,T];W^{1,q}(\omega))$, $C^{(h)} \in H^1(0,T;L^2(\omega))$, $y^{(h)} \in C_w([0,T];W^{2,q}(\omega)) \cap H^1(0,T;L^2(\omega))$, and $(y^{(h)}(t),b^{(h)}(t)) \in B$ for every $t \in [0,T]$. The strong convergences \eqref{i12}--\eqref{i13}, together with \eqref{y0conv} and \eqref{b0conv}, guarantee $y^{(h)}(0)=y_0$ and $b^{(h)}(0)=b_0$.

		\subsection{Convergence of the equation}
		\label{S83}
		Fix test functions $(\phi_1,\phi_2) \in L^2(0,T;V_0)$ and $\phi_3\in L^2(0,T;L^q(\Omega))$. We consider in \eqref{m1} the functions $\phi^\eps$ constructed from $\phi_1$, $\phi_2$, and $\phi_3$ as in Subsection \ref{S63}, with $[0,h]$ replaced by $[0,T]$. They take values in $V$, so that the duality pairings in \eqref{m1} are given by the explicit expressions \eqref{defDE}--\eqref{defDR}, and they satisfy
		\begin{alignat}{4}
			&\phi^\eps &&\to \phi_1 \quad &&\text{in} \quad L^2(0,T;W^{2,q}(\omega)\cap H_0^1(\omega))\label{i15},\\
			&\nabla_\eps \phi^\eps &&\to (\nabla' \phi_1 \vert \phi_2) \quad &&\text{in} \quad L^2(0,T;W_0^{1,q}(\omega)),\label{i16}\\
			&\nabla_\eps^2 \phi^\eps &&\to \begin{pmatrix}
				(\nabla')^2 \phi_1 & \nabla' \phi_2\\
				(\nabla'\phi_2)^T & \phi_3
			\end{pmatrix}\quad&&\text{in} \quad L^2(0,T;L^q(\Omega)).\label{i17}
		\end{alignat}
		Passing to the limit term by term in \eqref{m1}, tested against $\phi^\eps$, is done exactly as in Subsection \ref{S63}, using \eqref{i4}--\eqref{i17} and Assumption (A3), namely convergence \eqref{f0conv}, in place of the corresponding convergences and assumptions for $y_\eps^1$. This delivers
		\begin{align}
			&\int_0^T \biggl( \langle D_1 K (y^{(h)},b^{(h)}) , \phi_1 \rangle + \langle D_1 J (y^{(h)},b^{(h)}) , \phi_1 \rangle   \nonumber\\ &\qquad\; + \int_\Omega \pi_1^{(h)} : (\nabla')^2 \phi_1 \; \d x + \langle D_2 R (y^{(h)},\partial_t y^{(h)},b^{(h)},\partial_t b^{(h)}) , \phi_1 \rangle  \nonumber\\&\qquad\;- (f_0, \phi_1) + \frac{\rho}{h} (\partial_t y^{(h)} - \partial_t y^{(h)}(\cdot - h) , \phi_1) \biggr) \; \d t = 0,\label{i22}\\
			&\int_0^T \biggl( ( D_2 K(y^{(h)},b^{(h)}) , \phi_2 ) + ( D_2 J(y^{(h)},b^{(h)}) , \phi_2 ) + 2\int_\Omega \pi^{(h)}_2 : \nabla' \phi_2 \; \d x \nonumber\\&\qquad\;+ ( D_4 R (y^{(h)},\partial_t y^{(h)},b^{(h)},\partial_t b^{(h)}) , \phi_2 ) \biggr) \; \d t = 0,\label{i23}\\
			&\int_0^T  \int_\Omega \pi^{(h)}_3 \cdot \phi_3  \; \d x \;  \d t = 0,\label{i24}
		\end{align}
		for every $\phi_1,\phi_2,\phi_3$ as above.
		
		Concerning the initial velocity, since $y_\eps^{(h)}$ solves the time-delayed problem \eqref{tDelay}--\eqref{initialVelocity} in the sense of Definition \ref{defSoltimedelayedT}, one has $\partial_t y_\eps^{(h)}(t-h) = v_{0\eps}$ for a.e. $t \in (0,h)$. Hence, Assumption (A2), namely convergence \eqref{v0conv}, gives $\partial_t y^{(h)}(t-h) = v_0$ for a.e. $t \in (0,h)$, whereas, for $t \in (h,T)$, the shifted velocities converge weakly by \eqref{i7}.
		
		\subsection{Identification of the limit}
		\label{S84}
		The identification of $\pi_1^{(h)}, \pi_2^{(h)}, \pi_3^{(h)}$ follows the monotonicity argument of Subsection \ref{S64}, with $y_\eps^1$, $b^1$, and $[0,h]$ replaced by $y_\eps^{(h)}$, $b^{(h)}$, and $[0,T]$. Testing \eqref{i24} with an arbitrary $\phi_3$ gives
		$$
		\pi_3^{(h)}(t,x) = 0 \quad \text{for a.e. } (t,x) \in (0,T)\times\Omega.
		$$
		Using $y_\eps^{(h)} - {\rm id}_\eps$ as a test function in \eqref{m1}, and $y^{(h)} - {\rm id}_0$ and $b^{(h)} - (0,0,1)^T$ as test functions in \eqref{i22}--\eqref{i23}, the monotonicity of $D\kappa$, with $\kappa$ as in Subsection \ref{S64}, gives $d^{(h)} = 0$ almost everywhere in $(0,T)\times\Omega$,
		$$
		\pi_1^{(h)} = \bigl( \lvert (\nabla')^2 y^{(h)} \rvert^2 + 2 \lvert \nabla' b^{(h)} \rvert^2 \bigr)^{(q-2)/2} \, (\nabla')^2 y^{(h)},
		$$
		and
		$$
		\pi_2^{(h)} = \bigl( \lvert (\nabla')^2 y^{(h)} \rvert^2 + 2 \lvert \nabla' b^{(h)} \rvert^2 \bigr)^{(q-2)/2} \, \nabla' b^{(h)}
		$$
		almost everywhere in $(0,T)\times\omega$.

		\subsection{Improved integrability}
		\label{S85}
		Substituting these identifications into \eqref{i22}--\eqref{i23} gives
		\begin{align}
			&\int_0^T \biggl( \langle D_1 K (y^{(h)},b^{(h)}) , \phi_1 \rangle + \langle D_1 J (y^{(h)},b^{(h)}) , \phi_1 \rangle + \langle D_1 T(y^{(h)},b^{(h)}) , \phi_1 \rangle \nonumber\\ &\qquad\;+ \langle D_2 R (y^{(h)},\partial_t y^{(h)},b^{(h)},\partial_t b^{(h)}) , \phi_1 \rangle - (f_0, \phi_1) \nonumber\\ &\qquad\;+ \frac{\rho}{h} (\partial_t y^{(h)} - \partial_t y^{(h)}(\cdot-h) , \phi_1) \biggr) \; \d t = 0,\label{i34}\\
			&\int_0^T \biggl( ( D_2 K(y^{(h)},b^{(h)}) , \phi_2 ) + ( D_2 J(y^{(h)},b^{(h)}) , \phi_2 ) + \langle D_2 T(y^{(h)},b^{(h)}) , \phi_2 \rangle \nonumber\\&\qquad\;+ ( D_4 R (y^{(h)},\partial_t y^{(h)},b^{(h)},\partial_t b^{(h)}) , \phi_2 ) \biggr) \; \d t = 0,\label{i35}
		\end{align}
		for every $(\phi_1,\phi_2) \in L^2(0,T;V_0)$.
		Arguing as in Subsection \ref{S65}, Lemma \ref{LemC1} extends \eqref{i34} to every $\phi_1 \in L^2(0,T;H_0^1(\omega))$ and \eqref{i35} to every $\phi_2 \in L^2(0,T;L^2(\omega))$, as required by Definition \ref{MdefSoltimedelayedT}.
		
		This concludes the proof of Theorem \ref{Theorem5}.


	\subsection{Proof of Corollary \ref{CorConcat}}
	\label{S86}
	We first apply Theorem \ref{Theorem3} to $\{y_\eps^l\}_{l=1}^N$, and then the argument above to $y_\eps^{(h)}$ along the resulting subsequence, extracting a further subsequence (not relabeled). Along it, for every $l=1,\dots,N$, the gradients $\nabla_\eps y_\eps^{(h)}$ and $\nabla_\eps y_\eps^l$ coincide in $[(l-1)h,lh]\times\Omega$ and converge uniformly there to $(\nabla' y^{(h)} \vert b^{(h)})$ and $(\nabla' y^l \vert b^l)$, respectively. Hence, $(\nabla' y^{(h)} \vert b^{(h)}) = (\nabla' y^l \vert b^l)$ in $[(l-1)h,lh]\times\bar{\omega}$. Similarly, the uniform convergence \eqref{i12} and the analogue of \eqref{g12} on $[(l-1)h,lh]$ give $y^{(h)} = y^l$ in $[(l-1)h,lh]\times\bar{\omega}$. This concludes the proof of Corollary \ref{CorConcat}.
	
	\section{Proof of Theorem \ref{Theorem6}}
	\label{ProofT6}
	This section is devoted to the proof of Theorem \ref{Theorem6}. We perform the limit as $h \to 0$ for the family $\{(y^{(h)},b^{(h)})\}_h$ provided by Theorem \ref{Theorem5}. The estimates independent of $h$ are inherited from the bulk ones by lower semicontinuity. In contrast with Subsection \ref{S412}, the limit of the second-order stress is not identified by a monotonicity argument. Indeed, that argument relies on the convergence \eqref{quadterm} of a term which is quadratic in the velocity. This convergence follows from the weak convergence of the velocities $\partial_t y_\eps^{(h)}$ in $L^2(0,T;H_0^1(\Omega))$, given by Lemma \ref{H1normEps}, and the strong convergence of the difference quotients $\mathcal{Q}_\eps^{(h)}$ in $C^0([0,T];H^{-1}(\Omega))$. At the membrane level, the velocities $\partial_t y^{(h)}$ are bounded only in $L^2(0,T;L^2(\omega))$, since the dissipation controls only the time derivative of the Cauchy-Green tensor (see Remark \ref{RemKorn}), so that the analogue of \eqref{quadterm} is not available. For this reason, the limit of the second-order stress is described by the Young measure generated by the second gradients.
	
	\subsection{Estimates independent of $h$}
	\label{S91}
		For every $h$, let $y_\eps^{(h)}$ be the global interpolant provided by Theorem \ref{Theorem1}.$(iii)$. As observed in Subsection \ref{SGlob}, the estimates of Subsection \ref{S410} hold with $m = N$. In particular, inequality \eqref{EIeps} holds for every $t \in [0,T]$.
		
		Fix $h$ and let $\eps \to 0$ along the subsequence of Section \ref{ProofT5}. For every $t \in [0,T]$, Lemma \ref{ClosAB} applies to $\{y_\eps^{(h)}(t)\}_\eps$, and, by \eqref{i12} and \eqref{i13}, the corresponding limit is $(y^{(h)}(t),b^{(h)}(t))$ along every subsequence. Hence,
			$$
			E(y^{(h)}(t),b^{(h)}(t)) \leq \liminf_{\eps\to0} E_\eps(y_\eps^{(h)}(t)), \qquad \forall t \in [0,T].
			$$
			The weak convergences \eqref{i7} and \eqref{i9}, with $R_\eps(y_\eps^{(h)},\partial_t y_\eps^{(h)}) = \lVert \partial_t C_\eps^{(h)} \rVert^2$, the weak lower semicontinuity of the norm, Assumptions (A2)--(A3), namely convergences \eqref{v0conv} and \eqref{f0conv}, and $\limsup_{\eps\to0} E_\eps(y_{0\eps}) \leq \mathfrak{E}_0$ then allow us to pass to the limit in \eqref{EIeps}. We obtain, for every $t \in [0,T]$,
			\begin{align}
				\label{EIh}
				E (y^{(h)}(t),b^{(h)}(t)) &+ 2 \int_0^t \lVert \partial_t C^{(h)} \rVert^2 \; \d s + \frac{\rho}{2h}\int_{t-h}^t \lVert \partial_t y^{(h)} \rVert^2 \; \d s \nonumber\\
				&\leq \mathfrak{E}_0 + \int_0^t (f_0 , \partial_t y^{(h)}) \; \d s + \frac{\rho}{2}\lVert v_0 \rVert^2,
			\end{align}
			where $\partial_t y^{(h)} := v_0$ in $(-h,0)$, in accordance with Definition \ref{MdefSoltimedelayedT}. In the same way, estimates \eqref{est0T4} and \eqref{est0T5}, whose constants are independent of $\eps$ and $h$, deliver
			\begin{equation}
				\label{unifh}
				\sup_{t \in [0,T]} E(y^{(h)}(t),b^{(h)}(t)) + \int_0^T \Bigl( \lVert \partial_t C^{(h)} \rVert^2 + \lVert \partial_t y^{(h)} \rVert^2 \Bigr) \; \d t \leq c,
			\end{equation}
			with $c$ independent of $h$.
	
	\subsection{Compactness and Young measure}
	\label{S92}
	By Theorem \ref{Theorem5} and \eqref{unifh}, Remark \ref{RemTimeComp} applies to the family $\{(y^{(h)},b^{(h)})\}_h$ on $[0,T]$. Hence, there exist $y \in C_w([0,T];W^{2,q}(\omega)) \cap H^1(0,T;L^2(\omega))$ and $b \in C_w([0,T];W^{1,q}(\omega))$ such that $(y(t),b(t)) \in B$ for every $t \in [0,T]$, $\det(\nabla' y \vert b) \geq \mu$ in $[0,T]\times\bar{\omega}$, the Cauchy-Green tensor $C := (\nabla' y \vert b)^T(\nabla' y \vert b)$ belongs to $H^1(0,T;L^2(\omega))$, and, up to a subsequence (not relabeled),
	\begin{alignat}{4}
		&y^{(h)} &&\overset{*}{\rightharpoonup} y \quad &&\text{in} \quad L^\infty(0,T;W^{2,q}(\omega)),\label{j1}\\
		&b^{(h)} &&\overset{*}{\rightharpoonup} b \quad &&\text{in} \quad L^\infty(0,T;W^{1,q}(\omega)),\label{j2}\\
		&\partial_t y^{(h)} &&\rightharpoonup \partial_t y \quad &&\text{in} \quad L^2(0,T;L^2(\omega)),\label{j3}\\
		&\partial_t C^{(h)} &&\rightharpoonup \partial_t C \quad &&\text{in} \quad L^2(0,T;L^2(\omega)),\label{j4}\\
		&y^{(h)} &&\to y \quad &&\text{in} \quad C^0([0,T];C^{1,\alpha}(\bar{\omega})),\label{j5}\\
		&b^{(h)} &&\to b \quad &&\text{in} \quad C^0([0,T];C^{0,\alpha}(\bar{\omega})).\label{j6}
	\end{alignat}
	Since $y^{(h)}(0) = y_0$ and $b^{(h)}(0) = b_0$ for every $h$, the uniform convergences \eqref{j5} and \eqref{j6} guarantee $y(0) = y_0$ and $b(0) = b_0$.
	
		Set $\Theta^{(h)} := \bigl((\nabla')^2 y^{(h)},\nabla' b^{(h)}\bigr)$ and $\Theta := \bigl((\nabla')^2 y,\nabla' b\bigr)$. By \eqref{unifh}, $\Theta^{(h)}$ is bounded in $L^\infty(0,T;L^q(\omega;\mathcal{G}_0))$. Thus, by the classical theory of Young measures (see Subsection \ref{NotSol}), a further subsequence (not relabeled) of $\{\Theta^{(h)}\}_h$ generates a Young measure $\nu$ such that
	\begin{gather*}
		\esssup_{t\in(0,T)} \int_\omega \langle \nu_{t,x'} , \lvert \lambda \rvert^q \rangle \; \d x' < \infty,\\
		\Theta^{(h)} \overset{*}{\rightharpoonup} \langle \nu_{t,x'} , \lambda \rangle \quad \text{in} \quad L^\infty(0,T;L^q(\omega;\mathcal{G}_0)).
	\end{gather*}
	In view of \eqref{j1} and \eqref{j2}, we infer $\langle \nu_{t,x'} , \lambda \rangle = \Theta(t,x')$ for a.e. $(t,x') \in \omega_T$. Moreover, $\lvert D\kappa_0(\lambda) \rvert \leq c \lvert \lambda \rvert^{q-1}$, so that $D\kappa_0(\Theta^{(h)})$ is bounded in $L^\infty(0,T;L^{q'}(\omega;\mathcal{G}_0))$, and hence equiintegrable. Similarly to the convergence of the barycenter, we obtain
	\begin{equation}
		\label{j7}
		D\kappa_0(\Theta^{(h)}) \overset{*}{\rightharpoonup} \langle \nu_{t,x'} , D\kappa_0 \rangle \quad \text{in} \quad L^\infty(0,T;L^{q'}(\omega;\mathcal{G}_0)),
	\end{equation}
	namely, the weak-$*$ limits of $D_1\kappa_0(\Theta^{(h)})$ and $D_2\kappa_0(\Theta^{(h)})$ are $\langle \nu_{t,x'} , D_1\kappa_0 \rangle$ and $\langle \nu_{t,x'} , D_2\kappa_0 \rangle$, respectively.

	\subsection{Passage to the limit in the equations}
	\label{S93}
	Let $\Psi = \psi \, \phi$ with $\psi$ and $\phi$ as in Lemma \ref{TOT}, with $\mathcal{D} = \omega$ and $X = V_0^1$. Since $(y^{(h)},b^{(h)})$ is a weak solution of the membrane time-delayed problem, \eqref{MtimeDweakEw} holds for $\Psi$. As $\Psi$ takes values in $V_0^1$, by \eqref{MdefDE} the term involving $D_1 T$ reads
	$$
	\int_0^T \langle D_1 T(y^{(h)},b^{(h)}) , \Psi \rangle \; \d t = \int_0^T \int_\omega D_1\kappa_0(\Theta^{(h)}) : (\nabla')^2 \Psi \; \d x' \, \d t,
	$$
	which converges to $\int_0^T \int_\omega \langle \nu_{t,x'} , D_1\kappa_0 \rangle : (\nabla')^2 \Psi \; \d x' \, \d t$ by \eqref{j7}. The uniform convergences \eqref{j5} and \eqref{j6}, together with $\det(\nabla' y^{(h)} \vert b^{(h)}) \geq \mu$, guarantee the convergence of the terms involving $D_1 K$ and $D_1 J$, as in Subsection \ref{S63}. The dissipation term reads
	$$
	\int_0^T \langle D_2 R(y^{(h)},\partial_t y^{(h)},b^{(h)},\partial_t b^{(h)}) , \Psi \rangle \; \d t = \int_0^T \int_\omega 4 \, (\nabla' y^{(h)} \vert b^{(h)}) \, \partial_t C^{(h)} : (\nabla' \Psi \vert 0) \; \d x' \, \d t,
	$$
	and it converges by \eqref{j4}--\eqref{j6}. As for the inertial term, the argument of Subsection \ref{S412} relies only on the weak convergence of the velocities in $L^2(0,T;L^2(\omega))$ and on the initial condition on the velocity. Hence, \eqref{j3} and $\partial_t y^{(h)}(t-h) = v_0$ for a.e. $t \in (0,h)$ give
	$$
	\frac{\rho}{h} \int_0^T \bigl(\partial_t y^{(h)} - \partial_t y^{(h)}(\cdot - h) , \Psi\bigr) \; \d t \to - \rho \int_0^T (\partial_t y , \partial_t \Psi) \; \d t - \rho (v_0,\Psi(0)).
	$$
	Consequently, \eqref{MmvEq1} holds for $\Psi = \psi\,\phi$. Since each term in \eqref{MmvEq1} is continuous with respect to the topology of $L^1(0,T;W^{2,q}(\omega)) \cap H^1(0,T;H_0^1(\omega))$, Lemma \ref{TOT} extends \eqref{MmvEq1} to every $\Psi \in L^1(0,T;V_0^1) \cap H^1(0,T;H_0^1(\omega))$ such that $\Psi(T) = 0$.
	
	Next, let $\Psi \in L^1(0,T;W_0^{1,q}(\omega)) \cap L^2(0,T;L^2(\omega))$. By \eqref{MtimeDweakEw2} and \eqref{MdefDE}, one has
	\begin{align*}
		\int_0^T \biggl( ( D_2 K(y^{(h)},b^{(h)}) , \Psi ) &+ ( D_2 J(y^{(h)},b^{(h)}) , \Psi ) + \int_\omega D_2\kappa_0(\Theta^{(h)}) : \nabla' \Psi \; \d x' \\
		&+ ( D_4 R(y^{(h)},\partial_t y^{(h)},b^{(h)},\partial_t b^{(h)}) , \Psi ) \biggr) \; \d t = 0.
	\end{align*}
	Arguing as above, the convergences \eqref{j4}--\eqref{j6} and \eqref{j7} allow us to pass to the limit in each term, which proves \eqref{MmvEq2}.
	
	\subsection{Energy inequality}
	\label{S94}
		Let $\varphi \in C_c^0((0,T))$ be nonnegative. We multiply \eqref{EIh} by $\varphi(t)$, integrate over $(0,T)$, and let $h \to 0$. The uniform convergences \eqref{j5} and \eqref{j6} and the bound $\det(\nabla' y^{(h)} \vert b^{(h)}) \geq \mu$ give $K(y^{(h)}(t),b^{(h)}(t)) \to K(y(t),b(t))$ and $J(y^{(h)}(t),b^{(h)}(t)) \to J(y(t),b(t))$ uniformly in $t \in [0,T]$. Also, \eqref{YMlsc} with $f = \kappa_0$ gives
			$$
			\liminf_{h\to0} \int_0^T \varphi \, T(y^{(h)},b^{(h)}) \; \d t \geq \int_0^T \varphi \int_\omega \langle \nu_{t,x'} , \kappa_0 \rangle \; \d x' \, \d t.
			$$
		For every $t \in [0,T]$, \eqref{j4} and the weak lower semicontinuity of the norm give $\liminf_{h\to0} \int_0^t \lVert \partial_t C^{(h)} \rVert^2 \, \d s \geq \int_0^t R(y,\partial_t y,b,\partial_t b) \, \d s$, and Fatou's lemma applies to the dissipation term. As for the kinetic term, whose averaging interval $(t-h,t)$ shrinks as $h \to 0$, we set $\varphi_h(s) := \frac{1}{h}\int_s^{s+h} \varphi(t) \, \d t$, with $\varphi$ extended by zero outside $(0,T)$, and we observe that $\varphi_h = 0$ in $(-h,0)$ for $h$ small enough. By Fubini's theorem,
		$$
		\int_0^T \varphi(t) \, \frac{\rho}{2h}\int_{t-h}^t \lVert \partial_t y^{(h)} \rVert^2 \; \d s \, \d t = \frac{\rho}{2}\int_0^T \varphi_h \, \lVert \partial_t y^{(h)} \rVert^2 \; \d s.
		$$
		Since $\varphi_h \to \varphi$ uniformly and $\partial_t y^{(h)}$ is bounded in $L^2(0,T;L^2(\omega))$, the weak convergence \eqref{j3} and the weak lower semicontinuity of the norm deliver
		$$
		\liminf_{h\to0} \frac{\rho}{2}\int_0^T \varphi_h \, \lVert \partial_t y^{(h)} \rVert^2 \; \d s \geq \frac{\rho}{2}\int_0^T \varphi \, \lVert \partial_t y \rVert^2 \; \d s.
		$$
		Finally, \eqref{j3} gives the convergence of the load term for every $t \in [0,T]$, with a bound independent of $h$ and $t$, so that the dominated convergence theorem applies to it. By the arbitrariness of $\varphi$, we obtain \eqref{MmvEI} for a.e. $t \in (0,T)$.
	
	Since $\mathfrak{E}_0 \geq E(y_0,b_0)$ by Subsection \ref{Membrane}, the triple $(y,b,\nu)$ is a measure-valued solution of \eqref{Meq1S}--\eqref{Meq11S} with initial energy $\mathfrak{E}_0$ in the sense of Definition \ref{MdefSolMV}. This concludes the proof of Theorem \ref{Theorem6}.
	
	\appendix
	
	\section{Compactness results}
	\begin{theorem}{\cite[Corollary 4, p.~85]{Simon}}
		\label{cor4Simon}
		Let $X$, $B$, and $Y$ be Banach spaces such that $X \overset{c}{\hookrightarrow} B \hookrightarrow Y$. Let $\mathcal{F}$ be bounded in $L^\infty(0,T;X)$ and let $\partial_t \mathcal{F} = \{\partial_t f : f \in \mathcal{F}\}$ be bounded in $L^r(0,T;Y)$ for some $r>1$. Then, $\mathcal{F}$ is relatively compact in $C^0([0,T];B)$.
	\end{theorem}

	\begin{theorem}{\cite[Lemma 9, p.~85]{Simon}}
		\label{lemma9Simon}
		Let $X$, $B$, and $Y$ be Banach spaces such that $X \overset{c}{\hookrightarrow} B \hookrightarrow Y$. Assume $p \in [1,\infty]$, $\mathcal{F}$ bounded in $L^p(0,T;X)$ and relatively compact in $L^p(0,T;Y)$. Then $\mathcal{F}$ is relatively compact in $L^p(0,T;B)$.
	\end{theorem}

	\section{Functional inequalities}
	
	\begin{theorem}[Generalized Korn's inequality]{\cite[Theorem 3.3, p.~18]{MielkeRoub2}}
		\label{KornMR}
		Let $D \subset \mathbb{R}^3$ be a bounded Lipschitz domain and let $\Gamma_D \subset \partial D$ have positive surface measure. For $\lambda\in(0,1)$ and $K>1$, define
		\begin{equation}
			\label{FK}
			\mathcal{F}_K := \Bigl\{ F \in C^{0,\lambda}(\bar{D};\mathbb{M}^{3\times3}) \, : \, \lVert F \rVert_{C^{0,\lambda}(\bar{D})} \leq K, \; \min_{x\in\bar{D}} \det F(x) \geq 1/K \Bigr\}.
		\end{equation}
		Then, for every $K>1$ there exists a constant $c_K>0$ such that
		$$
		\int_D \lvert F^T \nabla v + (\nabla v)^T F \rvert^2 \; \d x \geq c_K \lVert v \rVert^2_{H^1(D)}
		$$
		for all $F\in\mathcal{F}_K$ and all $v\in H^1(D;\mathbb{R}^3)$ with $v = 0$ on $\Gamma_D$.
	\end{theorem}
	The constant $c_K$ depends also on $D$, $\Gamma_D$, and $\lambda$. For each fixed $F$, the inequality follows from \cite[Theorem 2.4]{Pompe}, since the linear map $G \mapsto F^T G + G^T F$ is elliptic in the sense of \cite[Definition 2.1]{Pompe} whenever $\det F \neq 0$. The uniformity of $c_K$ on $\mathcal{F}_K$ is then obtained in \cite{MielkeRoub2} from the compactness of $\mathcal{F}_K$ in $C^0(\bar{D};\mathbb{M}^{3\times3})$.
	
	\begin{theorem}[Gagliardo-Nirenberg]{\cite[Theorem, p.~125]{Nirenberg}}
		\label{GN}
		Let $\Omega \subset \mathbb{R}^d$ be an open bounded Lipschitz domain. Let $v \in L^{q}(\Omega)$ and $D^m v \in L^r(\Omega)$ for $1 \leq q,r \leq \infty$. Then, for any $0 \leq j \leq m$ and any $\tilde{q}>0$, there exists a constant $C$ depending on $\Omega,d,m,j,q,r,\tilde{q},\lambda$ such that the following inequality holds
		$$
		\lVert D^j v \rVert_{L^p(\Omega)} \leq C \biggl( \lVert D^m v \rVert_{L^r(\Omega)}^{\lambda} \lVert v \rVert_{L^q(\Omega)}^{1-\lambda} + \lVert v \rVert_{L^{\tilde{q}}(\Omega)} \biggr),
		$$
		where
		$$
		\frac{1}{p} = \frac{j}{d} + \lambda \biggl( \frac{1}{r} - \frac{m}{d}  \biggr) + \frac{1 - \lambda}{q}, \qquad \frac{j}{m} \leq \lambda \leq 1,
		$$
		with $\lambda < 1$ if $1 < r < \infty$ and $m - j - d/r$ is a nonnegative integer.
	\end{theorem}

	\begin{theorem}[Generalized Poincar\'e inequality]{\cite[Theorem B.3.15, p.~519]{martin}}
		\label{PoincareGen}
		Let $\Omega \subset \mathbb{R}^d$ be an open bounded Lipschitz domain with boundary $\Gamma$. Let $\Gamma_D \subset \Gamma$ be such that $\text{meas}_{d-1}(\Gamma_D) > 0$. 
		
		Then, for any $1 \le p < \infty$, there exists a constant $K > 0$ such that for all functions $v \in W^{1,p}(\Omega)$, the following inequality holds
		\begin{equation*}
			\int_{\Omega} |v(x)|^p \; \d x \le K \left( \int_{\Omega} |\nabla v(x)|^p \; \d x + \left| \int_{\Gamma_D} v(x) \; \d S \right|^p \right).
		\end{equation*}
	\end{theorem}

	\begin{theorem}[Lower bound on the determinant]{\cite[Lemma 4.1, p.~868]{healeyKromer}}
		\label{thHealeyKromer}
		Let $\Omega \subset \mathbb{R}^{d}$ be a bounded Lipschitz domain, $J \in C^{0,\gamma}(\bar{\Omega})$, and $\gamma \in (0,1)$. Suppose that $J(x) > 0$ almost everywhere in $\Omega$ and that there exist two positive constants $L,M$ such that
		\begin{equation*}
			\int_{\Omega} (J(z))^{-s} \; {\rm d}z \leq L  , \qquad \lVert J \rVert_{C^{0,\gamma}(\bar{\Omega})} \leq M,
		\end{equation*} 
		for some $s\geq d/\gamma$. Then, there exists a constant $\widetilde{c} = \widetilde{c}(\Omega,M,L,s,d)$ such that
		\begin{equation*}
			J(x) \geq \widetilde{c}, \qquad \forall x\in \bar{\Omega}.
		\end{equation*}
	\end{theorem}
	
\section{Two elementary lemmas}

\begin{lemma}[Extension by density]
	\label{LemC1}
	Let $I \subset \R$ be a bounded interval, let $Z$ be a Hilbert space, and let $X$ be a Banach space continuously and densely embedded in $Z$. Let $\xi \in L^2(I;X')$ and $\zeta \in L^2(I;Z')$ be such that
	$$
	\int_I \bigl( \langle \xi , \phi \rangle_X + \langle \zeta , \phi \rangle_Z \bigr) \; \d t = 0 \qquad \text{for every } \phi \in L^2(I;X).
	$$
	Then, $\xi$ extends to an element of $L^2(I;Z')$, namely, $\xi = -\zeta$, and the identity above holds for every $\phi \in L^2(I;Z)$.
\end{lemma}
\begin{proof}
	For every $\phi \in L^2(I;X)$, one has $\bigl\lvert \int_I \langle \xi , \phi \rangle_X \, \d t \bigr\rvert \leq \lVert \zeta \rVert_{L^2(I;Z')} \lVert \phi \rVert_{L^2(I;Z)}$. As $L^2(I;X)$ is dense in $L^2(I;Z)$, the functional $\phi \mapsto \int_I \langle \xi , \phi \rangle_X \, \d t$ extends uniquely by continuity to $L^2(I;Z)$, and the extension coincides with $\phi \mapsto -\int_I \langle \zeta , \phi \rangle_Z \, \d t$.
\end{proof}

\begin{lemma}[Minty's trick]
	\label{LemMinty}
	Let $Y$ be a Banach space and let $\mathcal{N} : Y \to Y'$ be hemicontinuous, namely, $s \mapsto \langle \mathcal{N}(u + s w) , z \rangle_Y$ is continuous on $[-1,1]$ for all $u,w,z \in Y$. If $u \in Y$ and $\mathcal{P} \in Y'$ satisfy $\langle \mathcal{P} - \mathcal{N}(v) , u - v \rangle_Y \geq 0$ for every $v \in Y$, then $\mathcal{P} = \mathcal{N}(u)$.
\end{lemma}
\begin{proof}
	For $w \in Y$ and $s \in (0,1]$, the choice $v = u - sw$ gives $\langle \mathcal{P} - \mathcal{N}(u - sw) , w \rangle_Y \geq 0$. Letting $s \to 0$, we obtain $\langle \mathcal{P} - \mathcal{N}(u) , w \rangle_Y \geq 0$ for every $w \in Y$, and replacing $w$ with $-w$ gives the claim.
\end{proof}
	
		\section*{Acknowledgment}
        This work was supported by the CSF under the project 23-04766S and  by the M\v{S}MT \v{C}R under the project   Roboprox [reg. no. CZ.02.01.01/00/22 008/0004590], within the Operational Programme Johannes Amos Comenius co-funded by the European Union.

	\end{document}